\documentclass{amsart}
\usepackage{amsmath,amsthm,amssymb}
\usepackage{enumerate}
\usepackage{graphicx}
\usepackage{cite}
\usepackage{comment}
\usepackage{oands}
\usepackage{tikz}
\usepackage{changepage}
\usepackage{bbm}
\usepackage{mathtools}
\usepackage[margin=1.2in]{geometry}
\usepackage{hyperref}
\usepackage{appendix}
\usepackage[pagewise,mathlines]{lineno}
\usepackage{multicol}

\theoremstyle{plain}
\newtheorem{thm}{Theorem}[section]
\newtheorem{cor}[thm]{Corollary}
\newtheorem{lem}[thm]{Lemma}
\newtheorem{prop}[thm]{Proposition}

\newtheorem{notation}[thm]{Notation}

\numberwithin{equation}{section}
 
\theoremstyle{definition}
\newtheorem{defn}[thm]{Definition}
\newtheorem{remark}[thm]{Remark}

\def\@rst #1 #2other{#1}

\hypersetup{
    backref,
    colorlinks=false,
    linktocpage,
    pdfborder = {0 0 1 [1 0] }
    }

\newcommand{\dsb}{\begin{adjustwidth}{2.5em}{0pt}
\begin{footnotesize}}
\newcommand{\dse}{\end{footnotesize}
\end{adjustwidth}}

\newcommand{\ssb}{\begin{adjustwidth}{2.5em}{0pt}}
\newcommand{\sse}{\end{adjustwidth}}

\newcommand{\aryb}{\begin{eqnarray*}}
\newcommand{\arye}{\end{eqnarray*}}
\def\alb#1\ale{\begin{align*}#1\end{align*}}
\newcommand{\eqb}{\begin{equation}}
\newcommand{\eqe}{\end{equation}}
\newcommand{\eqbn}{\begin{equation*}}
\newcommand{\eqen}{\end{equation*}}

\newcommand{\BB}{\mathbbm}
\newcommand{\ol}{\overline}
\newcommand{\ul}{\underline}
\newcommand{\op}{\operatorname}

\newcommand{\frk}{\mathfrak}
\newcommand{\eqD}{\overset{d}{=}}
\newcommand{\ep}{\epsilon}
\newcommand{\rta}{\rightarrow}

\newcommand{\wt}{\widetilde}
\newcommand{\wh}{\widehat} 
\newcommand{\mcl}{\mathcal}

\newcommand{\bdy}{\partial}

\newcommand*\tc[1]{\tikz[baseline=(char.base)]{\node[shape=circle,draw,inner sep=1pt] (char) {#1};}}
\newcommand*\tb[1]{\tikz[baseline=(char.base)]{\node[shape=rectangle,draw,inner sep=2.5pt] (char) {#1};}}

\let\originalleft\left
\let\originalright\right
\renewcommand{\left}{\mathopen{}\mathclose\bgroup\originalleft}
\renewcommand{\right}{\aftergroup\egroup\originalright}

\newcommand*\patchAmsMathEnvironmentForLineno[1]{  \expandafter\let\csname old#1\expandafter\endcsname\csname #1\endcsname
  \expandafter\let\csname oldend#1\expandafter\endcsname\csname end#1\endcsname
  \renewenvironment{#1}     {\linenomath\csname old#1\endcsname}     {\csname oldend#1\endcsname\endlinenomath}}\newcommand*\patchBothAmsMathEnvironmentsForLineno[1]{  \patchAmsMathEnvironmentForLineno{#1}  \patchAmsMathEnvironmentForLineno{#1*}}\AtBeginDocument{\patchBothAmsMathEnvironmentsForLineno{equation}\patchBothAmsMathEnvironmentsForLineno{align}\patchBothAmsMathEnvironmentsForLineno{flalign}\patchBothAmsMathEnvironmentsForLineno{alignat}\patchBothAmsMathEnvironmentsForLineno{gather}\patchBothAmsMathEnvironmentsForLineno{multline}}

\title[Scaling limits for the FK model I]{Scaling limits for the critical Fortuin-Kasteleyn model on a random planar map I: cone times} 
 
\author{Ewain Gwynne}

\author{Cheng Mao}

\author{Xin Sun}


\address{Department of Mathematics\\
  Massachusetts Institute of Technology\\
  Cambridge, MA 02139}
\email{ewain@mit.edu \\ maocheng@mit.edu \\ xinsun89@math.mit.edu}

\begin{document}

\begin{abstract}
Sheffield (2011) introduced an inventory accumulation model which encodes a random planar map decorated by a collection of loops sampled from the critical Fortuin-Kasteleyn (FK) model. He showed that a certain two-dimensional random walk associated with an infinite-volume version of the model converges in the scaling limit to a correlated planar Brownian motion. We improve on this scaling limit result by showing that the times corresponding to FK loops (or ``flexible orders") in the inventory accumulation model converge in the scaling limit to the $\pi/2$-cone times of the correlated Brownian motion. This statement implies a scaling limit result for the joint law of the areas and boundary lengths of the bounded complementary connected components of the FK loops on the infinite-volume planar map. In light of the encoding of Duplantier, Miller, and Sheffield (2014), the limiting object coincides with the joint law of the areas and boundary lengths of the bounded complementary connected components of a collection of CLE$_\kappa$ loops on an independent Liouville quantum gravity cone.
\end{abstract}
 
\maketitle

\tableofcontents

\section{Introduction} 
\label{sec-intro}

\subsection{Overview}
\label{sec-overview}

A \emph{(critical) Fortuin-Kasteleyn (FK) planar map of size $n\in\BB N$ and parameter $q > 0$} is a pair $(M,S)$ consisting of a planar map $M$ with $n$ edges and a subset $S$ of the set of edges of $M$, sampled with weight $q^{K(S)/2}$ where $K(S)$ is the number of connected components of $S$ plus the number of complementary connected components of $S$. 
This model is critical in the sense that its partition function has power law decay as $n\rta\infty$ (this is established in the sequel~\cite{gms-burger-local} to the present paper). 
If $(M,S)$ is a critical FK planar map of size $n$ and parameter $q$, then the conditional law of $S$ given $M$ is that of the uniform measure on edge sets of $M$ weighted by $q^{K(S)/2}$. This law is a special case of the FK cluster model on $M$~\cite{fk-cluster}. The FK model is closely related to the critical $q$-state Potts model~\cite{baxter-potts} for general integer values of $q$; to critical percolation for $q=1$; and to the Ising model for $q = 2$. See e.g.\ \cite{kager-nienhuis-guide, grimmett-fk} for more on the FK model and its relationship to other statistical physics models. 
  
The edge set $S$ on $M$ gives rise to a dual edge set $S^*$, consisting of those edges of the dual map $M^*$ which do not cross edges of $S$; and a collection $\mcl L$ of loops on $M$ which form the interfaces between edges of $S$ and $S^*$. Note that $\#\mcl L = K(S)$. 
The collection of loops $\mcl L$ determines the same information as $S$, so one can equivalently view a critical FK planar map as a random planar map decorated by a collection of loops. 

The critical FK planar map is conjectured to converge in the scaling limit to a \emph{conformal loop ensemble ($\op{CLE}_\kappa$)} with $\kappa \in (4,8)$ satisfying $q = 2 + 2\cos(8\pi/\kappa)$ on top of an independent \emph{Liouville quantum gravity (LQG) surface} with parameter $\gamma = 4/\sqrt{\kappa}$. 
See~\cite{kager-nienhuis-guide,shef-burger} and the references therein for more details regarding this conjecture. 
We will not make explicit use of CLE or LQG in this paper, but we briefly recall their definitions (with references) for the interested reader. 
A $\op{CLE}_\kappa$ is a countable collection of random fractal loops which locally look like Schramm's SLE$_\kappa$ curves~\cite{schramm0,schramm-sle}, which was first introduced in~\cite{shef-cle}. Many of the basic properties of $\op{CLE}_{\kappa}$ for $\kappa \in (4,8)$ are proven in~\cite{ig1,ig2,ig3,ig4} by encoding CLE$_{\kappa}$ by means of a space-filling variant of SLE$_\kappa$ which traces all of the loops. For $\gamma \in (0,2)$, a $\gamma$-LQG surface is, heuristically speaking, the random surface parametrized by a domain $D\subset \BB C$ whose Riemannian metric tensor is $e^{\gamma h} \, dx\otimes dy$, where $h$ is some variant of the Gaussian free field (GFF) on $D$ and $dx\otimes dy$ is the Euclidean metric tensor. This object is not defined rigorously since $h$ is a distribution, not a function. However, one can make rigorous sense of an LQG surface as a random measure space (equipped with the volume form induced by $e^{\gamma h} \, dx\otimes dy$), as is done in~\cite{shef-kpz}. See also~\cite{shef-zipper,wedges,sphere-constructions} for more on this interpretation of LQG surfaces. 

In \cite{shef-burger}, Sheffield introduces a simple inventory accumulation model described by a word $X$ consisting of five different symbols which represent two types of ``burgers" and three types of ``orders"; and constructs a bijection between certain realizations of this model and triples $(M , e_0 , S)$ consisting of a planar map with $n$ edges, an oriented root edge $e_0$, and a set $S$ of edges of $M$. This bijection generalizes a bijection due to Mullin~\cite{mullin-maps} (which is explained in more detail in~\cite{bernardi-maps}) and is equivalent to the construction of~\cite[Section 4]{bernardi-sandpile} if one treats the planar map $M$ as fixed, although the latter is phrased in a different way (see~\cite[Footnote 1]{shef-burger} for an explanation of this equivalence).  

There is a family of probability measures for the inventory accumulation model, indexed by a parameter $p \in (0,1/2)$, with the property that the law of the triple $(M ,e_0, S)$ when the inventory accumulation model is sampled according to the probability measure with parameter $p$ is given by the uniform measure on such triples weighted by $q^{K(S)/2}$, where $q = 4p^2/(1-p)^2$. That is, the law of $(M,e_0,S)$ is that of an FK planar map with a uniformly chosen oriented root edge. As alluded to in \cite[Section 4.2]{shef-burger} and explained in more detail in~\cite{blr-exponents,chen-fk}, there is also an infinite-volume version of the bijection of~\cite{shef-burger} which encodes an infinite-volume limit (in the sense of~\cite{benjamini-schramm-topology}) of finite-volume FK planar maps, which we henceforth refer to as an infinite-volume FK planar map. 

The inventory accumulation model of~\cite{shef-burger} is equivalent to a model on non-Markovian random walks on $\BB Z^2$ with certain marked steps. 
In \cite[Theorem 2.5]{shef-burger}, it is shown that a random walk which describes the infinite-volume version of the inventory accumulation model converges in the scaling limit to a pair of Brownian motions with correlation depending on $p$. 

In \cite{wedges} (see also \cite{ig4}), it is shown that for $\kappa \in (4,8)$, a whole-plane $\op{CLE}_\kappa$ on top of an independent $4/\sqrt\kappa$-LQG cone (a type of infinite-volume quantum surface) can be encoded by a pair of correlated Brownian motions via a procedure which is directly analogous to the bijection of \cite{shef-burger}. This procedure is called the \emph{peanosphere} (or \emph{mating of trees}) construction. The correlation between this pair of Brownian motions is the same as the correlation between the pair of limiting Brownian motions in \cite[Theorem 2.5]{shef-burger} provided
\eqb \label{eqn-p-kappa}
p = \frac{\sqrt{2 + 2\cos (8\pi/\kappa)}}{2 + \sqrt{2 + 2\cos (8\pi/\kappa)}}  , 
\eqe  
which is consistent with the conjectured relationship between the FK model and CLE described above. Thus \cite[Theorem 2.5]{shef-burger} can be viewed as a scaling limit result for FK planar maps toward $\op{CLE}_\kappa$ on a quantum cone in a certain topology, namely the one in which two loop-decorated surfaces are close if their corresponding encoding functions are close. However, this topology does not encode all of the information about the FK planar map. Indeed, the non-Markovian walk on $\BB Z^2$ does not encode the FK loops themselves but rather a pair of trees constructed from the loops.

In this paper, we will improve on the scaling limit result of \cite{shef-burger} by showing that the times corresponding to FK loops (or ``flexible orders") in the infinite-volume inventory accumulation model converge in the scaling limit to the $\pi/2$-cone times of the correlated Brownian motion (see Theorem~\ref{thm-cone-limit} below for a precise statement). We thus obtain a scaling limit in a topology which encodes all of the information about the FK planar map. The $\pi/2$-cone times of the correlated Brownian motion in the setting of \cite{wedges} encode the $\op{CLE}_\kappa$ loops in a manner which is directly analogous to the encoding of the FK loops in Sheffield's bijection. Hence our results imply the convergence of many interesting functionals of the FK loops to the corresponding functionals of $\op{CLE}_\kappa$ loops on an independent quantum cone. As a particular application, we will obtain the joint scaling limit of the boundary lengths and areas of all of the macroscopic bounded complementary connected components of the FK loops surrounding a fixed edge in an infinite-volume FK planar map (see Theorem~\ref{thm-loop-limit} below). This statement partially answers \cite[Question 13.3]{wedges} in the infinite-volume setting. 
  
\begin{figure}[ht!]
 \begin{center}
\includegraphics{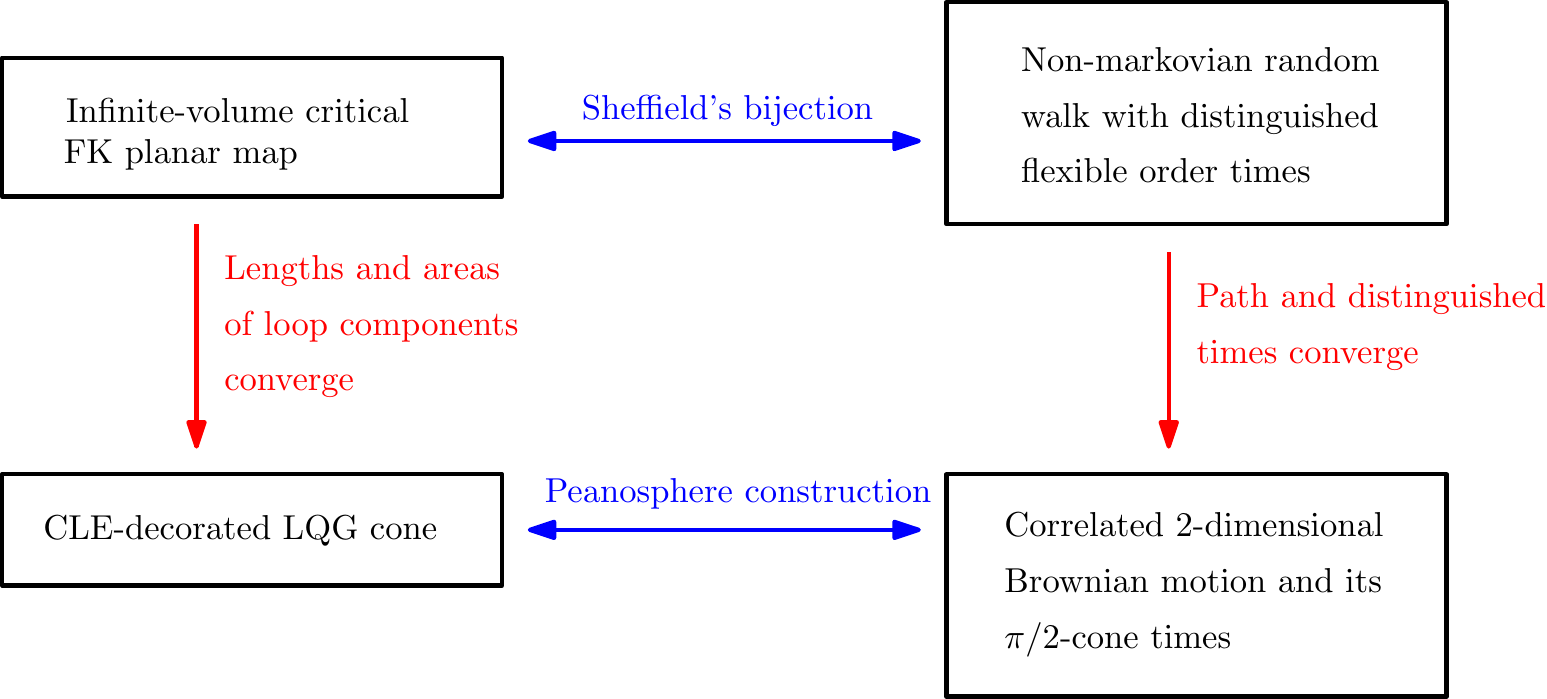} 
\caption{A schematic illustration of the bijections and scaling limit results involved in this paper. The top blue arrow corresponds to Sheffield's~\cite{shef-burger} encoding of critical FK planar maps via the inventory accumulation model. The bottom blue arrow corresponds to the encoding of a CLE-decorated LQG cone via correlated two-dimensional Brownian motion in~\cite{wedges}. The right red arrow corresponds to the scaling limit result for the non-Markovian random walk in~\cite{shef-burger} and our Theorem~\ref{thm-cone-limit}, which gives convergence of the flexible order times in the discrete model to the $\pi/2$-cone times of the correlated Brownian motion. The left arrow corresponds to our Theorem~\ref{thm-loop-limit}, which is deduced from the right arrow and the bijections in the figure. } \label{fig-models}
\end{center}
\end{figure} 
 
In the course of proving our main results, we will also prove several other results regarding the model of~\cite{shef-burger} which are of independent interest. We prove tail estimates for the laws of various quantities associated with this model, and in particular show that several such laws have regularly varying tails (see Sections~\ref{sec-reg-var} and~\ref{sec-no-order-reg-var}). We also obtain the scaling limit of the discrete path conditioned on the event that the reduced word contains no burgers, or equivalently the event that this path stays in the first quadrant until a certain time when run backward (Theorem~\ref{thm-no-burger-conv}) and the analogous statement when we instead condition on no orders (Theorem~\ref{thm-no-order-conv}). Scaling limit results for random walks with independent increments conditioned to stay in a cone are obtained in several places in the literature (see \cite{shimura-cone-walk,garbit-cone-walk,dw-cones} and the references therein). Our Theorems~\ref{thm-no-burger-conv} and~\ref{thm-no-order-conv} are analogues of these results for a certain random walk with non-independent increments. 

Although this paper is motivated by the relationship between the inventory accumulation model of~\cite{shef-burger}, FK planar maps, and $\op{CLE}_\kappa$ on a Liouville quantum gravity surface, our proofs use only basic properties of the inventory accumulation model and elementary facts from probability theory, so can be read without any background on SLE or LQG. 

This paper strengthens the topology of the scaling limit result of \cite[Theorem 2.5]{shef-burger}. Ideally, one would like to further strengthen this topology by embedding an FK planar map into the Riemann sphere and showing that the conformal structure of the loops converges in an appropriate sense to that of CLE loops on an independent quantum cone. We expect that proving this convergence is a substantially more difficult problem than proving the convergence statements of this paper. However, our result might serve as an intermediate step in proving such a stronger convergence statement.  See~\cite[Section 10.5]{wedges} for some (largely speculative) ideas regarding the relationship between convergence of the conformal structure of FK loops and the convergence statements proven in \cite{shef-burger} and the present paper. 

Stronger scaling limit results are known in the case of a uniformly chosen random planar map (which corresponds to the special case $p=1/3$ in the framework of~\cite{shef-burger}), without a collection of loops.
In particular, it is proven in~\cite{legall-uniqueness,miermont-brownian-map} that a uniformly chosen random quadrangulation with $2n$ edges converges in law in the Gromov-Hausdorff topology to a continuum random metric space called the \emph{Brownian map} (see also~\cite{bjm-uniform} for a proof of this result for a uniform planar map). 
This and similar scaling limit results are proven using a bijective encoding of planar quadrangulations in terms of labelled trees due to Schaefer~\cite{schaeffer-bijection}. Note that the bijection of~\cite{schaeffer-bijection} differs significantly from the bijection of~\cite{shef-burger}, in that the former encodes only a planar map (not a planar map decorated by a collection of edges) and more explicitly describes distances in the map. We refer the reader to the survey articles~\cite{miermont-survey,legall-sphere-survey} and the references therein for more details on uniform random planar maps and their scaling limits. It is shown in~\cite{qle,sphere-constructions,tbm-characterization,lqg-tbm1,lqg-tbm2,lqg-tbm3} that a $\sqrt{8/3}$-LQG cone can be equipped with a metric under which it is isometric to the Brownian plane~\cite{curien-legall-plane}. Hence the above scaling limit results can also be phrased in terms of LQG.
 
We end this subsection by pointing out some related works. 
This paper is the first of a series of three papers; the other two are~\cite{gms-burger-local,gms-burger-finite}. In~\cite{gms-burger-local}, the authors prove estimates for the probability that a reduced word in the inventory accumulation model of~\cite{shef-burger} contains a \emph{particular} number of symbols of a certain type, prove a related scaling limit result, and compute the exponent for the probability that a word sampled from this model reduces to the empty word. The work~\cite{gms-burger-finite} proves analogues of the scaling limit results of \cite{shef-burger} and of the present paper for the finite-volume version of the model of~\cite{shef-burger} (which encodes a finite-volume FK planar map). 

Shortly before this paper was first posted to the ArXiv, we learned of an independent work~\cite{blr-exponents} which calculates tail exponents for several quantities related to a generic loop on an FK planar map, and which was posted to the ArXiv at the same time as this work. In~\cite{sun-wilson-unicycle}, the third author and D. B. Wilson study unicycle-decorated random planar maps via the bijection of~\cite{shef-burger} and obtain the joint distribution of the length and area of the unicycle in the infinite volume limit. The work~\cite{chen-fk} studies some properties of the infinite-volume FK planar map at the discrete level. The recent work~\cite{gkmw-burger} uses a generalized version of Sheffield's inventory accumulation model to prove a scaling limit result analogous to that of~\cite{shef-burger} for a class of random planar map models which correspond to SLE$_\kappa$-decorated $\gamma$-Liouville quantum gravity surfaces for $\kappa >8$ and $\gamma =4/\sqrt\kappa < \sqrt 2$. 
 
The first author and J. Miller are currently preparing two papers which apply the results of the present paper and its sequels. The paper~\cite{gwynne-miller-cle} will use the scaling limit results of~\cite{shef-burger,gms-burger-local,gms-burger-finite} and the present paper to prove a scaling limit result which can be interpreted as the statement that FK planar maps converge to $\op{CLE}_{\kappa}$ on a Liouville quantum surface viewed modulo an ambient homeomorphism of $\BB C$. The paper~\cite{gwynne-miller-inversion} will use said scaling limit result to prove conformal invariance of whole-plane $\op{CLE}_{\kappa}$ for $\kappa \in (4,8)$ (see~\cite{werner-sphere-cle} for a proof of this statement in the case $\kappa \in (8/3,4]$).

\bigskip

\noindent{\bf Acknowledgments}
We thank Ga\"etan Borot, Jason Miller, and Scott Sheffield for helpful discussions, and Jason Miller for comments on an earlier version of this paper. We thank Nathana\"el Berestycki, Beno\^it Laslier, and Gourab Ray for sharing and discussing their work~\cite{blr-exponents} with us. We thank several anonymous referees for comments on an earlier version of this paper. We thank the Isaac Newton Institute for Mathematical Sciences, Cambridge, for support and hospitality during the Random Geometry programme, where part of this work was completed. The first author was supported by the U.S. Department of Defense via an NDSEG fellowship. The third author was partially supported by NSF grant DMS-1209044. 
 
\subsection{Inventory accumulation model}
\label{sec-burger-prelim}

The main focus of this paper will be the inventory accumulation model first introduced by Sheffield~\cite{shef-burger}, which we describe in this section. The notation introduced in this section will remain fixed throughout the remainder of the paper. 

Let $\Theta$ be the collection of symbols $\{\tc{H} , \tc{C} , \tb{H} , \tb{C} , \tb{F}\}$. We can think of these symbols as representing, respectively, a hamburger, a cheeseburger, a hamburger order, a cheeseburger order, and a flexible order. We view $\Theta$ as the generating set of a semigroup, which consists of the set of all finite words consisting of elements of $\Theta$, modulo the relations
\eqb \label{eqn-theta-ful}
\tc C \tb C = \tc H \tb H = \tc C \tb F = \tc H \tb F = \emptyset \qquad \text{(order fulfilment)}
\eqe
and
\eqb\label{eqn-theta-com}
\tc C \tb H = \tb H \tc C ,\qquad \tc H \tb C = \tb C \tc H \qquad \text{(commutativity).}
\eqe 
Given a word $x$ consisting of elements of $\Theta$, we denote by $\mathcal R(x)$ the word reduced modulo the above relations, with all burgers to the right of all orders. For example,
\eqbn
\mcl R\left( \tc H \tc C \tb H \tb F \tc H  \tb C  \right)  = \tb C \tc H .
\eqen
 In the burger interpretation, $\mcl R(x)$ represents the burgers which remain after all orders have been fulfilled along with the unfulfilled orders. We also write $|x|$ for the number of symbols in $x$ (regardless of whether or not $x$ is reduced). 

For $p\in [0,1]$ (in this paper we will in fact typically take $p\in (0,1/2)$, for reasons which will become apparent just below), we define a probability measure on $\Theta$ by 
\eqb \label{eqn-theta-prob}
\BB P\left(\tc{H}\right) = \BB P\left(     \tc{C} \right) = \frac14 ,\quad  \BB P\left( \tb{H} \right) = \BB P\left( \tb{C} \right) = \frac{1-p}{4} ,\quad \BB P\left(\tb{F} \right) =  \frac{p}{2} .
\eqe 
Let $X = \dots X_{-1} X_0 X_1 \dots$ be an infinite word with each symbol sampled independently according to the probabilities~\eqref{eqn-theta-prob}. For $a \leq b\in \BB R$, we set 
\eqb \label{eqn-X(a,b)}
X(a,b) := \mcl R\left(X_{\lfloor a\rfloor} \dots X_{\lfloor b\rfloor} \right) .
\eqe

\begin{remark}
There is an explicit bijection between words $x$ consisting of elements of $\Theta$ with $|x| = 2n$ and $\mcl R(x) = \emptyset$; and triples $(M,e_0,S)$, where $M$ is a planar map with $n$ edges, $e_0$ is an oriented root edge, and $S$ is a set of edges of $M$ \cite[Section~4.1]{shef-burger}. 
If $\dot X$ is a random word sampled according to the law of $X_1\dots X_{2n}$ (as above) with $p\in (0,1/2)$, conditioned on the event that $X(1,2n) =\emptyset$, then the law of the corresponding triple $(M,e_0, S)$ is that of a rooted FK planar map, as defined in Section~\ref{sec-overview}, with parameter $q = \frac{4p^2}{(1-p)^2}$. 

As alluded to in \cite[Section~4.2]{shef-burger} and explained more explicitly in~\cite{blr-exponents,chen-fk}, the unconditioned word $X$ corresponds to an infinite-volume limit of FK planar maps decorated by FK loops via an infinite-volume version of Sheffield's bijection. In this paper we focus on the infinite-volume case, and we will review the bijection in this case in Section~\ref{sec-burger-bijection}.   
\end{remark}
 
By \cite[Proposition 2.2]{shef-burger}, it is a.s.\ the case that each symbol $X_i$ in the word $X$ has a unique match which cancels it out in the reduced word (i.e.\ burgers are matched to orders and orders matched to burgers). Heuristically, the reduced word $X(-\infty,\infty)$ is a.s.\ empty. 

\begin{notation}\label{def-match-function}
For $i \in \BB Z$ we write $\phi(i)$ for the index of the match of $X_i$. 
We also write $\phi_*(i)$ for the index of the match of the rightmost order in $X(\phi(i) , i)$, or $\phi_*(i) = \phi(i)$ if $X(\phi(i) , i)$ contains no orders. 
\end{notation}

\begin{notation} \label{def-theta-count}
For $\theta\in \Theta$ and a word $x$ consisting of elements of $\Theta$, we write $\mcl N_{\theta}(x)$ for the number of $\theta$-symbols in $x$. We also let
\alb 
&d(x)  := \mcl N_{\tc H}(x) - \mcl N_{\tb H}(x ) ,\quad
d^*(x)  := \mcl N_{\tc C}(x) - \mcl N_{\tb C}(x)  ,\quad 
 D(x)  := \left(d(W) , d^*(x)\right) .
\ale
\end{notation}

The reason for the notation $d$ and $d^*$ is that these functions (applied to segments of the word $Y$ defined just below) give the distances to the root edge in the tree and dual tree obtained from the primal and dual edge sets in Sheffield's bijection; see Section~\ref{sec-burger-bijection}. 

For $i\in\BB Z$, we define $Y_i = X_i$ if $X_i \in \{\tc{H} , \tc{C} , \tb{H} ,  \tb{C}\}$; $Y_i = \tb{H}$ if $X_i = \tb F$ and $X_{\phi(i)} = \tc{H}$; and $Y_i = \tb{C}$ if $X_i = \tb F$ and $X_{\phi(i)} = \tc{C}$. For $a\leq b \in \BB R$, define $Y(a,b)$ as in~\eqref{eqn-X(a,b)} with $Y$ in place of $X$.  
 
Let $d(0) = 0$. For $n \in \BB N$, define $d(n) = d(Y(1,n))$ and $d(-n) = -d(Y( -n+1 , 0))$. Define $d^*(n)$ for $n\in\BB Z$ similarly and extend each of these functions from $\BB Z$ to $\BB R$ by linear interpolation. 

\begin{remark} \label{remark-minus-sign}
Note that we have inserted a minus sign in the definition of $d(n)$ and $d^*(n)$ when $n < 0$. This is done so that $d(\cdot) \eqD d(\cdot + n) - d(n)$ for each $n\in\BB Z$ and similarly for $d^*$. 
\end{remark}

For $t\in\BB R$, let
\eqb \label{eqn-discrete-paths}
  D(t) := (d(t) , d^*(t)) .
\eqe  
For $n\in\BB N$ and $t\in \BB R$, let 
\eqb \label{eqn-Z^n-def}
U^n(t) := n^{-1/2} d (n t) ,\quad V^n(t) := n^{-1/2} d^*(n  t) , \quad Z^n(t) := (U^n(t) , V^n(t)) .
\eqe 
For $p \in [0,1/2)$, we also let $Z= (U,V)$ be a two-sided two-dimensional Brownian motion with $Z(0) = 0$ and variances and covariances at each time $t\in \BB R$ given by 
\eqb \label{eqn-bm-cov}
\op{Var}(U(t) ) = \frac{1-p}{2} |t| \quad \op{Var}(V(t)) = \frac{1-p}{2} |t| \quad \op{Cov}(U(t) , V(t) ) = \frac{p}{2} |t| .
\eqe
It is shown in \cite[Theorem 2.5]{shef-burger} that as $n\rta \infty$, the random paths $U^n + V^n$ and $U^n - V^n$ converge in law in the topology of uniform convergence on compacts to a pair of independent Brownian motions, with respective variances 1 and $(1-2p)\vee 0$. The following result is an immediate consequence. 

\begin{thm}[Sheffield] \label{prop-burger-limit}
For $p \in (0,1/2)$, the random paths $Z^n$ defined in~\eqref{eqn-Z^n-def} converge in law in the topology of uniform convergence on compacts to the random path $Z$ of~\eqref{eqn-bm-cov}. 
\end{thm}

Throughout the remainder of this paper, we fix $p \in (0,1/2)$ and do not make dependence on $p$ explicit.  
 
\subsection{Cone times}
\label{sec-cone-times}
 
The first main result of this paper is Theorem~\ref{thm-cone-limit} below, which says that the times for which $X_i = \tb F$ converge under a suitable scaling limit to the \emph{$\pi/2$-cone times} of $Z$, defined as follows.

\begin{defn}\label{def-cone-time}
A time $t$ is called a \emph{(weak) $\pi/2$-cone time} for a function $Z = (U,V) : \BB R \rta \BB R^2$ if there exists $t'< t$ such that $U(s) \geq U(t)$ and $V(s) \geq V(t)$ for $s\in [t'   , t ]$. Equivalently, $Z([t'   , t ])$ is contained in the cone $Z_{t } + \{z\in \BB C : \op{arg} z \in [0,\pi/2]\}$. We write $  v_Z(t)$ for the infimum of the times $t'$ for which this condition is satisfied, i.e. $  v_Z(t)$ is the last entrance time of the cone before $t$. We say that $t$ is a \emph{left (resp. right) $\pi/2$-cone time} if $V_t = V(v_Z(t))$ (resp. $U(t) = U(v_Z(t))$). Two $\pi/2$-cone times for $Z$ are said to be in the \emph{same direction} if they are both left or both right $\pi/2$-cone times, and in the \emph{opposite direction} otherwise. For a $\pi/2$-cone time $t$, we write $ u_Z(t)$ for the supremum of the times $t^* < t$ such that
\eqbn
\inf_{s\in [t^* , t]} U(s)  < U(t ) \quad \op{and} \quad \inf_{s\in [t^* , t]} V(s)  < V(t)  .
\eqen
That is, $ u_Z(t)$ is the last time before $t$ that $Z$ crosses the boundary line of the cone which it does not cross at time $ v_Z(t)$.  
\end{defn} 

\begin{figure}[ht!]
 \begin{center}
\includegraphics{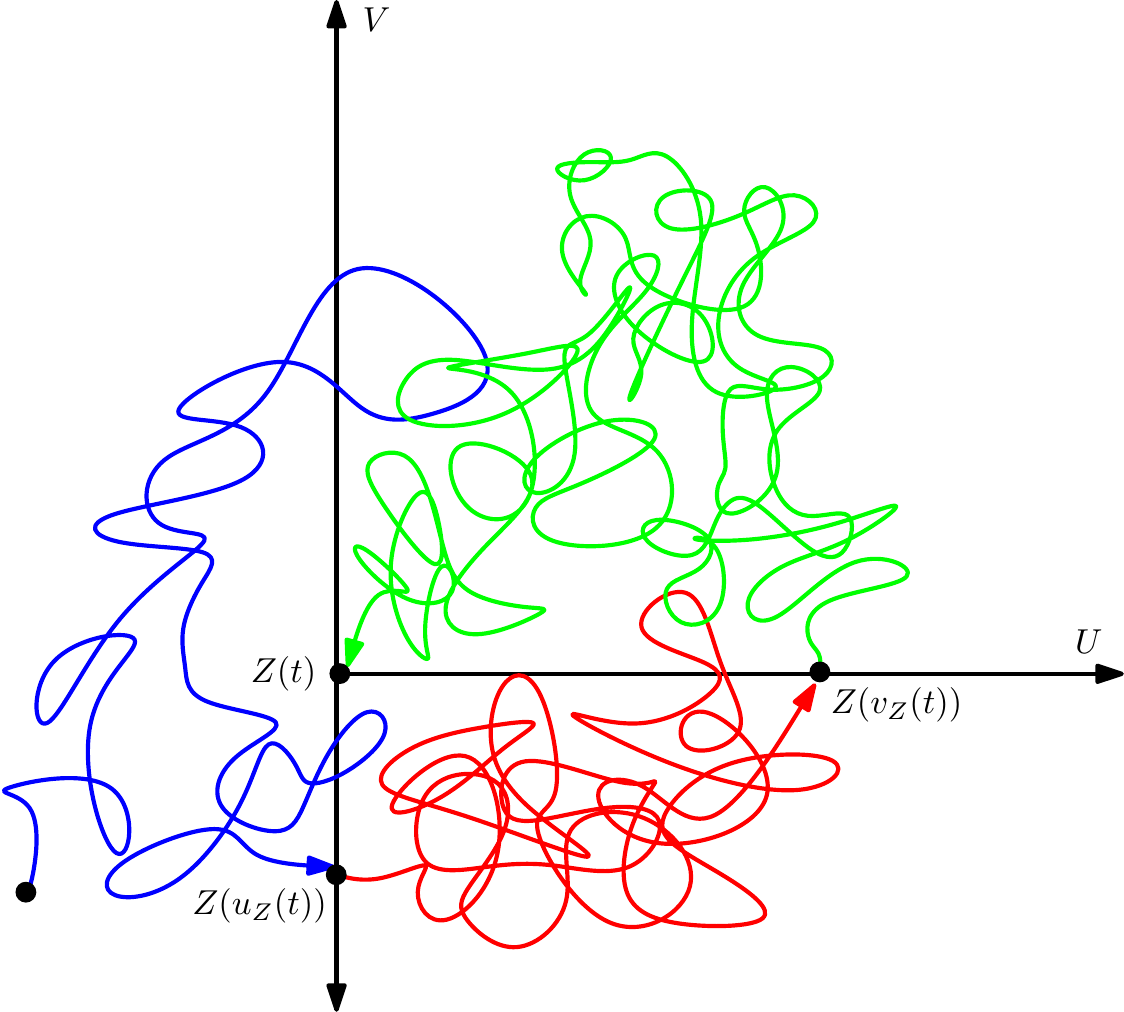} 
\caption{An illustration of a left $\pi/2$-cone time $t$ for a path $Z = (U,V)$. The set $Z([  u_Z(t) , v_Z (t)])$ is shown in red. The set $Z([  v_Z(t) , t])$ is shown in green. We note that we may have $V(u_Z(t)) < V(t)$ (as shown in the figure) or $V(u_Z(t)) \geq V(t)$. } \label{fig-cone-time}
\end{center}
\end{figure}

See Figure~\ref{fig-cone-time} for an illustration of Definition~\ref{def-cone-time}. The reader may easily check that if $i \in \BB Z$ is such that $X_i = \tb F$ and $i-\phi(i) \geq 2$, then $i/n$ and $(i-1)/n$ are both (weak) $\pi/2$-cone times for $Z^n$. Using Notation~\ref{def-match-function}, $v_{Z^n}((i-1)/n) = \phi(i)/n$ and $u_{Z^n}((i-1)/n)$ is equal to $n^{-1}$ times the largest $j < i$ for which $X(j , i)$ contains a burger of the type opposite $X_{\phi(i)}$. 
Equivalently, $u_{Z^n}((i-1)/n)$ is $n^{-1}$ times the largest $j < \phi_*(i)$ for which $X(j , \phi_*(i))$ contains a burger of the type opposite $X_{\phi(i)}$.
If $|X(\phi(i) , i) | \geq 1$, the direction of these $\pi/2$-cone times are determined by what type of burger $X_{\phi(i)}$ is.

A positively correlated Brownian motion a.s.\ has an uncountable fractal set of $\pi/2$-cone times~\cite{shimura-cone, evans-cone}. There is a substantial literature concerning cone times of Brownian motion; we refer the reader to~\cite[Sections 3 and 4]{legall-bm-notes}, \cite[Section 10.4]{peres-bm}, and the references therein for more on this topic. 
 
Our first main result states that the $\tb F$-times for $X$, re-scaled by $n^{-1}$, converge to the $\pi/2$-cone times of $Z$. One needs to be careful about the precise sense in which this convergence occurs. Indeed, there are uncountably many $\pi/2$-cone times for $Z$, but only countably many times for which $X_i = \tb F$. To get around this issue, we prove convergence of several large but countable sets of distinguished $\pi/2$-cone times which are dense enough to approximate most interesting functionals of the set of $\pi/2$-cone times for $Z$. One such set is defined as follows. 

\begin{defn}\label{def-maximal}
A $\pi/2$-cone time for a path $Z$ is called a \emph{maximal $\pi/2$-cone time} in an (open or closed) interval $I\subset \BB R$ if $[ v_Z(t) , t] \subset I$ and there is no $\pi/2$-cone time $t'$ for $Z$ such that $[v_Z(t') , t']\subset I$ and $[ v_Z(t) , t] \subset (v_Z(t') , t')$. An integer $i\in \BB Z$ is called a \emph{maximal flexible order time} in an interval $I\subset \BB R$ if $X_i = \tb F$, $[\phi(i) ,  i]_{\BB Z} \subset I$, and there is no $i' \in \BB Z$ with $X_{i'} = \tb F$, $[\phi(i) ,   i]_{\BB Z} \subset ( \phi(i')  , i')_{\BB Z}$, and $[\phi(i') , i']_{\BB Z} \subset I$. 
\end{defn}

\begin{thm} \label{thm-cone-limit}
Let $Z$ be a correlated Brownian motion as in~\eqref{eqn-bm-cov}. Fix a countable dense set $\mcl Q\subset \BB R$. There is a coupling of countably many instances $\{ X^n \}_{n\in\BB N}$ of the infinite word $X$ described in Section~\ref{sec-burger-prelim} with $Z$ such that when $Z^n$, $\phi^n$, and $\phi_*^n$ are defined as in~\eqref{eqn-Z^n-def} and Notation~\ref{def-match-function}, respectively, with $X^n$ in place of $X$, the following holds a.s.  
\begin{enumerate}
\item $Z^n \rta Z$ uniformly on compacts. \label{item-cone-limit-Z} 
\item Suppose given a bounded open interval $I \subset \BB R$ with endpoints in $\mcl Q$ and $a \in I \cap \mcl Q$. Let $t$ be the maximal (Definition~\ref{def-maximal}) $\pi/2$-cone time for $Z$ in $I$ with $a\in [ v_Z(t),t]$. For $n\in\BB N$, let $i_n$ be the maximal flexible order time (with respect to $X^n$) $i$ in $n I$ with $a n \in [\phi^n(i)   , i ]$ (or $i_n = \lfloor a n \rfloor$ if no such $i$ exists); and let $t_n := n^{-1} i_n$. Then $t_n \rta t$. \label{item-cone-limit-maximal}
\item For $r  > 0$ and $a\in \BB R$, let $\tau^{a,r}$ be the smallest $\pi/2$-cone time $t$ for $Z$ such that $t\geq a$ and $t - v_Z(t) \geq r$. For $n\in\BB N$, let $\iota_n^{a,r}$ be the smallest $i\in\BB Z$ such that $X^n_i = \tb F$, $i \geq a n$, and $i - \phi^n(i) \geq r n - 1$ (or $\iota_n^{a,r} =\infty$ if no such $i$ exists); and let $\tau_n^{a,r} := n^{-1} \iota_n^{a,r}$. Then $\tau_n^{a,r} \rta \tau^{a,r}$ for each $(a,r) \in \mcl Q\times (\mcl Q\cap (0,\infty))$. \label{item-cone-limit-stopping}
\item For each sequence of positive integers $n_k \rta \infty$ and each sequence $\{i_{n_k} \}_{k\in\BB N}$ such that $X^{n_k}_{i_{n_k}} = \tb F$ for each $k\in\BB N$, $n_k^{-1} i_{n_k} \rta t \in \BB R$, and $\liminf_{k\rta\infty} (i_{n_k} - \phi^{n_k}(i_{n_k}) ) > 0$, it holds that $t$ is a $\pi/2$-cone time for $Z$ which is in the same direction as the $\pi/2$-cone time $n_k^{-1} i_{n_k}$ for $Z^{n_k}$ for large enough $k$ and in the notation of Definition~\ref{def-cone-time}, 
\eqbn
\left(n_k^{-1} i_{n_k} , \, n_k^{-1} \phi^{n_k}(i_{n_k}) , \, n_k^{-1} \phi^{n_k}_*(i_{n_k}) \right) \rta \left( t , v_Z(t) , u_Z(t) \right) .
\eqen
\label{item-cone-limit-times}
\end{enumerate}
\end{thm}

We also prove a variant of Theorem~\ref{thm-cone-limit} in which we condition on the event that $X(-n,-1)$ contains no burgers; see Corollary~\ref{prop-cone-limit-no-burger} below. Furthermore, we can choose the coupling of Theorem~\ref{thm-cone-limit} in such a way that the statements of the theorem also hold with a certain class of times $i$ in place of $\tb F$-times; and $\pi/2$-cone times for the time reversal of $Z$ in place of $\pi/2$-cone times for $Z$. See Theorem~\ref{thm-cone-limit-forward} in Appendix~\ref{sec-no-order}. In the setting of \cite[Theorem 1.13]{wedges}, $\pi/2$-cone times for the time reversal of $Z$ correspond to ``local cut times" of the space-filling $\op{SLE}_\kappa$ curve (see the proof of \cite[Lemma 12.4]{wedges}). 
 
The main difficulty in the proof of Theorem~\ref{thm-cone-limit} is showing that there in fact exist ``macroscopic $\tb F$-excursions" in the discrete model with high probability when $n$ is large. More precisely, 
 
\begin{prop} \label{prop-late-F}
For $\delta >0$ and $n\in\BB N$, let $\mcl E_n(\delta)$ be the event that there is an $i\in \{\lfloor \delta n \rfloor ,\dots,n\}$ such that $X_i = \tb F$ and $\phi(i) \leq 0$. Then 
\eqbn
\lim_{\delta\rta 0} \liminf_{n\rta\infty} \BB P\left(\mcl E_n(\delta)\right) = 1. 
\eqen
\end{prop}

We will prove Proposition~\ref{prop-late-F} in Section~\ref{sec-reg-var}, via an argument which requires most of the results of Sections~\ref{sec-prob-estimates},~\ref{sec-F-reg}, and~\ref{sec-no-burger}.
Proposition~\ref{prop-late-F} is not obvious from the results of \cite{shef-burger}. At first glance, it may appear that one should be able to obtain large $\tb F$-excursions in the discrete model by applying \cite[Theorem 2.5]{shef-burger} and considering times $t$ which are ``close" to being $\pi/2$-cone times for $Z^n$. However, this line of reasoning only yields times $t$ at which $U^n(t) \leq U^n(s) + \ep$ and $V^n(t) \leq V^n(s) + \ep$ for each $s \in [t',t]$ for some $t' < t$. One still needs Proposition~\ref{prop-late-F} or something similar to clear out the remaining $\ep n^{1/2}$ burgers on the stack at time $\lfloor t n \rfloor$ and produce an actual $\tb F$-excursion. Said differently, the $\pi/2$-cone times of a path do not depend continuously on the path in the uniform topology.

\subsection{Implications for critical FK planar maps}
\label{sec-loop-result}

\subsubsection{Area, boundary length, and complementary connected components}
\label{sec-quasi}

Let $(M , e_0 , S)$ be an infinite-volume critical FK planar map, i.e.\ the object encoded by the bi-infinite word $X$ of Section~\ref{sec-burger-prelim} via Sheffield's bijection, and let $\mcl L$ be the set of FK loops on $M$. 
Theorem~\ref{thm-cone-limit} implies scaling limit statements for various quantities associated with $\mcl L$. The reason for this is that one can explicitly describe many such quantities in terms of the $\tb F$-times for the corresponding word $X$. To illustrate this idea, in this paper we will obtain the scaling limit of the areas and boundary lengths of the bounded complementary connected components of macroscopic FK loops. Scaling limit statements for other functionals of the FK loops, such as the intersections and self intersections of loops, will be proven in both the infinite-volume and finite-volume settings in the subsequent works~\cite{gms-burger-local,gms-burger-finite,gwynne-miller-cle}. 

To state our result formally, we first need to introduce some terminology. 
Let $M^*$ be the dual map of $M$ and let $Q = Q(M)$ be the graph whose vertex set is the union of the vertex sets of $M$ and $M^*$ (i.e.\ the set of vertices and faces of $M$), with two such vertices joined by an edge if and only if they correspond to a face of $M$ and a vertex incident to that face. Note that $Q$ is a quadrangulation and that each face of $Q$ is bisected by an edge of $M$ and an edge of $M^*$. We define the \emph{root edge} of $Q$ to be the edge $\BB e_0$ of $Q$ with the same initial vertex as $e_0$ and which is the next edge clockwise (among all edges of $M$ and $Q$ that start at that endpoint) from $e_0$.  
Let $S^*$ be the set of edges of $M^*$ which do not cross edges of $S$, so that each face of $Q$ is bisected by either an edge of $S$ or an edge of $S^*$, but not both. 
We view loops in $\mcl L$ as cyclically ordered sets of edges of $Q$ which separate connected components of $S$ and $S^*$.

\begin{defn} \label{def-discrete-length}
For a set of edges $U\subset Q$, the \emph{discrete area} of $U$, denoted by $\op{Area}(U)$, is the number of edges in $U$. For a set of edges $A\subset S\cup S^*$, the \emph{discrete length} of $A$, denoted by $\op{Len}(A)$, is the number of edges in $A$. 
\end{defn}

\begin{defn} \label{def-discrete-bdy}
Suppose $C $ is a simple cycle in $S$ (resp. $S^*$) and $U$ is the set of edges of $Q$ disconnected from $\infty$ by $C$. We write $C:= \partial U$.
\end{defn}
 
\begin{defn} \label{def-fk-component}
Let $\ell \in \mcl L$ be an FK loop. Let $A$ and $A^*$ be the clusters of edges in $S$ and $S^*$ which are separated by $\ell$ (so that $A$ and $A^*$ are connected). A \emph{primal (resp. dual) complementary connected component} of $\ell$ is a set of edges $U\subset Q$ such that the following is true. There exists a simple cycle $C$ of $S$ (resp. $S^*$) which is contained in $A$ (resp. $A^*$) such that $U$ is the set of edges of $Q$ disconnected from $\ell$ by $C$; and there is no set $U'$ of edges of $Q$ satisfying the above property which properly contains $U$.  
\end{defn}

See Figure~\ref{fig-loop-component} for an illustration of Definition~\ref{def-fk-component}.

\begin{figure}[ht!]
 \begin{center}
\includegraphics[scale=1.1]{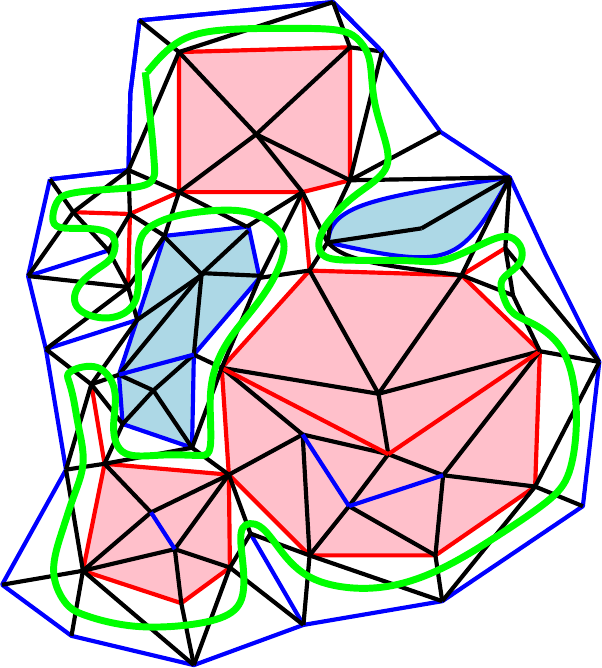} 
\caption{A subset of the quadrangulation $Q$ (black) and the edge sets $S$ and $S^*$ (red and blue) together with an FK loop (green) separating a primal cluster and a dual cluster. Primal (resp. dual) complementary connected components of the loop are shown in pink (resp. light blue). Other FK loops are not shown. } \label{fig-loop-component}
\end{center}
\end{figure}

\subsubsection{Statement of the scaling limit result} 
\label{sec-loop-limit}

Suppose we have coupled the sequence of words $\{X^n\}_{n\in\BB N}$ with the correlated two-dimensional Brownian motion $Z$ of~\eqref{eqn-bm-cov} in such a way that the conclusion of Theorem~\ref{thm-cone-limit} holds. 
For $n\in\BB N$, let $(M^n , e_0^n , S^n)$ be the infinite-volume critical FK planar map corresponding to $X^n$ under Sheffield's bijection. Also let $\mcl L^n$ be the corresponding set of FK loops. Let $\{\ell_j^n\}_{j\in\BB N}$ be the sequence of loops in $\mcl L^n$ which surround the root edge $e_0^n$. For $j\in\BB N$, let $U_{j,1}^n , \dots , U_{j,N_j}^n$ be the bounded complementary connected components of $\ell_j^n$, in order of decreasing area (with ties broken in some arbitrary manner). Let $M_j^{n,\infty}$ be the set of edges of $Q$ which are disconnected from $\infty$ by $\ell_j^n$ (including the edges on $\ell_j^n$). Also let $M_j^{n,\op{in}}$ be the set of edges of the quadrangulation $Q^n$ corresponding to $M^n$ which are surrounded by $\ell_j^n$ (including the edges on $\ell_j^n$). In other words, if $\ell_j^n$ surrounds a primal (resp. dual) cluster, then $M_j^{n,\op{in}}$ is obtained from $M_j^{n,\infty}$ by removing the dual (resp. primal) complementary connected components of $\ell_j^n$. 
 
Let $\{\sigma_j\}_{j\in\BB Z}$ be the ordered sequence of $\pi/2$-cone times (Definition~\ref{def-cone-time}) for $Z$ such that the following is true. We have $v_Z(\sigma_j) < 0 < \sigma_j$ and the largest $\pi/2$-cone time $t$ for $Z$ with $t < \sigma_j$ is in the opposite direction from $\sigma_j$. Also let $\Sigma_j$ be the set of maximal (Definition~\ref{def-maximal}) $\pi/2$-cone times $t$ for $Z$ in the interval $[v_Z(\sigma_j) ,\sigma_j]$ for which $u_Z(t) \geq v_Z(\sigma_j)$. Let $\Sigma_j^{\op{in}}$ be the set of $t\in \Sigma_j$ which are in the same direction as $\sigma_j$. Let $\{s_{j,k}\}_{k\in\BB N}$ be the elements of $\Sigma_j$, ordered so that $s_{j,k} - v_Z(s_{j,k})  > s_{j,k+1} - v_Z(s_{j,k+1})$ for each $k\in\BB N$. 

\begin{thm} \label{thm-loop-limit}
In the setting described just above (for any choice of coupling as in Theorem~\ref{thm-cone-limit}), the following is true almost surely. There is a random sequence of integers $\{b^n\}_{n \in \BB N}$ (the \emph{index shift}) such that for each $j,k\in\BB N$, we have (in the notation of Definitions~\ref{def-discrete-length} and~\ref{def-discrete-bdy}) 
\eqb \label{eqn-component-conv}
n^{-1} \op{Area} \left( U_{j + b^n , k}^n \right) \rta s_{j,k} - v_Z(s_{j,k}) \quad \op{and} \quad n^{-1/2} \op{Len} \left( \bdy U_{j + b^n , k}^n \right) \rta |Z_{v_Z(s_{j,k})} - Z_{s_{j,k}}|  .
\eqe 
Furthermore, for each $j\in\BB N$ we have
\eqb \label{eqn-full-conv}
n^{-1} \op{Area} \left( M_{j+b^n}^{n,\infty} \right) \rta  \sum_{t\in \Sigma_j} (t-v_Z(t)) .
\eqe 
and
\eqb \label{eqn-in-conv}
n^{-1} \op{Area} \left( M_{j+b^n}^{n,\op{in}} \right) \rta  \sum_{t\in \Sigma_j^{\op{in}} } (t-v_Z(t)) .
\eqe 
\end{thm}

The reason why we need the index shift $b^n$ in Theorem~\ref{thm-loop-limit} is that the FK loops surrounding the root edge in an FK planar map are naturally indexed by $\BB N$ (i.e., there is a smallest such loop) whereas the limiting times $\sigma_j$ are naturally indexed by $\BB Z$ (since a.s.\ there are infinitely many $\pi/2$-cone intervals for $Z$ surrounding 0). The shift $b^n$ can be chosen explicitly in several equivalent ways. For example, we can let $j_*^n$ for $n\in\BB N$ be the smallest $j\in\BB N$ for which the complementary connected component containing the root edge of the loop $\ell_j^n$ surrounding 0 has area at least $n$, let $j_*$ be the smallest $j\in\BB Z$ for which the maximal $\pi/2$-cone interval for $Z$ in $(v_Z(\sigma_j) , \sigma_j)$ which contains 0 has length at least 1, and let $b^n =  j_* - j_*^n $. 
 
Theorem~\ref{thm-loop-limit} will turn out to be a straightforward consequence of Theorem~\ref{thm-cone-limit}, once we have written down descriptions of the FK loops surrounding $\BB e_0$ and their complementary connected components in terms of the word $X$ (see Section~\ref{sec-fk-planar-map}).  
By re-rooting invariance of the planar maps $(M^n , e_0^n , S^n)$ (which is equivalent to translation invariance of the word $X$) and since the coupling of Theorem~\ref{thm-cone-limit} does not depend on the choice of root edge, Theorem~\ref{thm-loop-limit} immediately implies a joint scaling limit result for the sequences of FK loops surrounding countably many marked edges simultaneously.  

Note that Theorem~\ref{thm-loop-limit} does not include a scaling limit statement for the boundary length of the unbounded complementary connected components of FK loops. The description of this outer boundary length in terms of Sheffield's bijection is somewhat more complicated than that of the boundary lengths of the bounded complementary connected components (see Lemma~\ref{prop-discrete-outer} below), and proving that it converges requires estimates which are outside the scope of this paper. A scaling limit statement for the outer boundary lengths of FK loops will be proven in~\cite{gwynne-miller-cle}.

\begin{remark}
In this remark we explain how Theorem~\ref{thm-loop-limit} can be interpreted as a scaling limit result for FK loops toward a conformal loop ensemble on an independent Liouville quantum gravity cone.
It is not hard to see from the peanosphere construction of~\cite{wedges} together with some basic properties of CLE~\cite{shef-cle} and the LQG measure~\cite{shef-kpz} that the following is true. Let $\kappa $ be as in~\eqref{eqn-p-kappa} and let $\gamma = 4/\sqrt\kappa$. Let $(\mcl C , \Gamma)$ be the $\gamma$-LQG cone and independent CLE$_{\kappa }$ encoded by $Z$ as in~\cite[Theorems 1.13 and 1.14]{wedges}. Then the times $\sigma_j$ for $j\in\BB Z$ are in one-to-one correspondence with the CLE loops in $\Gamma$ surrounding the origin. Furthermore, for $j\in\BB Z$ the set $ \Sigma_j$ is in one-to-one correspondence with the set of bounded complementary connected components of the loop corresponding to $\sigma_j$. For $t\in \Sigma_j$, the quantum area and quantum boundary length of the corresponding complementary connected component are given by $t  - v_Z(t)$ and $|Z_{v_Z(t)} - Z_{t }|  $, respectively. Furthermore, $\Sigma_j^{\op{in}}$ corresponds to the set of complementary connected components which are surrounded by the loop. The proofs of these statements are straightforward once one has the results of~\cite{wedges} (essentially, these proofs are an exact continuum analogue of the descriptions of FK loops in terms of the inventory accumulation model found in Section~\ref{sec-fk-planar-map}). However, since we do not work directly with CLE or LQG here, these proofs are outside the scope of the present paper and will be given in~\cite{gwynne-miller-cle}. 
\end{remark}

\subsection{Basic notation}

Throughout the remainder of the paper, we will use the following notation. 

\begin{notation} \label{def-discrete-intervals}
For $a < b \in \BB R$, we define the discrete intervals $[a,b]_{\BB Z} := [a, b]\cap \BB Z$ and $(a,b)_{\BB Z} := (a,b)\cap \BB Z$. 
\end{notation}

\begin{notation}\label{def-asymp}
If $a$ and $b$ are two quantities, we write $a\preceq b$ (resp. $a \succeq b$) if there is a constant $C$ (independent of the parameters of interest) such that $a \leq C b$ (resp. $a \geq C b$). We write $a \asymp b$ if $a\preceq b$ and $a \succeq b$. 
\end{notation}

\begin{notation} \label{def-o-notation}
If $a$ and $b$ are two quantities which depend on a parameter $x$, we write $a = o_x(b)$ (resp. $a = O_x(b)$) if $a/b \rta 0$ (resp. $a/b$ remains bounded) as $x \rta 0$ (or as $x\rta\infty$, depending on context). We write $a = o_x^\infty(b)$ if $a = o_x(b^s)$ for each $s \in\BB R$. 
\end{notation}

Unless otherwise stated, all implicit constants in $\asymp, \preceq$, and $\succeq$ and $O_x(\cdot)$ and $o_x(\cdot)$ errors involved in the proof of a result are required to satisfy the same dependencies as described in the statement of said result.

\subsection{Outline}
The remainder of this paper is structured as follows. In Section~\ref{sec-fk-planar-map}, we assume Theorem~\ref{thm-cone-limit} and use it together with some elementary facts about Sheffield's bijection to deduce Theorem~\ref{thm-loop-limit}.

The remaining sections will be devoted to the proof of Theorem~\ref{thm-cone-limit}, which will require a number of estimates for the inventory accumulation model of~\cite{shef-burger}. 
In Section~\ref{sec-prob-estimates}, we prove a variety of probabilistic estimates related to this model. These include some estimates for Brownian motion, lower bounds for the probabilities of several rare events associated with the word $X$, and an upper bound for the number of flexible orders remaining on the stack at a given time which improves on~\cite[Lemma 3.7]{shef-burger}. 

In Section~\ref{sec-F-reg}, we prove a regularity result for the conditional law of the path $Z^n$ given that the word $X(-n,-1)$ contains no burgers. In Section~\ref{sec-no-burger}, we use said regularity result to prove convergence in the scaling limit of the conditional law of $Z^n|_{[-1,0]}$ given that $X(-n,-1)$ has no burgers (equivalently that $Z^n|_{[-1,0]}$ stays in the first quadrant) to the law of a correlated Brownian motion conditioned to stay in the first quadrant. In Section~\ref{sec-cone-conv}, we use the scaling limit result of Section~\ref{sec-no-burger} to obtain that a certain stopping time associated with the word $X$ has a regularly varying tail, deduce Proposition~\ref{prop-late-F} from this fact, and then deduce Theorem~\ref{thm-cone-limit} from Proposition~\ref{prop-late-F}. 

In Appendix~\ref{sec-no-order}, we will record analogues of some of the results of the paper when we consider words with no orders, rather than no burgers. These results are not needed for the proof of Theorems~\ref{thm-cone-limit} or~\ref{thm-loop-limit}, but are included for the sake of completeness and will be used in the subsequent papers~\cite{gms-burger-local,gms-burger-finite}.

For the convenience of the reader, we include an index of commonly used symbols in Appendix~\ref{sec-index}, along with the locations in the paper where they are first defined.

\section{Scaling limits for FK loops}
\label{sec-fk-planar-map}

In this section we will study the encoding of FK planar maps via Sheffield's bijection and see how Theorem~\ref{thm-cone-limit} implies Theorem~\ref{thm-loop-limit}. The rest of the paper will be devoted to the proof of Theorem~\ref{thm-cone-limit}. 
We start in Section~\ref{sec-burger-bijection} by reviewing the infinite-volume version of Sheffield's bijection, which encodes an infinite-volume FK planar map in terms of a bi-infinite word $X$ consisting of elements of $\Theta$ (recall Section~\ref{sec-burger-prelim}). In Sections~\ref{sec-detecting-inner-discrete} and~\ref{sec-fk-loops}, we will explain how this word $X$ encodes the complementary connected components of FK loops. Finally, in Section~\ref{sec-loop-limit-proof} we will explain how this encoding together with Theorem~\ref{thm-cone-limit} implies Theorem~\ref{thm-loop-limit}. 

\subsection{Sheffield's bijection} 
\label{sec-burger-bijection}

The primary reason for our interest in the inventory accumulation model of Section~\ref{sec-burger-prelim} is its relationship to FK planar maps via the bijection~\cite[Section 4.1]{shef-burger}.
Since the result of this paper primarily concern infinite-volume FK planar maps, in this subsection we will explain how to define an infinite-volume FK planar map and how to encode it by means of a bi-infinite word consisting of elements of $\Theta$. 

Fix $q \in (0,4)$. An infinite-volume (critical) FK planar map with parameter $q$ is a random triple $(M , e_0 , S)$ where $M$ is an infinite planar map, $e_0$ is an oriented root edge for $M$, and $S$ is a set of edges of $M$. This object is the limit in the Benjamini-Schramm topology~\cite{benjamini-schramm-topology} of finite-volume FK planar maps of size $n$ and parameter $q$ as $n\rta\infty$. The existence of this limit is alluded to in~\cite[Section 4.2]{shef-burger} and is explained more precisely in~\cite{chen-fk,blr-exponents}.
  
Suppose now that $(M , e_0 , S)$ is an infinite-volume FK planar map. We will describe how to associate a bi-infinite word with $(M , e_0 , S)$ which has the law of the word $X$ of Section~\ref{sec-burger-prelim}. The construction is essentially the same as the finite-volume bijection in~\cite[Section 4.1]{shef-burger} and is the inverse of the procedure described in~\cite[Proposition 9]{chen-fk}. See Figure~\ref{fig-hc-bijection} for an illustration of this construction.

\begin{figure}[ht!]
 \begin{center}
\includegraphics{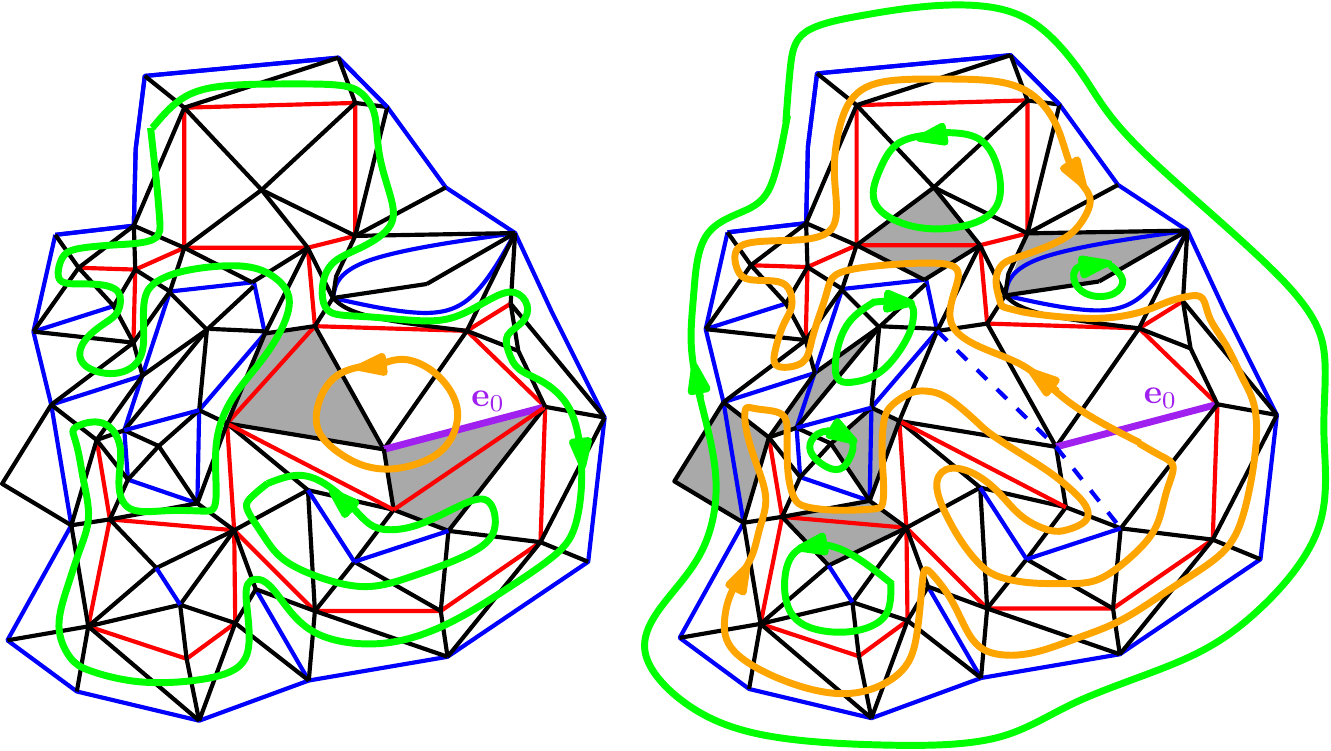} 
\caption{Illustration of the loop-joining procedure in Sheffield's bijection. Edges of $Q$ (resp. $S$, $S^*$) are shown in black (resp. blue, red).
The left panel shows the loop $\ell_0$ (orange) containing the root edge $\BB e_0$ of $Q$ (purple). The loops adjacent to $\ell_0$ are shown in green. The grey quadrilaterals are the last quadrilaterals crossed by $\ell_0$ which are also crossed by each of these loops. To join the orange and green loops into a single loop, we replace the edge of $S$ (resp. $S^*$) which bisects each of these grey quadrilaterals by the edge of $S^*$ (resp. $S$) which crosses it. 
The right panel shows the situation after making these replacements. We have joined the green loops in the left panel to the orange loop $\ell_0$ to get a new orange loop. 
The green loops in the right panel are the ones adjacent to this new orange loop, and the grey quadrilaterals are the ones in which we will replace an edge of $S$ or $S^*$ at the next stage of the construction.  
Iterating this procedure countably many times a.s.\ joins all of the loops together into a single space-filling path $\lambda$. 
 } \label{fig-hc-bijection}
\end{center}
\end{figure}


Define the dual map $M^*$, the rooted quadrangulation $(Q, \BB e_0)$, and the dual edge set $S^*$ as in Section~\ref{sec-quasi}. Also let $\mcl T$ be the graph whose edge set is the union of $S$, $S^*$, and the edge set of $Q$, and note that $\mcl T$ is a triangulation. 
 
Let $\mathcal C$ be the set of connected components of the graph obtained by removing all edges not in $S$ from $M$. Each element of $\mathcal C$ is surrounded by a loop $\ell$ (described by a cyclically ordered set of edges in $Q$) which passes through no edges in $S$ or $S^*$. Let $\mathcal L$ be the set of such loops and let $\ell_0$ be the loop in $\mcl L$ which passes through the root edge $\BB e_0$. 
Let $C_1 , \dots , C_k$ be the connected components in the set of triangles of $\mcl T$ obtained by removing the triangles crossed by $\ell_0$ from $\mcl T$. The boundary of each $C_j$ shares an edge with at least one triangle in $\ell_0$. Let $A_j$ be the last triangle sharing an edge with $\partial C_j$ hit by $\ell_0$ when it is traversed in the clockwise direction starting from $\BB e_0$. Let $a_j$ denote the shared edge. 

If $a_j \in S$, we replace $a_j$ by the edge in $M^*$ which it crosses, and if $a_j\in S^*$, we replace $a_j$ with the edge in $M$ which it crosses. Call the new edge a \emph{fictional edge}. Making these replacements for each $j\in [1,k]_{\BB Z}$ joins one loop in each of $C_1,\dots,C_k$ to the loop $\ell_0$. 
Since $(Q,S)$ is the local limit of finite-volume FK planar maps~\cite[Section 4.2]{shef-burger}, it follows that we can a.s.\ iterate this procedure countably many times (each time starting with a larger initial loop $\ell_0$) to join all of the loops in $\mcl L$ into a single bi-infinite path $\lambda$ which hits every edge of $Q$ exactly once and separates a spanning tree $ T$ of $M$ from a dual spanning tree $T^*$ of $M^*$. We view $\lambda$ as a function from $\BB Z$ to the edge set of $Q$, with $\lambda(0) = \BB e_0$.  
  
Each edge $\lambda(i)$ for $i\in\BB Z$ connects a vertex in $M$ to a vertex of $M^*$. For each $i\in  \BB Z $, write $d(i)$ for the distance in the primal tree $ T$ from the primal endpoint of $\lambda(i)$ to the primal endpoint of $\BB e_0$ and $d^*(i)$ for the distance in the dual tree $ T^*$ from the dual endpoint of $\lambda(i)$ to the dual endpoint of $\BB e_0$. We also write $D(i) = (d(i) ,d^*(i))$. We associate to the loop $\lambda$ a bi-infinite word $Y = \dots Y_{-1} Y_0 Y_1\dots$ with symbols in $\{\tc H,\tc C , \tb H , \tb C\}$ as follows. For $i\in \BB Z$, we set $Y_i = \tc H, \tc C,  \tb H, $ or $\tb C$ according to whether $D(i)  - D(i-1) = (1,0), (0,1) , (-1,0),$ or $(0,-1)$.   
Then in the terminology of Notation~\ref{def-theta-count}, we have
\eqbn
d(i) = d(Y(1,i))  \quad \op{and} \quad d^*(i) = d^*(Y(1,i))  , \quad \forall i \in \BB N
\eqen
where $Y(1,i)$ is as in~\eqref{eqn-X(a,b)} with $Y$ in place of $X$. 

Note that $\lambda$ crosses each quadrilateral of $Q$ twice. A burger in the word $Y$ corresponds to the first time at which $\lambda$ crosses some quadrilateral, and the order matched to this burger corresponds to the second time at which $\lambda$ crosses this quadrilateral. 

The bi-infinite word $X$ corresponding to the triple $(M,e_0,S)$ is constructed from $Y$ as follows. Whenever $\lambda$ crosses a quadrilateral bisected by a fictional edge for the second time at time $i$, we replace $Y_i$ by an $\tb F$-symbol. As explained in~\cite[Section 4.1]{shef-burger}, this does not change the match of any order in the word $Y$. Furthermore, passing to the infinite-volume limit in the finite-volume version of Sheffield's bijection shows that the symbols of $X$ are iid samples from the law~\eqref{eqn-theta-prob} with $p = \sqrt q/(2 + \sqrt q)$.

\subsection{Cycles and discrete ``bubbles"} 
\label{sec-detecting-inner-discrete}

Throughout the remainder of this section we continue to assume that $(M , e_0 , S)$ is an infinite-volume FK planar map and use the notation of Section~\ref{sec-burger-bijection}. 
In the next two subsections, we will give explicit descriptions of the objects involved in Theorem~\ref{thm-loop-limit} in terms of the bi-infinite word $X$ which encodes the infinite-volume FK planar map. We note that although the description given here is in the context of the infinite-volume version of Sheffield's bijection, a completely analogous description holds in the finite-volume case, with the same proofs.
  
Our first task is to describe how cycles in $S$ and $S^*$ are encoded by the word $X$. 
To this end, let $\mcl I$ be the set of $i\in \BB Z$ such that $X_i = \tb F$. Also let $\mcl I^L$ (resp. $\mcl I^R$) be the set of $i\in \mcl I$ such that $X_{\phi(i)} = \tc C$ (resp. $X_{\phi(i)} = \tc H$). We recall the notations $\phi(i)$ for the match of $i \in \BB Z$ and $\phi_*(i)$ for the index of the match of the rightmost order in $X(\phi(i) ,i)$ from Notation~\ref{def-match-function}.

The set $\mcl I$ is the discrete analogue of the set of $\pi/2$-cone times of the correlated Brownian motion $Z$. The match $\phi(i)$ of $i$ corresponds (modulo a constant-order error) to the time $v_Z(\cdot)$ in Definition~\ref{def-cone-time} and the time $\phi_*(i)$ corresponds (modulo a constant order error) to the time $u_Z(\cdot)$ in Definition~\ref{def-cone-time}. The sets $\mcl I^L$ and $\mcl I^R$ correspond to the left and right $\pi/2$-cone times of $Z$, respectively, which explains the choice of notation. 

Intervals $[\phi(i) , i-1]_{\BB Z}$ with $i \in \mcl I^L \cup \mcl I^R$ are closely related to cycles in $S\cup S^*$, as the following lemma demonstrates. 
  
\begin{lem} \label{prop-cone-loops-discrete}
Let $i\in \mcl I^L$ and let $U = \lambda([\phi(i) , i-1]_{\BB Z})$, so that $U$ is a set of edges of $Q$. There is a simple cycle $C\subset S$ such that $U$ is the set of edges of $Q$ disconnected from $\infty$ by $C$. In this case $\op{Area}(U) = i - \phi(i)$ and $\op{Len}(C) = |X(\phi(i) , i)|+1$ (recall Definition~\ref{def-discrete-length}). Furthermore, $\phi_*(i)$ is the first time at which $\lambda$ crosses a quadrilateral of $q$ bisected by an edge of $C$. The same holds with $\mcl I^R$ in place of $\mcl I^L$ and $S^*$ in place of $S$.   
\end{lem}

Lemma~\ref{prop-cone-loops-discrete} implies that one can interpret Theorem~\ref{thm-cone-limit} as a scaling limit result for the joint law of the areas and boundary lengths of certain macroscopic cycles of $S$ and $S^*$.

\begin{proof}[Proof of Lemma~\ref{prop-cone-loops-discrete}]
First consider a time $i\in\mcl I^L$. The construction of Sheffield's bijection implies that there is a quadrilateral $q$ of $Q$ bisected by an edge $a$ of $S$ such that $\lambda$ crosses $q$ for the first time at time $\phi(i)$ and for the second time at time $i$.  
The set $A$ of edges of $T$ which bisect quadrilaterals of $Q$ crossed (either once or twice) by $\lambda$ during the time interval $[\phi(i) , i]_{\BB Z}$ is a connected graph. Since each edge of $q$ is incident to an edge in $A$, the set $  A \cup \{a\}$ disconnects $\lambda([\phi(i) , i-1]_{\BB Z})$ from the root edge, so contains a simple cycle $C  \subset S$ which disconnects $\lambda([\phi(i) , i-1]_{\BB Z})$ from the root edge, none of whose edges are crossed by $\lambda$ except for $a$.
Since $\lambda$ cannot cross itself or $C \setminus \{a\}$ and hits every edge of $Q$, it must be the case that $U = \lambda([\phi(i) , i-1]_{\BB Z})  $ is precisely the set of edges of $Q$ disconnected from the root edge by $C$. 

We now claim that $C\setminus \{a\}$ is precisely the set of edges of $A$ which bisect quadrilaterals crossed only once by $\lambda$ during $[\phi(i ) , i]_{\BB Z}$. Indeed, if $b \in A$ is such an edge, then part of the quadrilateral bisected by $b$ is not disconnected from $\infty$ by $C$, so $b$ cannot be disconnected from $ \infty$ by $C$, so $b\in C \setminus \{a\}$. Conversely, if $b\in C\setminus \{a\}$, then some edge of the quadrilateral bisected by $b$ lies outside $C$, and this edge is not hit by $\lambda$ during $[\phi(i) , i]_{\BB Z}$. 
Since $X_i = \tb F$ and $X_{\phi(i)} = \tc C$, the word $X(\phi(i) , i)$ contains only hamburger orders, so the times during $[\phi(i) ,i]_{\BB Z}$ at which $\lambda$ crosses a quadrilateral bisected by an edge of $C\setminus \{a\}$ correspond precisely to the symbols in $X(\phi(i) , i)$.   

It is immediate from the above descriptions of $U$ and $C$ that $\op{Area}(U) = i - \phi(i)$ and $\op{Len}(C) = |X(\phi(i) , i)|+1$.
Furthermore, recalling Notation~\ref{def-match-function}, we see that $\phi_*(i)$ is the first time at which $\lambda$ crosses a quadrilateral of $q$ bisected by an edge of $S$ which is crossed for the second time during the time interval $[\phi(i)+1 ,i]_{\BB Z}$, i.e.\ the first time $\lambda$ crosses a quadrilateral of $q$ bisected by an edge of $C$.

The statement for $\mcl I^R$ follows from symmetry.
\end{proof}

In light of Lemma~\ref{prop-cone-loops-discrete}, it will be convenient to have a notation for the discrete ``bubble" corresponding to an $\tb F$-time $i\in\mcl I$. 
 
\begin{defn} \label{def-discrete-bubble-map}
For $i \in \mcl I $, we write $P(i ) := \lambda([\phi(i) ,   i-1]_{\BB Z})$.  
\end{defn}

We next state a partial converse to Lemma~\ref{prop-cone-loops-discrete}, giving conditions for a cycle in $S$ or $S^*$ to correspond to an $\tb F$.

\begin{defn} \label{def-max-cycle}
A \emph{maximal simple cycle} in the edge set $S$ (resp. $S^*$) is a simple cycle $C\subset S$ such that the following is true. Let $U$ be the set of edges of $Q$ disconnected from $\infty$ by $C$. There is no simple cycle $C'\subset S$ (resp. $C'\subset S^*$) such that $C'$ shares an edge with $C$ and $C'$ disconnects $U$ from $\infty$. 
\end{defn}

Our main example of a maximal simple cycle is the boundary of a complementary connected component of an FK loop (Definition~\ref{def-fk-component}).

\begin{lem} \label{prop-max-cycle}
Suppose $C \subset S \cup S^*$ is a maximal simple cycle. There exists $i\in \mcl I$ such that $P(i)$ is the set of edges of $Q$ disconnected from $\infty$ by $C$.  
\end{lem}
\begin{proof}
By symmetry it suffices to treat the cases of cycles in $S$. 
Suppose $C\subset S$ is a maximal simple cycle and let $U$ be the set of edges of $Q$ disconnected from $\infty$ by $C$. Let $i_U'$ be the smallest $i\in \BB Z$ for which $\lambda(i) \in U$ and let $i_U  := \phi(i_U')$. By Sheffield's bijection $X_{i_U'} = \tc C$ and $X_{i_U} = \tb F$. Let $U' := P(i_U)$. We will show that $U' = U$. By Lemma~\ref{prop-cone-loops-discrete}, $C' := \partial U'$ is a simple cycle in $S$. Furthermore, $C' \cap C$ contains the edge of $S$ which bisects the quadrilateral of $Q$ crossed by $\lambda$ at times $i_U$ and $i_U'$. By maximality of $C$ we must have $U' \subset U$.  Now suppose by way of contradiction that there is an edge $e$ of $U$ which is not contained in $U'$. Then there is a quadrilateral $q$ of $Q$ with edges contained in $U$ (bisected by an edge of $C'$) which is crossed by $\lambda$ for the first time during the time interval $[i_U' , i_U-1]_{\BB Z}$ and for the second time after time $i_U$. This contradicts the fact that $X(\phi(i_U) , i_U)$ contains no burgers.
\end{proof}

Our next lemma allows us to identify when two cycles in $S$ or $S^*$ intersect in terms of the word $X$. 
 
\begin{lem} \label{prop-L-nest}
Let $i,i'\in \mcl I^L$ or $i,i' \in \mcl I^R$. Suppose $P(i)\subset P(i')$. Then $\partial P(i ) \cap \partial P(i') \not= \emptyset$ (Definition~\ref{def-discrete-bdy}) if and only if $\phi_*(i) \leq \phi(i')$ (Definition~\ref{def-match-function}). 
\end{lem}
\begin{proof}
By symmetry we can assume without loss of generality that $i,i' \in \mcl I^L$.  

First suppose $\partial P(i) \cap \partial P(i') = \emptyset$. Then the cycle $\partial P(i)\subset S$ is disconnected from $\infty$ by $\partial P(i')$. Therefore each edge quadrilateral of $Q$ which contains an edge of $P(i)$ has all of its edges contained in $P(i')$. Consequently, each $k \in [\phi(i) , i]_{\BB Z}$ satisfies $\phi(k) \in [\phi(i') , i']_{\BB Z}$. In particular, $\phi_*(i) > \phi(i')$.  

Conversely, suppose $\partial P(i) \cap\partial P(i') \not=\emptyset$. Let $k$ be the first time at which $\lambda$ crosses a quadrilateral bisected by an edge of $\partial P(i)\cap\partial P(i')$. Then $k \leq \phi(i')$ and $\phi(k) \in [\phi(i) , i]_{\BB Z}$. Therefore $\phi_*(i) \leq k  \leq \phi(i')$. 
\end{proof}

\subsection{Complementary connected components of FK loops} 
\label{sec-fk-loops}

In this subsection we will describe the complementary connected components of FK loops on the infinite-volume FK planar map $(M , e_0 , S)$ in terms of the word $X$ (recall Definition~\ref{def-fk-component}).  

Let $\{\ell_j \}_{j\in\BB N}$ be the sequence of loops in $\mcl L$ which disconnect the root edge $\BB e_0$ from $\infty$, as in Section~\ref{sec-loop-limit}.
We assume that our numbering is such that if $j$ is odd then $\ell_j$ surrounds a component of $S$ and if $j$ is even then $\ell_j$ surrounds a component of $S^*$. This can be arranged by including (as $\ell_1$) the loop which contains $\BB e_0$ if and only if this loop surrounds a component of $S $. 
 
For $j \in \BB N$, let $M_j^0$ be the complementary connected component of $M\setminus \ell_j$ incident to the primal endpoint of $\BB e_0$ (Definition~\ref{def-fk-component}). If $j \geq 2$, this is the same as the complementary connected component of $M\setminus \ell_j$ containing $\BB e_0$. Also let $M_j^\infty$ be the union of $\ell_j $ and the set of edges in $Q$ which are disconnected from $ \infty$ by $\ell_j$. 

For $j\in\BB N$, let $\iota_j$ be the largest $i\in \BB N$ such that $\lambda(i-1) \in M_j^0$. Also let $\wt\theta_j$ (resp. $\theta_j$) be the time at which $\lambda$ starts tracing the loop $\ell_j $ (resp. the time immediately after $\lambda$ finishes tracing $\ell_j $). Let $I_j$ (resp. $\Theta_j$) be the set of maximal $\tb F$-times (Definition~\ref{def-maximal}) $i \in \mcl I $ in $(\wt\theta_j, \theta_j)_{\BB Z} $ such that the bubble $P(i)$ of Definition~\ref{def-discrete-bubble-map} is not (resp. is) contained in $M_j^\infty$. 

See Figure~\ref{fig-loop-sequence} for an illustration of these objects.

\begin{figure}[ht!]
 \begin{center}
\includegraphics{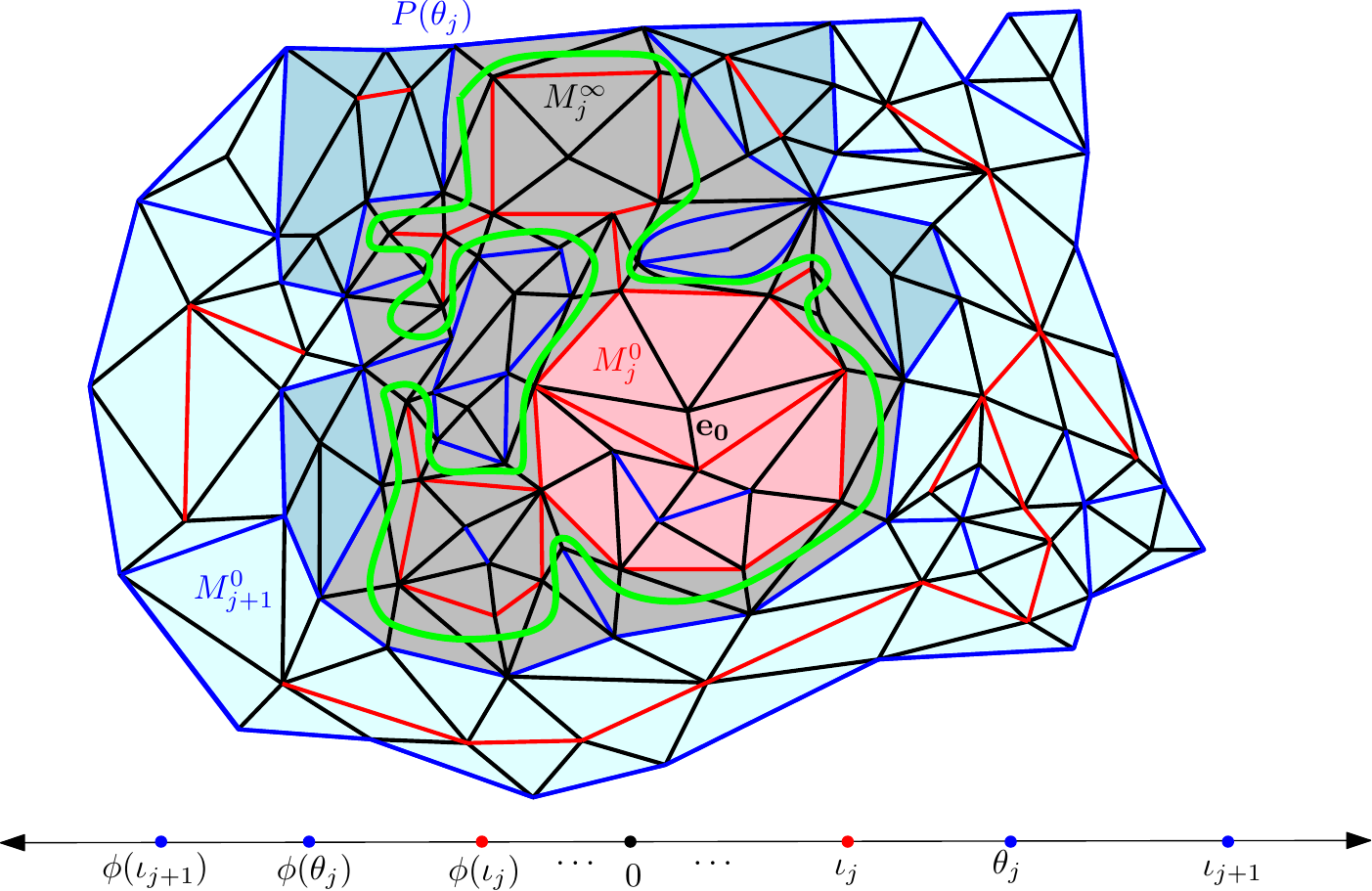} 
\caption{An illustration of the edge set $M_{j+1}^0 = P(\iota_{j+1}^0)$ and several other regions contained in it. The set $P(\theta_j)$ is the union of the blue and grey regions. The individual blue regions are sets of the form $P(i)$ for $i\in I_j$. The set $M_j^\infty$ is shown in grey, and is not traced by $\lambda$ in a single interval of time. The set $M_j^0 = P(\iota_j)$ is shown in pink. The loop $\ell_j $ is shown in green (other FK loops are not shown). The times which encode these sets are shown on a number line below the figure. } \label{fig-loop-sequence}
\end{center}
\end{figure}

We first describe the the times $\iota_j$ in terms of the word $ X$.
 
\begin{lem} \label{prop-discrete-inner}
Let $j\in \BB N$ be odd (resp. even). Then $\iota_j$ is the largest time $i \in \mcl I^L$ (resp. $i\in \mcl I^R$) such that $i   < \iota_{j+1}$ and $0 \in [\phi(i),  i]_{\BB Z}$. 
Furthermore, in the notation of Definition~\ref{def-discrete-bubble-map} we have $P(\iota_j) = M_j^0$.
\end{lem}
\begin{proof}
By symmetry we can assume without loss of generality that $j$ is odd. 
By the definition of the loops $\ell_j $ and by Lemmas~\ref{prop-cone-loops-discrete} and~\ref{prop-max-cycle}, we have $\iota_j \in \mcl I^L$ and $P(\iota_j) = M_j^0$. 
Since $0\in M_j^0 \subset M_{j+1}^0$, we consequently have $0 \in [\phi(\iota_{j}) ,  \iota_{j}]_{\BB Z} \subset [\phi(\iota_{j+1}) ,  \iota_{j+1}]_{\BB Z}$.  

Suppose $i \in \mcl I^L$ with $i  < \iota_{j+1}$ and $0\in  [ \phi(i),  i]_{\BB Z}$. Since $\partial P(i)$ consists of edges of $S$ and $\partial P(\iota_{j+1})$ consists of edges of $S^*$, it follows from the construction in Section~\ref{sec-burger-bijection} that $\lambda$ must branch into the interior of one of the outermost loops in $\mcl L$ contained in $P(\iota_{j+1})$ before entering $P(i)$, so $P(i)$ must be disconnected from $\infty$ by some loop $ \ell' \in \mcl L$ surrounded by $\ell_{j+1}$ with $\ell' \not= \ell_{j+1}$. Let $i'$ be the first time $\lambda$ traces an edge of this loop $\ell'$. By the construction in Section~\ref{sec-burger-bijection} we have $i' \in \mcl I^L \cap [\phi(i ) ,i]_{\BB Z}$ and by our choice of $i'$ we have $0 \in [\phi(i'), i']_{\BB Z}$. The loop $\ell' $ has the same orientation as $\ell_j $, so $\ell'  $ must either be equal to $\ell_j $ or disconnected from $\BB e_0$ by $\ell_j $. In the latter case, $\ell' $ and hence also $P(i)$ is contained in $M_j^0$. In particular, $i \leq \iota_j$. 

The final assertion of the lemma is immediate from Lemma~\ref{prop-cone-loops-discrete}.
\end{proof}

Next we will describe the times $\wt\theta_j$ and $\theta_j$ in terms of $ X$.

\begin{lem} \label{prop-discrete-outer}
Let $j \in [2,\infty)_{\BB Z}$ be odd (resp. even). In the notation of Section~\ref{sec-fk-loops}, we have that $\theta_j$ is the first time $i > \iota_j$ such that $i \in \mcl I^R$ (resp. $i\in \mcl I^L$) and $0\in [\phi(i), i]_{\BB Z}$; and $\wt\theta_j = \phi(\theta_j)$.  
Furthermore, 
\eqb \label{eqn-outer-area-discrete}
 \op{Area}\left( M_j^\infty\right)=  \theta_j - \wt\theta_j - \sum_{i  \in I_j} (i - \phi(i))  
\eqe
and
\eqb \label{eqn-nu-discrete} 
\op{Len} \left(\partial M_j^\infty \right) =  - |X(\wt\theta_j , \theta_j)|   +  \sum_{i  \in I_j } (  |X(\phi(i) , i)|   +1)    + 2 \# U_j + 1 , 
\eqe  
where $U_j$ is the set of $i \in [\wt\theta_j, \theta_j]_{\BB Z}$ with $X_i = \tb C$ (resp. $X_i = \tb H$) and $\phi(i)  < \wt\theta_j$ such that $i\notin [\phi(i' ) ,   i']_{\BB Z}$ for any $i'  \in I_j$. 
\end{lem}
\begin{proof} 
Assume without loss of generality that $j$ is odd. It is clear from Sheffield's bijection that $\theta_j \in \mcl I^R$, $\wt\theta_j = \phi(\theta_j)$,  $\theta_j > \iota_j$, and $0\in [\wt\theta_j, \theta_j]_{\BB Z}$. Since $\ell_j $ is the smallest loop in $\mcl L$ which surrounds $M_j^0$, it follows that $\theta_j$ is in fact the smallest time in $\mcl I^R$ with these properties. 
 
To prove the formulas for $\op{Area}(M_j^\infty)$ and $\op{Len}(\partial M_j^\infty)$, we observe that each element of $P(I_j)$ is contained in a loop in $\mcl L$ which lies outside $M_j^\infty$. Therefore 
\eqb \label{eqn-loop-outer-decomp}
M_j^\infty = P(\theta_j) \setminus \bigcup_{i  \in I_j} P(i ) .
\eqe 
This together with the maximality condition in the definition of $I_j$ immediately implies the formula~\eqref{eqn-outer-area-discrete}. 

To prove~\eqref{eqn-nu-discrete}, we first argue that 
\eqb \label{eqn-U-char}
\#\left( \bdy M_j^\infty  \cap \bdy P(\theta_j) \right) = \# U_j + 1 .
\eqe 
Indeed, if $i \in U_j$ then at time $i$ the path $\lambda$ crosses a quadrilateral bisected by an edge of $\bdy P(\theta_j)$ which does not belong to $\bdy P(i)$ for any $i\in I_j$. By~\eqref{eqn-loop-outer-decomp}, such an edge must belong to $\bdy M_j^\infty$. On the other hand, each $\bdy P(i)$ for $i\in I_j$ is a simple cycle (Lemma~\ref{prop-cone-loops-discrete}) so no edge of $\bdy P(i)$ belongs to both $\bdy M_j^\infty$ and $\bdy P(\theta_j)$.
Hence every edge in $\bdy M_j^\infty \cap \bdy P(\theta_j)$ except for the first edge of $\bdy P(\theta_j)$ (equivalently $\bdy M_j^\infty$) crossed by $\lambda$ belongs to $U_j$. We must obtain~\eqref{eqn-U-char}. 

The relation~\eqref{eqn-loop-outer-decomp} implies that each edge of $\bdy M_j^\infty \setminus \bdy P(\theta_j)$ belongs to $\bdy P(i)$ for some $i\in I_j$. Recalling the formula for boundary length from Lemma~\ref{prop-cone-loops-discrete}, we find that the sum of the first two terms on the right in~\eqref{eqn-nu-discrete} is equal to $\# \left( \bdy M_j^\infty \setminus \bdy P(\theta_j) \right) + 1$ minus the number of edges in $\bdy P(\theta_j)$ which do not belong to $\bdy P(i)$ for any $i\in I_j$. By~\eqref{eqn-U-char}, the number of such edges is $\# U_j + 1$. By combining this with~\eqref{eqn-U-char}, we obtain~\eqref{eqn-nu-discrete}.  
\end{proof}

We will now describe the time set $I_j$ defined as in the beginning of this subsection solely in terms of $X$. 
 
\begin{lem} \label{prop-extra-bubbles-discrete}
 Let $j\in\BB Z$ be odd (resp. even). 
The set $I_j$ is the same as the set of maximal elements $i$ of $\mcl I$ in $(\wt\theta_j , \theta_j)$ such that $i  \in \mcl I^R$ (resp. $i\in \mcl I^L $) and $\phi_*(i ) < \wt\theta_j$.  
\end{lem}
\begin{proof}
Assume without loss of generality that $j$ is odd. 
Suppose that $i  \in I_j$. Since $P(i ) \not\subset M_j^\infty$, it follows from Sheffield's bijection that $\lambda$ must branch outward from the loop $\ell_j $ when it begins tracing $P(i )$, i.e.\ it crosses $\partial M_j^\infty$. Therefore $i  \in \mcl I^R$. Since $\bdy P(i)$ is a simple cycle which is not disconnected from $\infty$ by $\bdy M_j^\infty$, 
we can find an edge $a\in \partial P(i) \setminus \partial M_j^\infty$. Let $q$ be the quadrilateral of $Q$ bisected by $a$. Note that the edge of $\partial P(i)$ which is crossed by $\lambda$ belongs to $\partial M_j^\infty$, so $a$ is not replaced by a fictional edge. Let $k$ be the first time $\lambda$ crosses the quadrilateral $q$. Then we have $X_k = \tc C$. We claim that $k  < \wt \theta_j$. If not, then $k \in [\phi(i') , i']_{\BB Z}$ for some $i'\in I_j$ with $i' < i$. But, $X(\phi(i') , i')$ contains only orders, so this is impossible. Hence $\phi_*(i) \leq k < \wt\theta_j$.  

Conversely, it follows from Lemmas~\ref{prop-cone-loops-discrete} and \ref{prop-L-nest} that any $i\in  \BB Z $ satisfying the conditions of the lemma is such that $\partial P(i) \cap \partial P(\theta_j) \not=\emptyset$ and $P(i)$ is not contained in $P(i')$ for any $i' \in \mcl I \cap (\wt\theta_j , \theta_j)_{\BB Z}$. Complementary connected components of $\ell_j$ which do not contain the root edge have boundaries disjoint from $\partial P(\theta_j)$. Therefore $P(i)$ cannot be contained in such a component, so we must have $i \in I_j$. 
\end{proof}

Finally, we describe the significance of the time set $\Theta_j$ (which we recall is the same as the set of maximal $\tb F$-times in $(\wt\theta_j , \theta_j)_{\BB Z}$ which are not contained in $I_j$). 

\begin{lem}\label{prop-discrete-components}
Let $j\in \BB N$ be odd (resp. even). Then $P$ maps $\Theta_j$ to the set of bounded complementary connected components of the loop $\ell_j$ (Definition~\ref{def-fk-component}). 
Elements of $\Theta_j \cap \mcl I^L$ (resp. $\Theta_j \cap \mcl I^R$) correspond to components in the interior of $\ell_j $ and elements of $\Theta_j \cap \mcl I^R$ (resp. $\Theta_j \cap \mcl I^L$) correspond to components which are not in the interior of $\ell_j$.
\end{lem}
\begin{proof}
Assume without loss of generality that $j$ is odd. 
Let $U$ be a bounded complementary connected component of $Q\setminus \ell_j $.  
The set $\partial U$ is a maximal simple cycle (Definition~\ref{def-max-cycle}). 
By Lemma~\ref{prop-max-cycle}, there exists $i \in \mcl I \cap (\wt\theta_j ,\theta_j)_{\BB Z}$ such that $P(i) = U$. This $i$ cannot belong to $I_j$ since $U\subset M_j^\infty $. To show that $i\in \Theta_j$ it remains to check that $i$ is maximal in $(\wt\theta_j , \theta_j)_{\BB Z}$. If not, then there is an $i' \in (\wt\theta_j , \theta_j)_{\BB Z} \cap \mcl I$ with $[\phi(i) , i]_{\BB Z}\subset (\phi(i'),i')_{\BB Z}$. By Lemma~\ref{prop-cone-loops-discrete}, $\partial P(i)$ is a cycle in either $S$ or $S^*$. Such a cycle cannot cross the loop $\ell_j$, so since it surrounds $\partial P(i)$ it must in fact surround $\ell_j$ (recall Definition~\ref{def-fk-component}). But then $P(i') \not\subset \lambda( (\wt\theta_j ,\theta_j)_{\BB Z}  )$, which contradicts our choice of $i'$. 

Conversely, suppose $i\in \Theta_j$. Since $i\notin I_j$, we have $P(i) \subset M_j^\infty$. Therefore $P(i) \subset U$ for some bounded complementary connected component $U$ of $\ell_j$. By Lemma~\ref{prop-max-cycle}, $U = P(i')$ for some $\tb F$-time $i' \in [\wt\theta_j , \theta_j]_{\BB Z}$. By maximality of $i$ we have $P(i) =U$. 
 
The distinction between $I_j\cap \mcl I^L$ and $I_j \cap \mcl I^R$ comes from the fact that $\ell_j$ surrounds a cluster of $S$. 
\end{proof}

\subsection{Proof of Theorem~\ref{thm-loop-limit}}
\label{sec-loop-limit-proof}
 
In this subsection we will prove scaling limit statements for the objects studied in Sections~\ref{sec-fk-loops} which will eventually lead to a proof of Theorem~\ref{thm-loop-limit}. 

Throughout this subsection, we fix a coupling of $\{X^n\}_{n\in\BB N}$ with $Z$ as in Theorem~\ref{thm-cone-limit} with $\mcl Q = \BB Q$ and let $\{(M^n , e_0^n , S^n)\}_{n\in\BB N}$ be the corresponding FK planar maps. We use the notation of Section~\ref{sec-fk-loops} but we add an extra superscript or subscript $n$ to each of the objects involved to denote which of the FK planar maps $\{(M^n , e_0^n , S^n)\}_{n\in\BB N}$ it is associated with. We define $\sigma_j > 0$ and $\Sigma_j \subset (v_Z(\sigma_j) , \sigma_j) $ for $j\in\BB Z$ as in Section~\ref{sec-loop-limit}. We also let $\tau_j$ be the largest $\pi/2$-cone time $t$ for $Z$ with $\tau_j < \sigma_j$, so that $\tau_j$ is in the opposite direction from $\sigma_j$ and is in the same direction as $\sigma_{j-1}$. Finally, we let $T_j$ be the set of maximal $\pi/2$-cone times $t$ for $Z$ in the interval $(v_Z(\sigma_j) , \sigma_j)$ which satisfy $u_Z(t)  < v_Z(\sigma_j)$, i.e.\ those which do not belong to $\Sigma_j$.  

The reader should note that the only inputs in the proofs of the results in this section are Theorem~\ref{thm-cone-limit} and the description of the FK loops in Section~\ref{sec-fk-loops}. In particular, if we had a finite-volume analogue of Theorem~\ref{thm-cone-limit} (which will be proven in~\cite{gms-burger-finite}) the argument of this subsection would immediately yield a finite-volume version of Theorem~\ref{thm-loop-limit}. 

Our first lemma gives convergence of the times corresponding to the connected component of a given loop which contains the root edge, and implies the existence of the index shift $b^n$ appearing in Theorem~\ref{thm-loop-limit}.

\begin{lem} \label{prop-inner-conv}  
For $n\in\BB N$, let $\tau_j^n := n^{-1} \iota_j^n$ (in the notation of Lemma~\ref{prop-discrete-inner}). Let $b^n$ be the smallest $j\in \BB N $ for which $\tau_j^n \geq (\tau_0 + \tau_{-1})/2 $, with $\tau_{-1}$ defined just above. Almost surely, for each $j\in\BB Z$ we have $\tau_{j+b^n}^n\rta \tau_j$ as $n\rta\infty$. 
\end{lem} 
\begin{proof}
Recall that each $\tau_j$ is a $\pi/2$-cone time for $Z$ with $v_Z(\tau_j) < 0 < \tau_j$ such that the next $\pi/2$-cone time $t> \tau_j$ for $Z$ with $v_Z(t) < 0 < t$ is in the opposite direction from $\tau_j$. Therefore, for each $j\in\BB Z$ there exists $\ep > 0$ such that there are no $\pi/2$-cone times $t$ for $Z$ in $[\tau_j , \tau_j + \ep]$ with $0 \in [v_Z(t) , t]$. Hence we can a.s.\ find a random open interval $A_j$ with rational endpoints such that $\tau_j$ is the maximal $\pi/2$-cone time $t$ for $Z$ in $A_j$ with $0 \in [v_Z(t) , t]$. For $n\in\BB N$, let $i_j^n$ be the maximal element $i$ of $\mcl I_n^L$ (if $j$ is even) or $\mcl I_n^R$ (if $j$ is odd) in $A_j$ with $ 0 \in [\phi^n(i) , i]_{\BB Z}$. Let $t_j^n := n^{-1} (i_j^n-1)$. 
By condition~\ref{item-cone-limit-maximal} in Theorem~\ref{thm-cone-limit}, we a.s.\ have
\eqb \label{eqn-tau-limits}
t_j^n \rta \tau_j \quad \op{and} \quad v_Z(t_j^n) \rta v_Z(\tau_j) \quad \forall j\in \BB N. 
\eqe 
Furthermore, the $\pi/2$-cone times $t_j^n$ and $\tau_j$ are in the same direction for sufficiently large $n$. 

For $j\in\BB Z$, let $\psi^n(j)$ be the largest $j' \in \BB N$ for which $\tau_{j'}^n \leq t_{j+1}^n$ and $\tau_{j'}^n$ is not in the same direction as $t_j^n$. By Lemma~\ref{prop-discrete-inner}, for large enough $n$, $\iota_{\psi^n(j) }^n$ is the largest $i \in \BB N$ such that $X_i^n = \tb F$, $i \leq i_{j+1}^n$, $\phi^n(i)  < 0 $, and $X_{\phi^n(i)}^n \not= X_{\phi^n(i_{j+1}^n)}^n$. 

We claim that a.s.\
\eqb \label{eqn-tau-psi-conv}
\lim_{n\rta\infty} \tau_{\psi^n(j)}^n  =  \tau_j .
\eqe
By~\eqref{eqn-tau-limits} and our characterization of $\iota_{\psi^n(j)}^n$, we have $ \tau_{\psi^n(j)}^n \geq    t_j^n $ for sufficiently large $n$. By compactness, from any sequence of positive integers tending to $\infty$, we can extract a subsequence $n_k\rta \infty$ such that $\tau_{\psi_{n_k}(j)}^{n_k} \rta \wt t \in [\tau_j , \tau_{j+1}]$. 
By~\eqref{eqn-tau-limits} and since any two $\pi/2$-cone intervals are either nested or disjoint, $\liminf_{n\rta\infty} ( \tau_{\psi_{n_k}(j)}^{n_k} - v_{Z^{n_k}}(\tau_{\psi_{n_k}(j)}^{n_k}) ) \geq \tau_j - v_Z(\tau_j) > 0$. Hence condition~\ref{item-cone-limit-times} in Theorem~\ref{thm-cone-limit} implies that $\wt t$ is a $\pi/2$-cone time for $Z$ in the opposite direction from $ \tau_{j+1}$ with $v_Z(\wt t) \leq v_Z(\tau_j) \leq \tau_j \leq \wt t$. Therefore $\wt t=\tau_j$. 

Next we claim that for each $j\in\BB N$, there a.s.\ exists a (random) $n_* = n_*(j) \in\BB N$ such that for $n\geq n_*$, we have 
\eqb \label{eqn-psi-add}
\psi^n(j)+1 = \psi^n(j+1)\quad \forall n\geq n_* .
\eqe
Suppose by way of contradiction that this is not the case, i.e.\ there exists $j_0 \in \BB Z$ and a sequence $n_k\rta\infty$ such that $\psi_{n_k}(j_0) < \psi_{n_k}(j_0+1) - 1$ for each $k$. For $k\in\BB N$ let $j_{n_k}  := \psi_{n_k}(j_0) + 1$ and $j_{n_k}' := \psi_{n_k}(j_0+1)-1$. Since
\eqbn
\tau_{\psi_{n_k}(j_0)}^{n_k} < \tau_{j_{n_k}}^n  \leq  \tau_{j_{n_k}'}^n  <  \tau_{\psi_{n_k}(j_0+1)}^{n_k}
\eqen
and the two times on the left and right converge to $\tau_{j_0}$ and $\tau_{j_0+1}$, respectively, as $k\rta\infty$, 
we can (by possibly passing to a further subsequence) arrange that $\tau_{j_{n_k} }^n \rta t$ and $\tau_{j_{n_k}' }^n \rta t'$ for some $t , t'\in [\tau_{j_0} , \tau_{j_0+1}]$ with $t\leq t'$. By condition~\ref{item-cone-limit-times} in Theorem~\ref{thm-cone-limit}, $t$ (resp. $t'$) is a $\pi/2$-cone time for $Z$ with $v_Z(t) < 0 < t$ (resp. $v_Z(t') < 0 < t'$), in the opposite direction from $\tau_{j_0}$ (resp. $\tau_{j_0+1}$). 
Since $[v_Z(\tau_{j_0}) , \tau_{j_0}]$ is the outermost $\pi/2$-cone interval for $Z$ containing 0 which is contained in $[v_Z(\tau_{j_0+1} , \tau_{j_0+1}]$ and is in the opposite direction from $\tau_{j_0+1}$, we infer that $t' = \tau_{j_0}$. Since $\tau_{j_0} \leq t\leq t'$ we have $t  = \tau_{j_0 }$. But, $t $ is in the opposite direction from $\tau_{j_0}$, so we obtain a contradiction and conclude that~\eqref{eqn-psi-add} holds.

To conclude the proof of the lemma, we observe that~\eqref{eqn-tau-psi-conv} implies that $b^n = \psi^n(0)$ for large enough $n$. By~\eqref{eqn-psi-add}, for each $j \in \BB Z$, it holds for sufficiently large $n\in\BB N$ that $j + b^n = \psi^n(j)$. Therefore~\eqref{eqn-tau-psi-conv} implies $\tau_{j+b^n}^n\rta \tau_j$ as $n\rta\infty$. 
\end{proof}

Next we prove convergence of the times when the exploration path $\lambda$ finishes tracing a loop. 

\begin{lem} \label{prop-outer-conv} 
For $n\in\BB N$, let $\sigma_j^n := n^{-1} \theta_j^n$ (in the notation of Lemma~\ref{prop-discrete-outer}). Let $b^n$ be as in Lemma~\ref{prop-inner-conv}. Almost surely, for each $j\in\BB Z$, we have $\sigma_{j+b^n}^n\rta \sigma_j$ as $n\rta\infty$. 
\end{lem}
\begin{proof}
Recall that $\sigma_j$ is the smallest $\pi/2$-cone time $t > \tau_j$ for $Z$ such that $0 \in [v_Z(t) , t]$ and $t$ is in the opposite direction from $\tau_j$. By Lemma~\ref{prop-discrete-inner}, an analogous characterization holds for the times $\theta_j^n$. 

Since $\sigma_{j+b^{n }}^{n } \in [v_{Z^{n }}(\tau_{j+1+b^{n}}^{n }) ,\tau_{j+1+b^{n }}^{n }]$ and the endpoints of this interval converge (by Lemma~\ref{prop-inner-conv}), from any sequence of $n$'s tending to $\infty$, we can extract a subsequence $n_k \rta\infty$ such that $\sigma_{j+b^{n_k}}^{n_k}$ converges to some $t \in [0,2]$. By condition~\ref{item-cone-limit-times} in Theorem~\ref{thm-cone-limit}, this time $t$ is a $\pi/2$-cone time for $Z$ in the same direction as $\sigma_j$ with $0\in [v_Z(t) , t]$ and $t\in [ \tau_j,   \sigma_j]$. We must show that in fact $t = \sigma_j$. 

It is clear from the above characterization of $\sigma_j$ that $t\geq  \sigma_j$. On the other hand, we can a.s.\ find $r \in \BB Q \cap (0,2)$ such that $\sigma_j  = \tau^{0 , r}_n$ (defined in condition~\ref{item-cone-limit-stopping} from Theorem~\ref{thm-cone-limit}). Then Theorem~\ref{thm-cone-limit} implies $\tau^{0 , r}_n \rta \sigma_j$ and $v_{Z^n}(\tau^{0 , r}_n) \rta v_Z(\sigma_j)$. Hence for sufficiently large $n\in\BB N$, we have $\tau^{0 , r}_n  \geq \tau_{j+b^n}^n$, $0 \in [v_{Z^n}(\tau^{0 , r}_n)  , \tau^{0 , r}_n ]$, and $\tau^{0 , r}_n $ is in the same direction as $\sigma_j^n$. Therefore $\sigma_j^n \leq \tau^{0, r}_n$ for sufficiently large $n$. Passing to the limit along the subsequence $(n_k)$ we get $t \leq \sigma_j$. 
\end{proof}

Recall the set $\Theta_j^n$ and $I_j^n$ considered in Lemmas~\ref{prop-extra-bubbles-discrete} and~\ref{prop-discrete-components}, respectively, which correspond to excursions of the exploration path outside of the loop $\ell^n$ and bounded complementary connected components of $\ell^n$, respectively. 
Our next definition will be used to isolate the ``macroscopic" $\tb F$-times in $I_j^n$ and $\Theta_j^n$. 

\begin{defn} \label{def-big-max-cones}
For $j\in\BB Z$, let $T_j$ be defined as in the beginning of this subsection and for $n\in\BB N$ let $I_j^n$ be as in Section~\ref{sec-fk-loops}. For $\zeta>0$, let $T_j(\zeta)$ (resp. $I_j^n(\zeta)$) be the set of $t\in T_j$ (resp. $i\in I_j$) with $t-v_Z(t) \geq \zeta $ (resp. $i-\phi^n(i) \geq \zeta n$). 
Also let $\Sigma_j$ be as in Section~\ref{sec-loop-limit} and for $n\in\BB N$ let $\Theta_j^n$ be as in Section~\ref{sec-fk-loops}. For $\zeta>0$, let $\Sigma_j(\zeta)$ (resp. $\Theta_j^n(\zeta)$) be the set of $t\in\Sigma_j$ (resp. $i\in \Theta_j^n$) with $t-v_Z(t) \geq \zeta $ (resp. $i-\phi^n(i) \geq \zeta n$). 
\end{defn}  

Recall that $T_j \cup \Sigma_j$ is the set of maximal $\pi/2$-cone times for $Z$ in $(v_Z(\sigma_j) , \sigma_j)$. In particular, $T_j(\zeta) \cup \Sigma_j(\zeta)$ is a finite set. By Lemmas~\ref{prop-extra-bubbles-discrete} and~\ref{prop-discrete-components}, $I_j^n \cup \Theta_j^n$ is the set of maximal $\tb F$-times for $X^n$ in $(\wt\theta_j^n , \theta_j^n)_{\BB Z}$. 

The following lemma will imply~\eqref{eqn-component-conv} of Theorem~\ref{thm-loop-limit}. 
 
\begin{lem} \label{prop-all-bubble-conv} 
Fix $\zeta>0$ and $j\in\BB Z$. Let $t_1,\dots,t_{\BB m}$ be the elements of $ T_j(\zeta) \cup \Sigma_j(\zeta)$, listed in chronological order. 
For $n\in\BB N$, let $b^n$ be as in Lemma~\ref{prop-inner-conv} and let $i_1^n,\dots,i_{\BB m^n}^n$ be the elements of $I_{j+b^n}^n(\zeta)\cup \Theta_{j+b^n}^n(\zeta)$, listed in chronological order. Almost surely, for sufficiently large $n\in\BB N$ we have $\BB m^n = \BB m$. Furthermore, it is a.s.\ the case that for each $k\in [1,\BB m]_{\BB Z}$, it holds for sufficiently large $n\in\BB N$ that the $\pi/2$-cone times $t_k$ for $Z$ and $n^{-1}  i_k^n $ for $Z^n$ are in the same direction; $i_k^n \in I_j^n(\zeta)$ (resp. $i_k^n \in \Theta_j^n(\zeta)$) for large enough $n$ if and only if $t_k \in T_j(\zeta)$ (resp. $t_k \in \Sigma_j(\zeta)$); and
\begin{align} \label{eqn-component-conv-t}
 n^{-1} i_k^n \rta t_k ,
\quad n^{-1} \phi^n(i_k^n) \rta v_Z(t_k ) ,
\quad n^{-1} \phi^n_*(i_k^n) \rta u_Z(t_k) .
\end{align} 
\end{lem} 
\begin{proof} 
Let $\BB m_* := \lceil 2\zeta^{-1} (\sigma_j -v_Z(\sigma_j) ) \rceil$. Since elements of $T_j(\zeta) \cup \Sigma_j(\zeta)$ correspond to disjoint time intervals contained in $[v_Z(\sigma_j) , \sigma_j]$, we have $\BB m \leq \BB m_*$. 
Using Lemma~\ref{prop-outer-conv} and condition~\ref{item-cone-limit-times} in Theorem~\ref{thm-cone-limit}, we also have $\BB m^n \leq\BB m_*$ for large enough $n\in\BB N$. 
For $k\in [\BB m+1,\BB m_*]_{\BB Z}$ (resp. $k\in [\BB m^n+1,\BB m_*]_{\BB Z}$) let $t_k := t_{\BB m}$ (resp. $i_k^n := i_{\BB m^n}^n$).   

For each $k\in [1, \BB m_*]_{\BB Z}$, we can a.s.\ find an open interval $A_k \subset  (v_Z(\sigma_j) , \sigma_j)$ with rational endpoints and a rational $a_k \in A_k$ such that $t_k$ is the maximal $\pi/2$-cone time $t$ for $Z$ in $A_k$ with $a_k \in [v_Z(t) , t] $. 
For $k\in [1, \BB m]_{\BB Z}$, let $\wt i_k^n$ be the maximal $\tb F$-time for $X^n$ in $n I_k$ with $\phi^n(i) \leq  n a_k \leq i$ and let $\wt t_k^n  = n^{-1} \wt i_k^n$. By condition~\ref{item-cone-limit-maximal} in Theorem~\ref{thm-cone-limit}, we a.s.\ have $\wt t_k^n \rta t_k$. 
 
On the other hand, from any sequence of integers tending to $\infty$ we can extract a subsequence $n_l\rta\infty$ such that $(\wt t_k^{n_l} )$ converges to some $\wh t_k \in [v_Z(\sigma_j ) , \sigma_j]$ for each $k \in [1,  \BB m_*]_{\BB Z}$. By condition~\ref{item-cone-limit-times} in Theorem~\ref{thm-cone-limit}, $\wh t_k$ is a $\pi/2$-cone time for $Z$ with $\wh t_k - v_Z(\wh t_k) \geq \zeta $, and $v_Z(\wh t_k) = \lim_{l\rta\infty} n_l^{-1} \phi^{n_l}(\wh i_k^{n_l})$. 

We claim that $\wh t_k \not= \sigma_j$. Indeed, if this is not the case then $n_l^{-1}(\phi^{n_l}(i^{n_l}_k) - \wt\theta_j^{n_l}) \rta 0$ and $n_l^{-1}(\theta_{j + b^{n_l}}^{n_l} - i^{n_l}_k) \rta 0$ as $l\rta\infty$. This is a contradiction since $\iota_{j + b^{n_l}}^{n_l}$ is a maximal $\tb F$-time in $(\wt\theta_{j + b^{n_l}}^{n_l} , \theta_{j + b^{n_l}}^{n_l})$ (Lemma~\ref{prop-discrete-inner}) and $n_l^{-1}  \theta_{j + b^{n_l}}^{n_l}$ and $n_l^{-1}\iota_{j + b^{n_l}}^{n_l}$ converge a.s.\ to distinct times (Lemmas~\ref{prop-inner-conv} and~\ref{prop-outer-conv}). 

It follows that for each $k\in [1, \BB m]_{\BB Z}$, there is some $\wh k \in [1, \BB m]_{\BB Z}$ such that $[v_Z(\wh t_k) , \wh t_k]\subset   [v_Z(t_{\wh k}) , t_{\wh k}]$. Hence for each given $\ep   > 0$, it holds for sufficiently large $l\in\BB N$ that $t^{n_l}_k \in [v_{Z^{n_l}}(\wt t^{n_l}_{\wh k}) - \ep , \wt t^{n_l}_{\wh k} + \ep]$. By maximality of $i^{n_l}_k$, it is necessarily the case that for sufficiently large $l$, we have $t^{n_l}_k \in [\wt t^{n_l}_{\wh k} , \wt t^{n_l}_{\wh k} + \ep]$. Hence $t^{n_l}_k\rta t_{\wh k}$. 

The times $t^{n_l}_k$ and $t^{n_l}_{k+1}$ differ by at least $\zeta$ for $k\in [1, \BB m^n-1]_{\BB Z}$. Hence the mapping $k\mapsto \wh k$ is increasing on $[1,\BB m ]_{\BB Z}$. In particular this mapping is injective and $\BB m^{n_l} \leq \BB m$ for sufficiently large $l$. 

We next argue that for each $k_* \in [1,\BB m]_{\BB Z}$, there is some $k \in [1,\BB m_*]_{\BB Z}$ for which $\wh k = k_*$. To see this, first observe that a.s.\ $t_{k_*} - v_Z(t_{k_*}) > \zeta$, so it is a.s.\ the case that for each sufficiently large $l \in \BB N$ we have $n_l^{-1} (\wt i_{k_*}^{n_l} - \phi^{n_l}(\wt i_{k_*}^{n_l}))  > \zeta$ and $[\phi^{n_l}(\wt i_{k_*}^{n_l}) , \wt i_{k_*}^{n_l} ]_{\BB Z} \subset (\wt\theta_{j + b^{n_l}}^{n_l} ,\theta_{j + b^{n_l}}^{n_l} )_{\BB Z}$. For such an $l$ we have $[\phi^{n_l}(\wt i_{k_*}^{n_l}) , \wt i_{k_*}^{n_l} ]_{\BB Z} \subset [\phi^{n_l}(i_{k_l}^{n_l}) , i_{k_l}^{n_l}]_{\BB Z}$ for some $k_l \in [1,\BB m^{n_l}]_{\BB Z}$. Upon passing to the scaling limit, we find that there is some $ k\in [1,\BB m_*]_{\BB Z}$ for which $[v_Z(t_{k_*}) , t_{k_*}] \subset [v_Z(t_{\wh k}) , t_{\wh k}]$ which (by the argument above) implies $\wh k = k_*$. 

It follows that the mapping $k\mapsto \wh k$ is an increasing bijection from $[1,\BB m^{n_l}]_{\BB Z}$ to $[1,\BB m]_{\BB Z}$ for sufficiently large $l$, which implies that in fact $\BB m^{n_l} =\BB m$ for sufficiently large $l$ and $t_k^{n_l} \rta t_k$ for each $k\in [1,\BB m]_{\BB Z}$. Since our initial choice of sequence was arbitrary, we infer that $\BB m^n = \BB m$ for sufficiently large $n$ and $t_k^n \rta t_k$ for each $k\in [1,\BB m]_{\BB Z}$. 

By condition~\ref{item-cone-limit-times} in Theorem~\ref{thm-cone-limit}, it is a.s.\ the case that for each $k\in [1,\BB m]_{\BB Z}$, it holds for sufficiently large $n\in\BB N$ that the $\pi/2$-cone times $t_k$ for $Z$ and the $\pi/2$-cone times $t_k^n$ and $t_k^n + n^{-1}$ for $Z^n$ are in the same direction. Furthermore, $ n^{-1} \phi^n(i_k^n) \rta v_Z(t)$ and $ n^{-1} \phi^n_*(i_k^n) \rta u_Z(t)$. Hence~\eqref{eqn-component-conv-t} holds. 
  
By definition, we have $t_k \in T_j(\zeta)$ if and only if $u_Z(t_k) < \wt\sigma_j$. By Lemma~\ref{prop-extra-bubbles-discrete}, we have $i_k^n\in I_{j + b^{n }}^n(\zeta)$ if and only if $\phi^n_*(i_k^n) < \wt\theta_j^n$. Hence~\eqref{eqn-component-conv-t} implies that $i_k^n \in I_{j + b^{n }}^n(\zeta)$ (resp. $i_k^n \in \Theta_{j + b^{n}}^n(\zeta)$) for large enough $n$ if and only if $t_k \in T_j(\zeta)$ (resp. $t_k \in \Sigma_j(\zeta)$)
\end{proof}

Our next lemma will be used for the proof of~\eqref{eqn-full-conv} of Theorem~\ref{thm-loop-limit}. We prove a slightly more general statement than we need here, since the proof is no more difficult and the more general statement will be used in~\cite{gwynne-miller-cle}.

\begin{lem} \label{prop-bubble-sum}
For $n\in\BB N$, let $b^n$ be as in Lemma~\ref{prop-inner-conv}. Also fix $j\in\BB N$. The following is true almost surely. Let $\wt a ,a \in [v_Z(\sigma_j) , \sigma_j]_{\BB Z}$ be two times with $\wt a  < a$. Then  
\begin{align} \label{eqn-bubble-sum}
&n^{-1} \sum_{i\in I_{j+b^n}^n \cap (\wt a n  , a  n)_{\BB Z}} (i-\phi^n(i)) \rta \sum_{t\in T_j \cap (\wt a , a)} (t-v_Z(t)) \notag\\
&n^{-1} \sum_{i\in \Theta_{j+b^n}^n \cap \mcl I_n^L \cap (\wt a n  , a  n)_{\BB Z}} (i-\phi^n(i)) \rta \sum_{t\in \Sigma_j \cap \mcl T^L \cap (\wt a , a)} (t-v_Z(t))   \\
&n^{-1} \sum_{i\in \Theta_{j+b^n}^n \cap \mcl I_n^R \cap (\wt a n  , a  n)_{\BB Z}} (i-\phi^n(i)) \rta \sum_{t\in \Sigma_j \cap \mcl T^R \cap (\wt a , a)} (t-v_Z(t))   \notag
\end{align} 
where here $\mcl T^L$ (resp. $\mcl T^R$) denotes the set of left (resp. right) $\pi/2$-cone times for $Z$ (Definition~\ref{def-cone-time}).
\end{lem}
\begin{proof}
Almost surely, Lebesgue-a.e. $s\in [v_Z(\sigma_j) ,\sigma_j] $ belongs to $(v_Z(t) , t)$ for some $t\in T_j \cup \Sigma_j$. Hence for each $\ep  > 0$, there a.s.\ exists $\zeta >0$ such that 
\eqbn
\sum_{t\in  T_j(\zeta ) \cup \Sigma_j(\zeta)  } (t-v_Z(t)) \geq \sigma_j - v_Z(\sigma_j) -\ep ,
\eqen
so since intervals $[v_Z(t) , t] $ for distinct $t \in T_j \cup \Sigma_j$ are disjoint,
\eqb \label{eqn-bubble-small}
\sum_{t\in  (T_j \cup \Sigma_j) \setminus ( T_j(\zeta ) \cup \Sigma_j(\zeta) )  } (t-v_Z(t)) \leq \ep . 
\eqe 
By Lemma~\ref{prop-all-bubble-conv}, it is a.s.\ the case that for large enough $n\in\BB N$, we have
\eqbn
n^{-1}\sum_{i\in (I_{j+b^n}^n(\zeta) \cup \Theta_{j+b^n}^n(\zeta)) } (i-\phi^n(i)) \geq n^{-1} \theta_{j+b^n}^n -n^{-1} \wt\theta^n_{j+b^n} - 2\ep ,
\eqen
so since intervals $[\phi^n(i) , i]_{\BB Z}$ for distinct $i \in I_{j+b^n}^n \cup \Theta_{j+b^n}^n$ are disjoint, 
\eqb \label{eqn-bubble-sum-small-n}
n^{-1}\sum_{i\in (I_{j+b^n}^n \cup \Theta_{j+b^n}^n)\setminus  (I_{j+b^n}^n(\zeta) \cup \Theta_{j+b^n}^n(\zeta)) } (i-\phi^n(i))  \leq 2\ep .
\eqe 
By Lemma~\ref{prop-all-bubble-conv}, is is a.s.\ the case that for each $\wt a , a$ as in the statement of the lemma,  
\begin{align} \label{eqn-bubble-sum-zeta}
&n^{-1} \sum_{i\in I_{j+b^n}^n(\zeta) \cap (\wt a n  , a  n)_{\BB Z}} (i-\phi^n(i)) \rta \sum_{t\in T_j(\zeta)  \cap (\wt a , a)} (t-v_Z(t)) \notag\\
&n^{-1} \sum_{i\in \Theta_{j+b^n}^n(\zeta)  \cap \mcl I_n^L \cap (\wt a n  , a  n)_{\BB Z}} (i-\phi^n(i)) \rta \sum_{t\in \Sigma_j(\zeta)  \cap \mcl T^L \cap (\wt a , a)} (t-v_Z(t)) \\
&n^{-1} \sum_{i\in \Theta_{j+b^n}^n(\zeta)  \cap \mcl I_n^R \cap (\wt a n  , a  n)_{\BB Z}} (i-\phi^n(i)) \rta \sum_{t\in \Sigma_j(\zeta)  \cap \mcl T^R \cap (\wt a , a)} (t-v_Z(t)) .\notag 
\end{align} 
Since $\ep $ is arbitrary, we can now conclude by combining~\eqref{eqn-bubble-small},~\eqref{eqn-bubble-sum-small-n}, and~\eqref{eqn-bubble-sum-zeta}. 
\end{proof}

\begin{proof}[Proof of Theorem~\ref{thm-loop-limit}]
For $n\in\BB N$, let $ (M^n , e_0^n , S^n) $ be the infinite-volume FK planar map corresponding to $X^n$ under Sheffield's bijection. 
The convergence~\eqref{eqn-component-conv} follows from Lemma~\ref{prop-discrete-components} and Lemma~\ref{prop-all-bubble-conv}. 

To obtain~\eqref{eqn-full-conv}, recall the formula for $\op{Area}  \left(M_{j+b^n}^{n,\infty} \right) $ from Lemma~\ref{prop-discrete-outer}. By Lemma~\ref{prop-outer-conv} we a.s.\ have $n^{-1} (\theta_{j+b^n}^n -\wt\theta_{j+b^n}^n) \rta \sigma_j - v_Z(\sigma_j)$. By Lemma~\ref{prop-bubble-sum} we a.s.\ have $n^{-1} \sum_{i\in I_{j+b^n}^n } (i-\phi^n(i)) \rta \sum_{t\in T_j } (t-v_Z(t))$. Almost surely, Lebesgue-a.e.\ point of $[v_Z(\sigma_j) , \sigma_j]$ is contained in $(v_Z(t) , t)$ for some $t\in T_j\cup \Sigma_j$, so since these intervals are disjoint for different values of $t$,
\eqbn
 \sigma_j - v_Z(\sigma_j) - \sum_{t\in T_j } (t-v_Z(t)) = \sum_{t\in \Sigma_j } (t-v_Z(t)) .
\eqen 
Thus~\eqref{eqn-full-conv} holds a.s. 

To obtain~\eqref{eqn-in-conv}, we note that $\Sigma_j^{\op{in}} = \Sigma_j\cap \mcl T^L$ or $\Sigma_j\cap \mcl T^R$, depending on the direction of the $\pi/2$-cone time $\sigma_j$. Furthermore, by Lemma~\ref{prop-discrete-components}, if $j$ is odd we have
\[
 M_j^{n,\op{in}} = M_j^{n,\infty}\setminus \bigcup_{i\in \Theta_j \cap \mcl I_n^R} [\phi^n(i) , i-1]_{\BB Z}
 \]
and we have a similar formula if $j$ is even. We then conclude using a similar argument as in the proof of~\eqref{eqn-full-conv}. 
\end{proof}

\section{Probabilistic estimates}
\label{sec-prob-estimates}

Now that we have seen why Theorem~\ref{thm-cone-limit} implies our scaling limit result for FK loops (Theorem~\ref{thm-loop-limit}), we turn our attention to the proof of Theorem~\ref{thm-cone-limit}.  
In this section we will prove a variety of probabilistic estimates for the inventory accumulation model of~\cite{shef-burger}. In Section~\ref{sec-bm-prelim}, we will prove some estimates for Brownian motion, mostly using results from \cite{shimura-cone}, and make sense of the notion of a Brownian motion conditioned to stay in the first quadrant. In Section~\ref{sec-prob-lower}, we will use our estimates for Brownian motion to prove lower bounds for various rare events associated with the word $X$. In Section~\ref{sec-F-prob}, we will prove an upper bound for the number of $\tb F$-symbols in the reduced word $X(1,n)$, which is a sharper version of \cite[Lemma 3.7]{shef-burger}.

Throughout this section, we let $p \in (0,1/2)$ and $\kappa \in (4,8)$ be related as in~\eqref{eqn-p-kappa}. Many of the estimates in this section will involve the exponents
\eqb \label{eqn-cone-exponent}
\mu := \frac{\pi}{2\left( \pi - \arctan \frac{\sqrt{1-2p} }{p} \right) }  = \frac{\kappa}{8} ,\quad \mu' := \frac{\pi}{2\left( \pi + \arctan \frac{\sqrt{1-2p} }{p} \right) }   = \frac{\kappa }{4(\kappa -2)} .
\eqe

\subsection{Brownian motion lemmas}
\label{sec-bm-prelim}
 
In \cite[Theorem 2]{shimura-cone}, the author constructs for each $\theta\in (0,2\pi)$ a probability measure on the space of continuous functions $[0,1] \rta \BB R^2$ which can be viewed as the law of a standard two-dimensional Brownian motion started from 0 conditioned to stay in the cone $\{z\in\BB C \,:\, 0 \leq \op{arg} z  \leq \theta\}$ until time 1. We want to define a Brownian motion started from 0 with variances and covariances as in~\eqref{eqn-bm-cov}, conditioned to stay in the first quadrant. To this end, we define
\eqb \label{eqn-bm-matrix}
A := \sqrt{\frac{2(1-p)}{1-2p }} \left(  \begin{array}{cc}
1 &  - \frac{p}{1-p}   \\ 
0  &  \frac{\sqrt{1-2p}}{1-p}
\end{array}        \right) ,
\eqe
so that if $Z$ is as in~\eqref{eqn-bm-cov}, then $A Z$ is a standard planar Brownian motion. A Brownian motion with variances and covariances as in~\eqref{eqn-bm-cov} conditioned to stay in the first quadrant until time 1 is the process $\wh Z := A^{-1} \wh Z'$, where $\wh Z'$ is a standard linear Brownian motion conditioned to stay in the cone
\eqb \label{eqn-F-cone}
F_p  := \left\{w\in \BB C  \,:\, 0 < \op{arg} w < \pi - \arctan\frac{\sqrt{1-2p} }{p} \right\}  
\eqe
for one unit of time. By \cite[Equation 3.2]{shimura-cone} and Brownian scaling, the law of $\wh Z(t)$ for $t\in (0,1]$ is absolutely continuous with respect to Lebesgue measure on $(0,\infty)^2$ and its density is given by
\eqb \label{eqn-f-density-def}
  \frac{\op{det} A }{ 2^\mu \Gamma(\mu)  t^{1/2 + 2\mu}  } |Az|^{2\mu} e^{-|Az|^2/2t} \sin\left(2\mu \op{arg} (Az)\right) \BB P_z\left( T > 1-t \right)  \, dz ,
\eqe  
where here $\BB P_z$ denotes the law of $Z$ started from $z$ and $T$ is the first exit time of $Z$ from the first quadrant.
Note that our $\mu$ is equal to $1/2$ times the exponent $\mu$ of \cite{shimura-cone}. 
 
The law of the process $\wh Z$ is uniquely characterized as follows lemma, which is an analogue of~\cite[Theorem 3.1]{sphere-constructions}.

\begin{lem} \label{prop-bm-meander}
Let $\wh Z = (\wh U , \wh V)  : [0,1]\rta\BB R^2$ be sampled from the conditional law of $Z|_{[0,1]}$ given that it stays in the first quadrant. Then $\wh Z$ is a.s.\ continuous and satisfies the following conditions.
\begin{enumerate}
\item For each $t \in (0,1]$, a.s.\ $\wh U(t)>0$ and $\wh V(t)>0$. \label{item-bm-meander-pos}
\item For each $\zeta \in (0,1)$, the regular conditional law of $\wh Z|_{[ \zeta, 1]}$ given $\wh Z|_{[0,\zeta]}$ is that of a Brownian motion with covariances as in~\eqref{eqn-bm-cov}, starting from $\wh Z(\zeta)$, parametrized by $[\zeta,1]$, and conditioned on the (a.s.\ positive probability) event that it stays in the first quadrant.  \label{item-bm-meander-markov}
\end{enumerate}
If $\wt Z  = (\wt U , \wt V)  : [0,1]\rta\BB R^2$ is another random a.s.\ continuous path satisfying the above two conditions, then $\wt Z\eqD \wh Z$.
\end{lem}
\begin{proof}
First we verify that $\wh Z$ satisfies the above two conditions. It is clear from the form of the density~\eqref{eqn-f-density-def} that condition~\ref{item-bm-meander-pos} holds. 
To verify condition~\ref{item-bm-meander-markov}, fix $\zeta >0$. By~\cite[Theorem 2]{shimura-cone}, $\wh Z$ is the limit in law in the uniform topology as $\delta \rta 0$ of the law of $Z|_{[0,1]}$ conditioned on the event $E_\delta$ that $U(t) \geq -\delta$ and $V(t) \geq -\delta$ for each $t\in [0,1]$. By the Markov property, for each $\zeta > 0$, the conditional law of $ Z|_{[\zeta,1]}$ given $Z|_{[0,\zeta]}$ and $E_\delta$ is that of a Brownian motion with covariances as in~\eqref{eqn-bm-cov}, starting from $ Z(\zeta)$, parametrized by $[\zeta,1]$, and conditioned to stay in the $\delta$-neighborhood of the first quadrant. As $\delta \rta 0$, this law converges to the law described in condition~\ref{item-bm-meander-markov}. 

Now suppose that $\wt Z = (\wt U , \wt V)  : [0,1]\rta\BB R^2$ is another random continuous path satisfying the above two conditions. For $\zeta>0$, let $\wt Z^\zeta : [0,1]\rta \BB R^2$ be the random continuous path such that $\wt Z^\zeta(t) = \wt Z(t+\zeta)$ for $t\in [0, 1-\zeta ]$; and conditioned on $\wt Z  |_{[0,1 ]}$, $\wt Z^\zeta$ evolves as a Brownian motion with variances and covariances as in~\eqref{eqn-bm-cov} started from $\wt Z(1)$ and conditioned to stay in the first quadrant for $t \in [1-\zeta,1]$. By condition~\ref{item-bm-meander-markov} for $\wh Z$ and \cite[Theorem 2]{shimura-cone}, we can find $\ep \in (0,\alpha/2)$ such that the Prokhorov distance (in the uniform topology) between the conditional law of $\wt Z^\zeta$ given any realization of $\wt Z|_{[0,\zeta]}$ for which $|\wt Z(\zeta)| \leq \ep$ is at most $\alpha/2$. By continuity, we can find $\zeta_0  > 0$ such that for $\zeta \in (0,\zeta_0]$, we have $\BB P\left(\sup_{t\in [0,\zeta]} |\wt Z(t)| \geq \alpha/2\right) \leq \alpha/2$.
Hence for $\zeta \in  (0,\zeta_0]$ the Prokhorov distance between the law of $\wt Z^\zeta$ and the law of $\wh Z$ is at most $\alpha$. Since $\alpha$ is arbitrary we obtain $\wt Z^\zeta \rta \wh Z$ in law. By continuity, $\wt Z^\zeta$ converges to $\wt Z$ in law as $\zeta\rta 0$. Hence $\wt Z \eqD \wh Z$. 
\end{proof}

We record an estimate for the probability that $Z$ has an approximate $\pi/2$-cone time or an approximate $3\pi/2$-cone time, which is essentially a consequence of the results of~\cite{shimura-cone}.

\begin{lem} \label{prop-bm-cone-asymp}
Let $Z = (U,V)$ be as in~\eqref{eqn-bm-cov} and let $\mu$ and $\mu'$ be as in~\eqref{eqn-cone-exponent}. For $  \delta > 0$ and $C > 1$, let 
\alb
E_\delta &:=  \left\{\text{$\inf_{t \in [0,1]} U(t) \geq -\delta^{1/2}$ and $\inf_{t \in [0,1]} V(t) \geq -\delta^{1/2}$ }\right\} \\
E_\delta' &:= \left\{\text{$ U(t) \geq -\delta^{1/2}$ or $  V(t) \geq -\delta^{1/2}$ for each $t\in [0,1]$ }\right\} \\
G(C) &:= \left\{\sup_{t \in [0,1]} |Z(t)| \leq C \right\} \cap \left\{\text{$U(1) \geq  C^{-1} $ and $V(1) \geq  C^{-1} $}\right\} .
\ale
For each $C>1$ we have
\eqb \label{eqn-bm-cone-asymp}
 \BB P\left(E_\delta\cap G(C) \right)  \asymp \BB P\left( E_\delta\right)  \asymp \delta^{  \mu} 
\eqe 
and
\eqb \label{eqn-bm-cone-asymp'}
\BB P\left(E_\delta' \cap G(C) \right)  \asymp \BB P\left(E_\delta'\right)  \asymp \delta^{   \mu'}
\eqe 
with the implicit constants independent of $\delta$. 
\end{lem}
\begin{proof}
Let $A$ be as in~\eqref{eqn-bm-matrix}, so that $\wt Z  = (\wt U , \wt V):= A Z$ is a standard two-dimensional Brownian motion. Note that $A$ maps the first quadrant to the cone $F_p$ defined in~\eqref{eqn-F-cone} and the complement of the third quadrant to the cone
\eqb \label{eqn-F'-cone}
F_p'  := \left\{w\in \BB C  \,:\,   \op{arg} w \notin \left[\pi, 2\pi - \arctan\frac{\sqrt{1-2p} }{p} \right] \right\}   .
\eqe
Let $F_p^\delta$ be the $\delta^{1/2}$-neighborhood of $F_p$ and let $z := \exp\left(\frac{i}{2} \left(\pi - \arctan \frac{\sqrt{1-2p} }{p}  \right)   \right)$ be the unit vector pointing into $F_p$. We have
\eqbn
\{\wt Z([0,1]) \subset F_p^{c_1\delta} \} \subset E_\delta \subset   \{\wt Z([0,1]) \subset F_p^{c_2\delta} \}
\eqen 
for positive constants $c_1$ and $c_2$ depending only on $A$.
By scale invariance of Brownian motion, we have
\eqbn
\delta^{ \mu}     \BB P\left( \wt Z([0,1]) \subset F_p^\delta \right)
= \delta^{ \mu}     \BB P\left(   \wt Z([0,\delta^{-1} ])  +  z \subset F_p   \right)  .
\eqen
By \cite[Equation 4.3]{shimura-cone} 
this quantity converges to a finite positive constant as $\delta \rta 0$.  
We therefore obtain
$\BB P\left( E_\delta \right) \asymp \delta^{ \mu}$.
Similarly, $\BB P\left( E_\delta' \right) \asymp \delta^{ \mu'}$.
This proves the second proportions in~\eqref{eqn-bm-cone-asymp} and~\eqref{eqn-bm-cone-asymp'}. By \cite[Theorem 2]{shimura-cone}, the conditional law of $ \wt Z|_{[0,1]}$ given $\{ \wt Z([0,1]) \subset F_p^\delta\}$ converges in the uniform topology as $\delta\rta 0$ to the law $\wh{\BB P}$ of a continuous path $\wh Z : [0,1]\rta \BB C$ satisfying (with $G(C)$ as in the statement of the lemma)
\eqbn
\wh{\BB P}\left( G(C)  \right) > 0 \quad \forall C >1 , \quad \op{and} \quad  \lim_{C\rta\infty} \wh{\BB P}\left( G(C)  \right) = 1 . 
\eqen
By combining this observation with our argument above, we obtain the first proportionality in~\eqref{eqn-bm-cone-asymp}. We similarly obtain the first proportionality in~\eqref{eqn-bm-cone-asymp'}.
\end{proof}

\subsection{Lower bounds for various probabilities}
\label{sec-prob-lower}

In this section we will prove lower bounds for the probabilities of various rare events associated with the word $X$. This will be accomplished by breaking up a segment of the word $X$ of length $n$ into sub-words of length approximately $\delta^k n$ for $\delta$ small but independent from $n$; then estimating the probabilities of events for each sub-word using \cite[Theorem 2.5]{shef-burger} and Lemma~\ref{prop-bm-cone-asymp}. 
We start with a lower bound for the probability that a word of length $n$ contains either no burgers or no orders (plus some regularity conditions).

\begin{lem} \label{prop-X-asymp} 
Let $\mu$ be as in~\eqref{eqn-cone-exponent}. For $n\in\BB N$ and $C >1$, let $R_n(C)$ be the event that the following is true.
\begin{enumerate}
\item $X(-n,-1) $ contains no burgers.
\item $X(-n,-1)$ contains at least $C^{-1} n^{1/2}$ hamburger orders, at least $C^{-1} n^{1/2}$ cheeseburger orders, and at most $C n^{1/2}$ total orders. 
\end{enumerate}
Also let $R_n^*(C)$ be the event that the following is true. 
\begin{enumerate}
\item $X(1,n) $ contains no orders.
\item $X(1,n)$ contains at least $C^{-1} n^{1/2}$ burgers of each type and at most $C n^{1/2}$ total burgers.
\end{enumerate}
If $C > 4$, then  
\eqb \label{eqn-X-burger-asymp}
\BB P\left(  R_n(C) \right) \geq n^{-\mu + o_n(1)} 
\eqe 
and
\eqb \label{eqn-X-order-asymp}
\BB P\left( R_n^*(C)  \right) \geq n^{-\mu + o_n(1)} .
\eqe 
\end{lem}

In terms of the walk $D = (d ,d^*)$ defined in Section~\ref{sec-burger-prelim}, the event $R_n(C)$ of Proposition~\ref{prop-X-asymp} is the same as the event that the time reversal of $(D - D(-1))|_{[-n,-1]_{\BB Z}}$ stays in the first quadrant for $n$ units of time and ends up at distance of order $n^{1/2}$ away from the boundary of the first quadrant. The event $R_n^*(C)$ is equivalent to a similar condition for the walk $D|_{[1,n]_{\BB Z}}$. Hence the estimates of Lemma~\ref{prop-X-asymp} are natural in light of Lemma~\ref{prop-bm-cone-asymp} and the scaling limit result for $D$ (Theorem~\ref{prop-burger-limit}). 

\begin{remark}
We will prove a sharper version of the estimate~\eqref{eqn-X-burger-asymp} later, which also includes an upper bound (see Proposition~\ref{prop-J-reg-var} below). 
\end{remark}


\begin{proof}[Proof of Lemma~\ref{prop-X-asymp}]
We will prove~\eqref{eqn-X-burger-asymp}. The estimate~\eqref{eqn-X-order-asymp} is proven similarly, but with the word $X$ read in the forward rather than the reverse direction. 

Fix $C > 4$. Also fix $\delta < 1/4C^2$ to be chosen later independently of $n$. Let
\eqb \label{eqn-delta-floor-def}
\BB k_n := \left\lceil\frac{\log n}{\log \delta^{-1}} \right\rceil 
\eqe 
be the smallest integer $k$ such that $\delta^k n \leq 1$. 
Also fix a deterministic sequence $\xi = (\xi_j)_{j\in\BB N}$ with $\xi_j = o_j(\sqrt j)$ and $\xi_j \leq j^{1/2}$ (to be chosen later, independently of $n$) and for $k\in [1,  \BB k_n]_{\BB Z}$ let $E_{n,k}$ be the event that the following is true. 
\begin{enumerate}
\item $X(  - \lfloor \delta^{k-1} n \rfloor , -\lfloor \delta^k n \rfloor -1 )$ has at most $ 0\vee ( C^{-1} (\delta^k n)^{1/2} -1 )$ burgers of each type.
\item $  C^{-1}(\delta^{k-1} n)^{1/2} \leq \mcl N_\theta \left( X(- \lfloor \delta^{k-1} n \rfloor , -\lfloor \delta^k n \rfloor -1) \right)   \leq C (\delta^{k-1} n)^{1/2} $ for $\theta \in \{\tb H , \tb C\}$.  
\item $\mcl N_{\tb F}\left( X( - \lfloor \delta^{k-1} n \rfloor , -\lfloor \delta^k n \rfloor -1 ) \right) \leq \xi_{\lfloor \delta^{k-1} n \rfloor}$. 
\end{enumerate}
On $\bigcap_{k=1}^{\BB k_n  } E_{n,k}$, the word $X(-n,-1)$ contains no burgers (since each burger in $X(  - \lfloor \delta^{k-1} n \rfloor , -\lfloor \delta^k n \rfloor -1 )$ is cancelled by an order in $X(  - \lfloor \delta^{k} n \rfloor , -\lfloor \delta^{k+1} n \rfloor -1 )$) and at most
\[
2(C+1)n^{1/2} \sum_{k=1}^\infty \delta^{\frac{k-1}{2}} \leq (4C+4) n^{1/2}
\]
total orders. Furthermore, since $X(-n,- \lfloor \delta n \rfloor)$ contains at least $C^{-1} n^{1/2}$ hamburger orders and at least the same number of cheeseburger orders, so does $X(-n,-1)$. 
Consequently,  
\eqb \label{eqn-no-burgers-contain}
\bigcap_{k=1}^{\BB k_n } E_{n,k} \subset R_{n} (4 C+4)   .
\eqe   
The events $E_{n,k}$ for $k\in [1,  \BB k_n]_{\BB Z}$ are independent, so to obtain~\eqref{eqn-X-burger-asymp} (with $4C$ in place of $C$) we just need to prove a suitable lower bound for $\BB P(E_{n,k})$. 
We will do this using Lemma~\ref{prop-bm-cone-asymp} and the scaling limit result for the walk $D = (d,d^*)$ from Definition~\ref{def-theta-count}. 

We first define an event in terms of this walk. In particular, we let $\wt E_{n,k}$ be the event that the following is true.
\begin{enumerate}
\item $\inf_{j \in [ \lfloor \delta^k n\rfloor +1  , \lfloor \delta^{k-1} n \rfloor ]_{\BB Z} } ( d( -j) - d(-\lfloor \delta^k n \rfloor - 1) ) \geq  - \left( 0\vee ( C^{-1} (\delta^k n)^{1/2} -1  -  \xi_{\lfloor \delta^{k-1} n \rfloor}  )   \right) $ and similarly with $d^*$ in place of $d$.
\item $   C^{-1}(\delta^{k-1} n)^{1/2} +  \xi_{\lfloor \delta^{k-1} n \rfloor} \leq  d(-\lfloor \delta^{k-1} n \rfloor )- d(-\lfloor \delta^k n \rfloor - 1)  \leq  C (\delta^{k-1} n)^{1/2} -  \xi_{\lfloor \delta^{k-1} n \rfloor}$ and similarly with $d^*$ in place of $d$.
\item $\mcl N_{\tb F}\left( X( - \lfloor \delta^{k-1} n \rfloor , -\lfloor \delta^k n \rfloor -1 ) \right) \leq \xi_{\lfloor \delta^{k-1} n \rfloor }$. 
\end{enumerate}
The running infimum of $j\mapsto  d(X(-j,-1))$ up to time $m\in\BB N$ is equal to $-\mcl N_{\tc H}(X(-m,-1))$. A similar statement holds for $d^*$. 
From this, we infer that $\wt E_{n,k} \subset E_{n,k}$. 
By \cite[Lemma 3.7]{shef-burger}, we can choose the sequence $\xi$ in such a way that it holds with probability tending to 1 as $m\rta\infty$ that $X(1,m)$ has at most $\xi_{m}$ flexible orders. By \cite[Theorem 2.5]{shef-burger}, as $n\rta \infty$ ($k$ and $\delta$ fixed), the probability of the event $\wt E_{n,k}$ converges to the probability of the event that $ Z$ stays within the $C^{-1} \delta^{1/2}$-neighborhood of the first quadrant in the time interval $[0,1-\delta]$ and satisfies $C^{-1}  \leq -U(1) \leq C$ and $C^{-1} \leq -V(1) \leq C$. By~\eqref{eqn-bm-cone-asymp} of Lemma~\ref{prop-bm-cone-asymp} this latter event has probability $\succeq \delta^\mu$ with the implicit constant independent of $\delta$. Hence we can find $b\in (0,1)$, independent of $\delta$, and $ m_* = m_*(\delta , C ,\xi)$ such that whenever $\lfloor \delta^k n \rfloor \geq m_*$, we have $\BB P(\wt E_{n,k}  ) \geq b \delta^{ \mu }$.

Let $k_*$ be the largest $k\in [1, \BB k_n]_{\BB Z}$ for which $\lfloor \delta^k n \rfloor \geq m_*$. Then 
\eqbn
\BB P\left(\bigcap_{k=1}^{k_*} E_{n,k} \right) \geq b^{k_*} \delta^{k_* \mu} \geq b^{\BB k_n} \delta^{\BB k_n \mu} \geq n^{-\mu    + o_\delta(1) } ,
\eqen 
with the $o_\delta(1)$ independent of $n$. Since $\lfloor \delta^{k_*+1} n \rfloor  \leq m_*$, the event $\bigcap_{k=k_*+1}^{\BB k_n} E_{n,k} $ is determined by the word $X_{-m_* }\dots X_{-1}$, $\BB P\left(\bigcap_{k=k_*+1}^{\BB k_n} E_{n,k} \right) $ is at least a positive constant which does not depend on $n$. We infer from~\eqref{eqn-no-burgers-contain} that
\eqbn
\BB P\left( R_{n} (4C)  \right) \succeq n^{-\mu   + o_\delta(1) } ,
\eqen
with the implicit constant depending on $\delta$, but not $n$. Since $\delta$ is arbitrary, this implies~\eqref{eqn-X-burger-asymp}. 
\end{proof}

From Lemma~\ref{prop-X-asymp}, we obtain the following.

\begin{prop}\label{prop-infinite-F}
Almost surely, there are infinitely many $i \in \BB N$ for which $X(1,i)$ contains no burgers; infinitely many $j\in\BB N$ for which $X(-j,-1)$ contains no orders; and infinitely many $\tb F$-symbols in $X(1,\infty)$. 
\end{prop}
\begin{proof}
For $m\in\BB N$, let $K_m$ be the $m$th smallest $i\in\BB N$ for which $X(1,i)$ contains no burgers (or $K_m = \infty$ if there are fewer than $m$ such $i$). Observe that $K_m$ can equivalently be described as the smallest $i \geq K_{m-1}+1$ for which $X(K_{m-1}+1 ,i)$ contains no burgers. Hence the words $X_{K_{m-1}+1} \dots X_{K_m}$ are iid. It follows that $\{K_m\}_{m\in\BB N}$ is a renewal process. Note that $i\in\BB N$ is equal to one of the times $K_m$ if and only if the word $X(1,i)$ contains no burgers. By Lemma~\ref{prop-X-asymp}, we thus have
\eqbn
\sum_{i=1}^\infty \BB P\left(\text{$i = K_m$ for some $m\in\BB N$}\right) \geq \sum_{i=1}^\infty i^{-\mu + o_i(1)} = \infty
\eqen
since $\mu < 1$. By elementary renewal theory, 
$K_1$ is a.s.\ finite, whence there are a.s.\ infinitely many $i\in\BB N$ for which $X(1,i)$ contains no burgers.  
We similarly deduce from~\eqref{eqn-X-order-asymp} that there are a.s.\ infinitely many $j\in\BB N$ for which $X(-j,-1)$ contains no orders. To obtain the last statement, we note that for each $m\in\BB N$, we have $\BB P\left(X_{K_m+1} = \tb F\right) = p/2$, so there are a.s.\ infinitely many $m\in\BB N$ for which $X_{K_m+1} = \tb F$. For each such $m$, an $\tb F$ symbol is added to the order stack at time $K_{m}+1$. 
\end{proof}

Next we consider an analogue of Lemma~\ref{prop-X-asymp} which involves $3\pi/2$-cone times instead of $\pi/2$-cone times.

\begin{lem}\label{prop-X-asymp'}
For $n\in\BB N$ and $C>4$, let $R_n'(C) $ be the event that the following is true.
\begin{enumerate}
\item $X(1,i)$ contains a burger for each $i\in [1,  n]_{\BB Z}$.\label{item-R'-burgers}
\item $X(1,n)$ contains at least $C^{-1} n^{1/2}$ hamburger orders and at least $C^{-1} n^{1/2}$ cheeseburger orders.\label{item-R'-orders}
\item $|X(1,n)| \leq C n^{1/2}$. \label{item-R'-length}
\end{enumerate}
 Also let $(R_n')^*(C)$ be the event that the following is true.
\begin{enumerate}
\item $X(-j,-1)$ contains either a hamburger order or a cheeseburger order for each $j\in [1, n]_{\BB Z}$.
\item $X(-n,-1)$ contains at least $C^{-1} n^{1/2}$ burgers of each type and at most $C n^{1/2}$ total burgers.
\item $|X(-n,-1)| \leq C n^{1/2}$. 
\end{enumerate}
For $C>4$ we have
\eqb \label{eqn-X-burger-asymp'}
\BB P\left(R_n'(C) \right) \geq n^{-\mu' + o_n(1)}
\eqe 
and
\eqb \label{eqn-X-order-asymp'}
\BB P\left((R_n')^*(C) \right) \geq n^{-\mu' + o_n(1)}
\eqe 
with $\mu'$ as in~\eqref{eqn-cone-exponent}.
\end{lem}

In terms of the walk $D = (d, d^*)$, the event $R_n'(C)$ of Lemma~\ref{prop-X-asymp'} says that the coordinates $d$ and $d^*$ do not attain a simultaneous running infimum on the time interval $[1,n]_{\BB Z}$ and that $D$ does not come close to staying in the first quadrant during this time interval or get too far away from 0 during this time interval. The event $(R_n')^*(C)$ has a similar interpretation in terms of the time reversal of $D|_{[-n,-1]_{\BB Z}}$. 

 \begin{proof}[Proof of Lemma~\ref{prop-X-asymp'}]
We will prove~\eqref{eqn-X-burger-asymp'}. The estimate~\eqref{eqn-X-order-asymp'} is proven similarly, but with the word $X$ read in the reverse, rather than the forward, direction. The proof is similar to that of Lemma~\ref{prop-X-asymp}: we break the word $X$ into increments of length approximately $\delta^k n$ and estimate the probability of an event corresponding to each segment using~\cite[Theorem 2.5]{shef-burger} and Lemma~\ref{prop-bm-cone-asymp}. 

Fix $C>4$, $\delta \in (0, (8C)^{-2}]$, and a deterministic sequence $\xi = (\xi_j)_{j\in\BB N}$ with $\xi_j = o_j(\sqrt j)$ to be chosen later independently of $n$. We assume $\xi_j \leq \delta j^{1/2}$ for each $j\in\BB N$. Let $\BB k_n$ be as in~\eqref{eqn-delta-floor-def}. For $k\in [1, \BB k_n]_{\BB Z}$, let $ E_{n,k}'$ be the event that the following is true.  
\begin{enumerate}
\item For each $i\in [\lfloor \delta^k n \rfloor + 1 , \lfloor \delta^{k-1} n \rfloor  ]_{\BB Z}$, at least one of the following three conditions holds: $\mcl N_{\tb H} \left( X(\lfloor \delta^k n \rfloor + 1 , i) \right) \leq 0\vee \left( C^{-1} (\delta^k n)^{1/2} - \xi_{\lfloor \delta^{k-1} n \rfloor  } \right)$; $\mcl N_{\tb C} \left( X(\lfloor \delta^k n \rfloor + 1 , i) \right) \leq 0\vee \left( C^{-1} (\delta^k n)^{1/2} - \xi_{\lfloor \delta^{k-1} n \rfloor }\right)$; or $  X(\lfloor \delta^k n \rfloor+ 1 , i)  $ contains a burger.  \label{item-nonempty-conds} 
\item $  \mcl N_\theta \left( X( \lfloor \delta^k n \rfloor + 1 , \lfloor \delta^{k-1} n \rfloor ) \right)   \geq C^{-1} (\delta^{k-1} n)^{1/2} $ for $\theta \in \{\tc H , \tc C\}$.  \label{item-nonempty-burgers} 
\item $\mcl N_{\theta} \left(X(\lfloor \delta^k n \rfloor +1 , \lfloor \delta^{k-1} n \rfloor  )\right) \geq C^{-1} (\delta^{k-1} n)^{1/2} - \xi_{\lfloor \delta^{k-1} n \rfloor }$ for $\theta \in \{\tb H, \tb C\}$. \label{item-nonempty-orders}
\item $|X( \lfloor \delta^k n \rfloor +1 , \lfloor \delta^{k-1} n \rfloor )| \leq C (\delta^{k-1} n)^{1/2}$. \label{item-nonempty-length}
\item $\mcl N_{\tb F}\left( X(\lfloor \delta^k n \rfloor +1 , \lfloor \delta^{k-1} n \rfloor ) \right) \leq \xi_{ \lfloor \delta^k n \rfloor }$.  \label{item-nonempty-F}
\end{enumerate}
We claim that
\eqb  \label{eqn-nonempty-contain}
\bigcap_{k=1}^{\BB k_n } E_{n,k}' \subset R_{n} '(8C)  .
\eqe  
First we observe that conditions~\ref{item-nonempty-conds}, \ref{item-nonempty-burgers}, and~\ref{item-nonempty-F} in the definition of $E_{n,k}'$ imply that condition~\ref{item-R'-burgers} in the definition of $R_n'(8C)$ holds on $\bigcap_{k=1}^{\BB k_n } E_{n,k}' $. From condition~\ref{item-nonempty-orders} and~\ref{item-nonempty-length} in the definition of $E_{n,k}'$, we infer that on $\bigcap_{k=1}^{\BB k_n } E_{n,k}'$, we have for $\theta\in \{\tb H, \tb C\}$ that
\alb
\mcl N_\theta \left( X(1,n ) \right) &\geq   C^{-1} n^{1/2} - \xi_{n} - C n^{1/2} \sum_{k=2}^{\BB k_n}  \delta^{(k-1 )/2}\\
& \geq  \frac12 C^{-1} n^{1/2} -  2\delta^{1/2} C n^{1/2} \geq \frac18 C^{-1} n^{1/2}
\ale
where the last inequality is by our choice of $\delta$. Thus condition~\ref{item-R'-orders} in the definition of $R_n'(8C)$ holds. Finally, it is clear from condition~\ref{item-nonempty-length} in the definition of $E_{n,k}'$ that condition~\ref{item-R'-length} in the definition of $R_n'(8C)$ holds on $\bigcap_{k=1}^{\BB k_n } E_{n,k}' $. This completes the proof of~\eqref{eqn-nonempty-contain}. 

The events $E_{n,k}'$ for $k\in [1,  \BB k_n]_{\BB Z}$ are independent, so in light of~\eqref{eqn-nonempty-contain}, to obtain~\eqref{eqn-X-burger-asymp'} (with $8C$ in place of $C$) we just need to prove a suitable lower bound for $\BB P(E_{n,k}')$. 
To this end, for $k\in [1, \BB k_n]_{\BB Z}$ let $\wt E_{n,k}'$ be the event that the following is true. 
\begin{enumerate}
\item For each $i \in[ \lfloor \delta^k n \rfloor +1 , \lfloor \delta^{k-1} n \rfloor ]_{\BB Z}$, either $ d( i) - d( \lfloor \delta^k n \rfloor + 1) \geq 0\wedge \left(- C^{-1} (\delta^k n)^{1/2} +\xi_{\lfloor \delta^{k-1} n \rfloor   } \right)$ or $ d^*( i) - d^*( \lfloor \delta^k n \rfloor  + 1) \geq 0\wedge \left( -C^{-1} (\delta^k n)^{1/2} +\xi_{\lfloor \delta^{k-1} n \rfloor  } \right)$.  \label{item-nonempty-conds'}
\item  $d( \lfloor \delta^{k-1} n \rfloor   )  - d( \lfloor \delta^k n \rfloor  +1)   $ and $d^*(\lfloor \delta^{k-1} n \rfloor   )  - d( \lfloor \delta^k n \rfloor +1) $ are each at least $C^{-1} (\delta^{k-1} n)^{1/2}$. \label{item-nonempty-burgers'}
\item $\inf_{ i \in[ \lfloor \delta^k n \rfloor  +1 , \lfloor \delta^{k-1} n \rfloor  ]_{\BB Z}} \left( d(i) - d(\lfloor \delta^k n \rfloor +1)\right) \leq - C^{-1} (\delta^{k-1} n)^{1/2} - \xi_{\lfloor \delta^{k-1} n \rfloor }   $ and similarly with $d^*$ in place of $d$. \label{item-nonempty-orders'}
\item $ \sup_{i \in[\lfloor \delta^k n \rfloor  +1 , \lfloor \delta^{k-1} n \rfloor  ]_{\BB Z}} |D(i)|    \leq (C/2) (\delta^{k-1} n)^{1/2} - \xi_{ \lfloor \delta^{k-1} n \rfloor }  $. \label{item-nonempty-length'}
\item $\mcl N_{\tb F}\left( X( \lfloor \delta^k n \rfloor  + 1 , \lfloor \delta^{k-1} n \rfloor  ) \right) \leq \xi_{\lfloor \delta^{k-1} n \rfloor }$.  \label{item-nonempty-F'}
\end{enumerate}
We claim that $\wt E_{n,k}' \subset E_{n,k}'$. It is clear that conditions~\ref{item-nonempty-burgers'}, and~\ref{item-nonempty-F'} in the definition of $\wt E_{n,k}'$ imply the corresponding conditions in the definition of $E_{n,k}'$. Since the running infima of $i\mapsto d(X(1,i))$ and $i\mapsto d^*(X(1,i))$ up to time $m$ differ from $\mcl N_{\tb H}\left(X(1,m)\right)$ and $\mcl N_{\tb C}\left(X(1,m)\right)$, respectively, by at most $\mcl N_{\tb F}\left(X(1,n)\right)$, we find that conditions~\ref{item-nonempty-orders'} and~\ref{item-nonempty-length'} imply the corresponding conditions in the definition of $E_{n,k}'$. 

Suppose condition~\ref{item-nonempty-conds'} in the definition of $\wt E_{n,k}'$ holds. If $i \in [ \lfloor \delta^k n \rfloor  +1  ,  \lfloor \delta^{k-1} n \rfloor    ]_{\BB Z}$ and $X(\lfloor \delta^k n \rfloor   + 1 , i)$ contains no burgers, then the condition $ d( i) - d( \lfloor \delta^k n \rfloor  + 1) \geq 0\wedge \left(- C^{-1} (\delta^k n)^{1/2} +\xi_{\lfloor \delta^{k-1} n \rfloor  } \right)$ together with condition~\ref{item-nonempty-F'} in the definition of $\wt E_{n,k}'$ implies $\mcl N_{\tb H} \left( X(\lfloor \delta^k n \rfloor  + 1 , i) \right) \leq 0\vee \left( C^{-1} (\delta^k n)^{1/2} - \xi_{  \lfloor \delta^{k-1} n \rfloor } \right)$. A similar statement holds if $ d^*( i) - d^*( \lfloor \delta^k n \rfloor  + 1) \geq 0\wedge \left( -C^{-1} (\delta^k n)^{1/2} +\xi_{\lfloor \delta^{k-1} n \rfloor  } \right)$. This proves our claim. 

It now follows from~ \cite[Theorem 2.5 and Lemma 3.7]{shef-burger} together with~\eqref{eqn-bm-cone-asymp'} of Lemma~\ref{prop-bm-cone-asymp} (c.f.\ the proof of Lemma~\ref{prop-X-asymp}) that if $\xi$ is chosen appropriately (independently of $n$) then there is a constant $b \in (0,1)$, independent of $n$ and $\delta$, and a constant $m_* = m_*(\delta , C ,\xi )$ such that whenever $\lfloor \delta^k n \rfloor  \geq m_*$, we have $\BB P(E_{n,k}' ) \geq b \delta^{ \mu'  }$. We conclude exactly as in the proof of Lemma~\ref{prop-X-asymp}. 
\end{proof}

\subsection{Estimate for the number of flexible orders}
\label{sec-F-prob}

The main goal of this section is to prove the following more quantitative version of \cite[Lemma 3.7]{shef-burger} (which says that the number of $\tb F$'s in $X(1,n)$ is $o_n(n^{1/2})$ with high probability), which will turn out to be a relatively straightforward consequence of Lemma~\ref{prop-X-asymp'}.

\begin{lem} \label{prop-few-F}
Let $\mu'$ be as in~\eqref{eqn-cone-exponent}. For each $n\in\BB N$ and each $\nu > \mu'$ we have
\eqb \label{eqn-few-F}
\BB P\left(\text{$\exists i \geq n$ with $\mcl N_{\tb F}\left(X(1,i)\right) \geq i^\nu$}   \right)    = o_n^\infty(n)       
\eqe   
(recall notation~\ref{def-o-notation}). The same holds if we fix $C>1$, let $n' \in [n, Cn]_{\BB Z}$, and condition on the event $\{\text{$X(1,n')$ has no burgers}\}$, in which case the $o_n^\infty(n)$ depends on $C$ but not the particular choice of $n'$. 
\end{lem}
 
Since $\mu' \in (1/3,1/2)$ for each $p \in (0,1/2)$, we have in particular that~\eqref{eqn-few-F} holds for some $\nu  < 1/2$. In other words, with high probability the number of flexible orders in $X(1,i)$ is of strictly smaller polynomial order than the length of $X(1,i)$, for each $i \geq n$.

\begin{remark} \label{remark-few-F-stronger}
The exponent $\mu'$ in Lemma~\ref{prop-few-F} is not optimal. We will show in Corollary~\ref{prop-few-F-optimal} below that $\mu'$ can be replaced by $1-\mu \leq \mu'$. However, the proof of Corollary~\ref{prop-few-F-optimal} indirectly uses Lemma~\ref{prop-few-F}. 
\end{remark}

We will extract Lemma~\ref{prop-few-F} from the following general fact about renewal processes, which will also be used in the proof of the stronger version of Lemma~\ref{prop-few-F} mentioned in Remark~\ref{remark-few-F-stronger}.

\begin{lem} \label{prop-few-renewal}
Let $(Y_j)$ be a sequence of iid positive integer valued random variables and for $m\in\BB N$ let $S_m := \sum_{j=1}^m Y_j$. For $i\in\BB N$, let $E_i$ be the event that $i = S_m$ for some $m\in\BB N$ and for $n\in\BB N$, let $M_n := \sup\{m \,:\, S_m \leq n\}$ be the number of $i\leq n$ for which $E_i$ occurs. Suppose that for some $\alpha > 0$, either 
\eqb \label{eqn-few-renewal-upper}
\BB P(E_i)  \leq i^{-\alpha + o_i(1)} ,\qquad \forall \, i\in\BB N   
\eqe
or
\eqb
\BB P(Y_1	\geq n)  \geq n^{-(1-\alpha) +o_n(1)} , \qquad \forall \, n\in\BB N.  \label{eqn-few-renewal-lower}
\eqe
Then for each $\nu  > 1-\alpha$, 
\eqb \label{eqn-few-renewal}
\BB P\left(\text{$\exists i \geq n$ with $M_i \geq i^\nu$}   \right)    = o_n^\infty(n)      .  
\eqe
\end{lem}

We will prove Lemma~\ref{prop-few-renewal} by obtaining a moment bound for the quantities $M_n$. This, in turn, will be proven using the following recursive relation between the probabilities of the events $E_i$. 

\begin{lem} \label{prop-renewal-relation}
Suppose we are in the setting of Lemma~\ref{prop-few-renewal}. Suppose given integers $0 = i_0 < i_1 < \dots <  i_n  $. Then 
\eqb \label{eqn-E_i-split}
\BB P\left(\bigcap_{k=1}^n E_{i_k}\right) = \prod_{k=1}^n \BB P(E_{i_k - i_{k-1}}) .
\eqe
\end{lem}
\begin{proof}
Let $i' > i$ and let $K_i$ be the smallest $m\in\BB N$ for which $S_m \geq i$. Then $E_i =\{S_{K_i} = i\}$ so by the strong Markov property, 
\eqbn
\BB P\left(E_{i'} \,|\, Y_1 , \dots , Y_{K_i}  \right) \BB 1_{E_i} = \BB P\left(  \text{$i' - i = S_m - S_{M_n}$ for some $m > M_n$} \right)\BB 1_{E_i} =  \BB P\left( E_{i'-i} \right)\BB 1_{E_i}  .
\eqen
Hence, in the setting of~\eqref{eqn-E_i-split} we have
\[
\BB P\left(\bigcap_{k=1}^n E_{i_k} \,|\, \bigcap_{k=1}^{n-1} E_{i_k} \right) = \BB P\left(E_{i_n - i_{n-1}}\right) ,
\]
so 
\[
\BB P\left(\bigcap_{k=1}^n E_{i_k} \right)  = \BB P\left(E_{i_n-i_{n-1} }\right)  \BB P\left( \bigcap_{k=1}^{n-1} E_{i_k} \right) .
\]
We can now obtain~\eqref{eqn-E_i-split} by induction on $n$.  
\end{proof}

Now we can prove a $k$th moment bound for $M_n$ by induction on $k$. 
 
\begin{lem} \label{prop-renewal-moment}
Suppose we are in the setting of Lemma~\ref{prop-few-renewal}. Then for $k\in\BB N$ we have  
\eqb \label{eqn-renewal-moment}
\BB E\left(M_n^k \right) \leq  n^{k(1-\alpha)  + o_n(1)}    .
\eqe 
\end{lem}
\begin{proof}
First consider the case $k=1$. If the hypothesis~\eqref{eqn-few-renewal-upper} holds, then 
\eqbn
\BB E\left(M_n \right) = \sum_{i=1}^n \BB P\left(E_i\right) \leq  \sum_{i=1}^n i^{-\alpha + o_i(1)} = n^{1-\alpha + o_n(1) }. 
\eqen
Alternatively, if~\eqref{eqn-few-renewal-lower} holds, then for $m\in\BB N$, 
\alb
\BB P\left(M_n \geq m \right)  = \BB P\left(S_m \leq n \right) \leq  \BB P\left(\max_{j \in [1, m]_{\BB Z} } Y_j \leq n \right) = \BB P\left(Y_1 \leq n\right)^m .
\ale
By~\eqref{eqn-few-renewal-lower} we have
\eqbn
\BB P\left(Y_1 \leq n\right)^m \leq \left(1 - n^{-(1-\alpha) + o_n(1)} \right)^m .
\eqen
Hence
\eqbn
\BB E\left(M_n \right) = \sum_{m=1}^n \BB P\left(M_n \geq m  \right) \leq \sum_{m=1}^n\left(1 - n^{-(1-\alpha) + o_n(1)} \right)^m \leq n^{1-\alpha + o_n(1)} .
\eqen
This proves~\eqref{eqn-renewal-moment} for $k=1$. 

Now consider the case $k>1$.  
By Lemma~\ref{prop-renewal-relation},
\begin{align} \label{eqn-k-moment-expand}
\BB E\left(M_n^k \right)  
&\preceq   \sum_{i=1}^n \sum_{i \leq  j_1 ,  \dots , j_{k-1} \leq n} \BB P\left( E_i\cap E_{j_1} \cap\dots \cap E_{j_{k-1}} \right)\notag \\ 
&\preceq  \sum_{i=1}^n \BB P(E_i) + \sum_{i=1}^n \sum_{m=1}^{k-1}  \sum_{i <  j_1 < \dots < j_m \leq n} \BB P\left( E_i\cap E_{j_1} \cap\dots \cap E_{j_m} \right)\notag \\
&=  \sum_{i=1}^n \BB P(E_i) +    \sum_{i=1}^n \BB P(E_i) \sum_{m=1}^{k-1}  \sum_{i <  j_1 <\dots < j_m \leq n} \BB P\left( E_{j_1-i} \cap\dots \cap E_{j_m -i} \right) \notag\\
&\leq   \sum_{i=1}^n \BB P(E_i) +     \sum_{i=1}^n \BB P(E_i) \sum_{m=1}^{k-1} \sum_{1 \leq  j_1< \dots< j_m \leq n } \BB P\left( E_{j_1 } \cap\dots \cap E_{j_m  } \right) \notag \\
&\leq \BB E(M_n)  \sum_{m=0}^{k-1}  \BB E\left(M_n^m   \right),
\end{align}
with implicit constants depending on $k$, but not $n$. We can now obtain~\eqref{eqn-renewal-moment} by induction on $k$. 
\end{proof}

\begin{proof}[Proof of Lemma~\ref{prop-few-renewal}]
By Lemma~\ref{prop-renewal-moment} and the Chebyshev inequality, for $\nu > 1-\alpha$ and $k\in\BB N$, we have
\eqbn
\BB P\left(M_n \geq i^{ \nu} \right) \leq i^{k(1-\alpha-\nu)  + o_i(1)} .
\eqen
We conclude by applying the union bound. 
\end{proof}

\begin{proof}[Proof of Lemma~\ref{prop-few-F}]
Let $K_0 = 0$ and for $m\in\BB N$, let $K_m$ be the $m$th smallest $i\in\BB N$ such that $X(1,i)$ contains no burgers. The times $K_m - K_{m-1}$ are iid and each has the same law as $K_1$. If $X(1,i)$ contains a burger for each $i\in [1,n]_{\BB Z}$, then $K_1 > n$. By Lemma~\ref{prop-X-asymp'}, we therefore have
\eqbn
\BB P\left(K_1 > n\right)  \geq n^{-\mu' + o_n(1)}. 
\eqen
Each time $i$ at which $\mcl N_{\tb F}\left(X(1,i)\right)$ increases is necessarily one of the times $K_m$. Thus~\eqref{eqn-few-F} follows from Lemma~\ref{prop-few-renewal}. The conditional version of the lemma follows by combining the unconditional version with Lemma~\ref{prop-X-asymp}. 
\end{proof}

\section{Regularity conditioned on no burgers}
\label{sec-F-reg}

\subsection{Statement and overview of the proof}
\label{sec-F-reg-setup}

The goal of this section is to prove a regularity statement for the conditional law of the word $X(1,n)$ given the event that it contains no burgers. It will be convenient to read the word backwards, rather than forward, so we will mostly work with $X(-n,-1)$ instead of $X(1,n)$. 

We will use the following notation. Let $J$ be the smallest $j\in\BB N$ for which $X(-j,-1)$ contains a burger. Note that $\{J > n\} $ is the same as the event that $X(-n,-1)$ contains no burgers, or the event that the walk $D$, run backward from time 0, stays in the first quadrant for $n$ units of time.
Let $\mu'$ be as in Lemma~\ref{prop-few-F} and fix $\nu \in (\mu',1/2)$. Let $F_n$ be the event that $\mcl N_{\tb F} \left(X(-n,-1)\right) \leq n^\nu$, so that by Lemma~\ref{prop-few-F} we have $\BB P(F_n) \geq 1- o_n^\infty(n)$. For $\ep > 0$ and $n\in\BB N$, let $E_n(\ep)$ be the event that $J > n$ and $X(-n,-1)$ contains at least $\ep n^{1/2}$ hamburger orders and at least $\ep n^{1/2}$ cheeseburger orders. Let 
\eqb \label{eqn-a-def}
a_n(\ep) := \BB P\left( E_n(\ep) \,|\, J  >n   \right) .
\eqe 
The main result of this section is the following.

\begin{prop} \label{prop-F-regularity}
In the above setting, 
\eqb \label{eqn-a-lim}
\lim_{\ep\rta 0} \liminf_{n\rta\infty} a_n(\ep) = 1 .
\eqe 
\end{prop}

Proposition~\ref{prop-F-regularity} is the key input in the proof of Theorem~\ref{thm-no-burger-conv} below, which gives a scaling limit for the path $Z^n$ conditioned on the event $\{J > n\}$. This theorem, in turn, is the key input in the proof of Theorem~\ref{thm-cone-limit}. 
  
We now give a brief overview of the proof of Proposition~\ref{prop-F-regularity}. We will start by reading the word $X$ forward. For $n\in\BB N$, let $K_n$ be the last time $i\leq n$ for which $X(1,i)$ contains no burgers. We will argue (via an argument based on translation invariance of the word $X$) that $X(1,K_n)$ has uniformly positive probability to contain at least $\ep n^{1/2}$ hamburger orders and at least $\ep n^{1/2}$ cheeseburger orders if $\ep$ is chosen sufficiently small. For $m\in\BB N$, the conditional law of $X_1 \dots X_{m}$ given $\{K_n = m+1\}$ is the same as its conditional law given that $X(1,m)$ contains no burgers, which by translation invariance is the same as the law of $X(-m,-1)$ given $\{J>m\}$. This will allow us to extract a (possibly very sparse) sequence $m_j \rta \infty$ for which $\liminf_{j\rta\infty} a_{m_j}(\ep) >0$. This is accomplished in Section~\ref{sec-F-reg-forward}. 
 
In Section~\ref{sec-F-reg-cond}, we will prove a general result which, for $s\in (0,1)$, allows us to compare the conditional law of $Z^n(\cdot) - Z^n(-s)$ given $\{J > n\}$ and a realization of $X_{-\lfloor n s \rfloor} \dots X_{-1}$ to the law of $Z(\cdot) - Z(-s)$ conditioned to stay in a neighborhood of the third quadrant. 

In Section~\ref{sec-F-reg-reverse}, we will use the result of Section~\ref{sec-F-reg-cond} to show that if $a_m(\ep)$ is bounded below for some small $\ep > 0$ and $m$ is very large, then $a_n(\ep)$ is close to 1 for $n\geq m$ such that $m/n$ is of constant order. The intuitive reason why this is the case is that if $\ep$ is very small and $E_m(\ep)$ fails to occur, then it is unlikely that $J > n$; and if $E_m(\ep) \cap \{J>n\}$ occurs, then (by \cite[Theorem 2.5]{shef-burger}) $E_n(\ep)$ is likely to occur for small $\ep$. We will then complete the proof of Proposition~\ref{prop-F-regularity} using an induction argument and the results of Section~\ref{sec-F-reg-forward}. See Figure~\ref{fig-F-regularity} for an illustration of the basic idea of this argument.

\begin{figure}[ht!]
 \begin{center}
\includegraphics{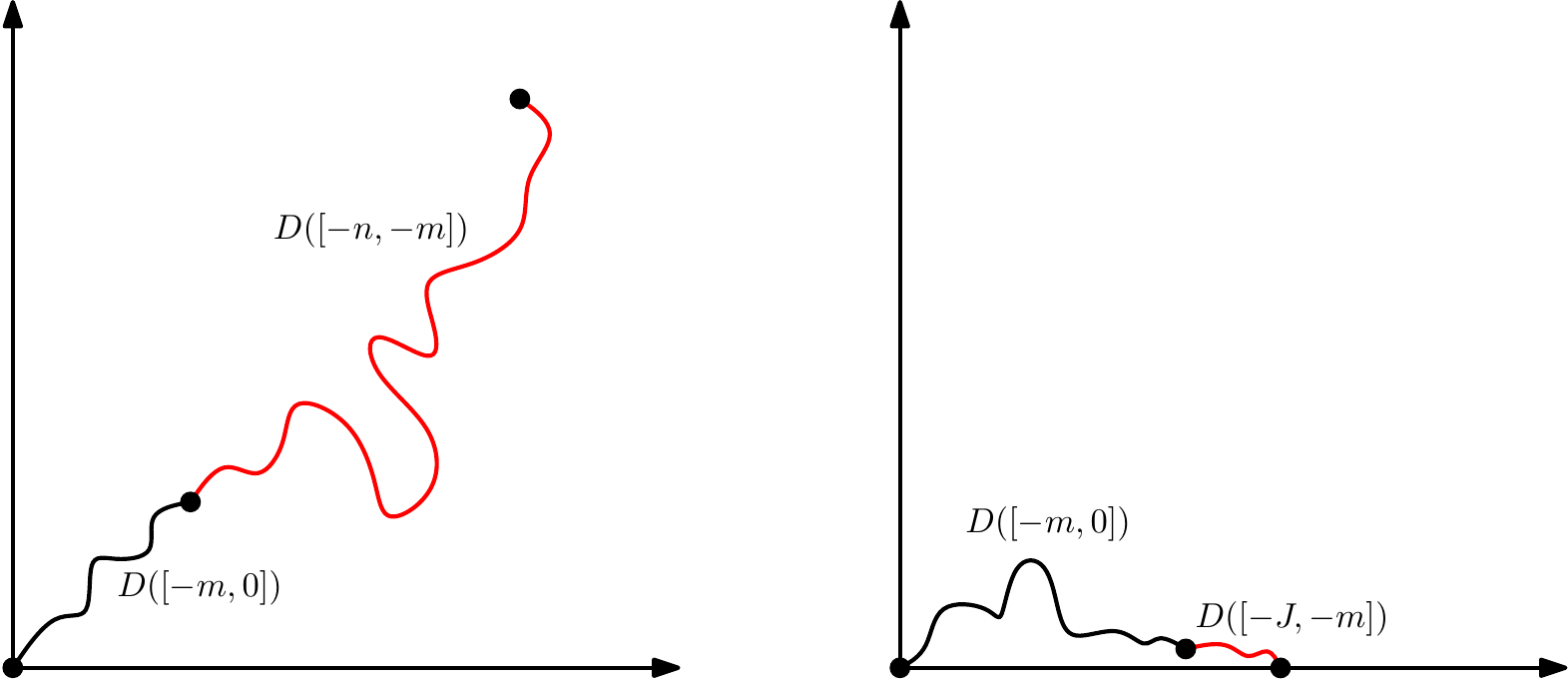} 
\caption{An illustration of the main ideas of the proof of Proposition~\ref{prop-F-regularity}. Fix $\delta > 0$ and suppose $m < n \in\BB N$ with $m \geq \delta n$. Left figure: suppose the event $E_m(\ep)$ occurs, i.e.\ the path $D$ (defined as in~\eqref{eqn-discrete-paths}) is at uniformly positive distance from the boundary of the first quadrant at time $m$. By Lemma~\ref{prop-Z-cond-limit}, if $m$ is very large then it holds with uniformly positive probability that $J>n$ and $E_n(\ep)$ occurs, i.e.\ $D$ stays in the first quadrant until time $n$ and ends up at uniformly positive distance away from the boundary. Right figure: if $E_m(\ep)$ fails to occur and $n$ is very large, then it is unlikely that $J>n$. Hence if we are given an $m$-independent lower bound for $a_m(\ep)$ for some $m\in\BB N$ and $\ep  >0$, then Bayes' rule and an induction argument imply that $a_n(\ep)$ is close to 1 for $n > 2m$, say. We prove the existence of arbitrarily large values of $m$ for which $a_m(\ep)$ is uniformly positive in Section~\ref{sec-F-reg-forward}.} \label{fig-F-regularity}
\end{center}
\end{figure}

\subsection{Regularity along a subsequence} 
\label{sec-F-reg-forward}

In this section we will prove the following result, which is a much weaker version of Proposition~\ref{prop-F-regularity}.  

\begin{lem}  \label{prop-F-regularity-subsequence}
In the notation of~\eqref{eqn-a-def}, there is a $\ep_0 > 0$ and a $q_0 \in (0,1)$ such that for $\ep \in (0,\ep_0]$ there exists a sequence of positive integers $m_j \rta \infty$ (depending on $\ep$) such that for each $j\in\BB N$, 
\eqb \label{eqn-subsequence-regularity} 
a_{m_j}(\ep) \geq  q_0 . 
\eqe 
\end{lem}

For the proof of Lemma~\ref{prop-F-regularity-subsequence}, we first need a result to the effect that the $\tb F$-excursions around 0, i.e.\ the discrete interval $[\phi(i) , i]_{\BB Z}$ containing 0 with $X_i = \tb F$, have uniformly positive probability to have a positive fraction of their length on the left side of 0.

\begin{figure}
 \begin{center}
\includegraphics{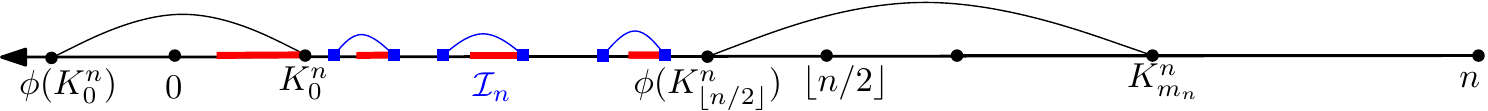} 
\caption{An illustration of the proof of Lemma~\ref{prop-F-symmetry} in the case when the event $Q_n$ occurs and $K_0^n < \phi(K_{\lfloor n/2\rfloor}^n)$.
Each pair of times corresponding to a flexible order and its match are joined by an arc. Intervals belonging to $\mcl I_n$ are shown in blue with square endpoints. Points shown in red are those for which we know $A_i^n(\ep)$ occurs. If we make $\ep > 0$ small enough, the red points occupy a uniformly positive fraction of $[0,n]$ with uniformly positive probability. Since $\BB P(A_i^n(\ep))$ does not depend on $i$, this yields a lower bound for $\BB P(A_i^n(\ep))$. On the other hand, if $Q_n$ occurs and $K_0^n \geq \phi(K_{\lfloor n/2\rfloor}^n)$ then $A_{\lfloor n/2\rfloor}^n(\ep)$ must occur, so we get a lower bound for $\BB P(A_{\lfloor n/2\rfloor}^n(\ep))$. We then get a lower bound for $\BB P(A_n(\ep)) = \BB P(A_0^n(\ep))$ using translation invariance.
} \label{fig-F-symmetry}
\end{center}
\end{figure}

\begin{lem} \label{prop-F-symmetry}
For $n\in\BB N$, let $K_n$ be the largest $i\in[1, n]_{\BB Z}$ for which $X_i = \tb F$ and $\phi(i) \leq 0$ (or $K_n = 0$ if no such $k$ exists). For $\ep \geq 0$, let $A_n(\ep)$ be the event that $K_n \not=0$ and $K_n \leq (1-\ep) (K_n - \phi(K_n))$. There exists $\ep_0 > 0$, $n_0\in\BB N$, and $q_0 \in (0,1/3)$ such that for each $\ep \in (0,\ep_0]$ and $n\geq n_0$, 
\eqbn
\BB P\left(   A_n(\ep) \right)  \geq 3 q_0 .
\eqen
\end{lem} 
\begin{proof}
The idea of the proof is as follow. We look at a carefully chosen collection of disjoint discrete intervals $I = [\phi(j) ,  j]_{\BB Z}$ with $X_j = \tb F$. We will choose these intervals in such a way that for each such interval $I$, the event $A_n(\ep)$ occurs (with $i$ rather than 0 playing the role of the starting point of the word $X$) whenever $i \in I$ with $i \geq \ep(j-\phi(j)) + \phi(j)$ (i.e., for ``most" points of $I$). We then use translation invariance to conclude the statement of the lemma. See Figure~\ref{fig-F-symmetry} for an illustration.

For $n\in\BB N$ and $i\in\BB Z$, let $K_i^n$ be the largest $j\in [i+1, i+n]_{\BB Z}$ for which $X_j = \tb F$ and $\phi(j) \leq i$ (if such a $j$ exists) and otherwise let $K_i^n  = i$. For $\ep \geq 0$, let $A_i^n(\ep)$ be the event that $K_i^n \not=i$ and $K_i^n - i \leq (1- \ep)(K_i^n - \phi(K_i^n))$, so in particular $A_i^n(0) = \{K_i^n \not=i\}$. Note that $A_0^n(\ep) = A_n(\ep)$, and on the event $A_n(0)$ we have $K_n = K_0^n$. By translation invariance,
\eqb \label{eqn-A-translate}
\BB P\left(A_i^n(\ep) \right) = \BB P\left(A_n(\ep)\right) ,\qquad \forall i \in\BB Z,\qquad \forall \ep \geq 0.
\eqe
 
Let $Q_n$ be the event that the following is true (using the re-scaled discrete paths from~\eqref{eqn-Z^n-def}).
\begin{enumerate} 
\item For each $t\in   [1,2]$ we have (in the notation~\eqref{eqn-Z^n-def}) either $U^n(t) \geq U^n(1/2) + 1$ or $V^n(t) \geq V^n(1/2)+1$. \label{item-shift-event1,2}
\item For each $t\in [1/2 , 1]$, either $U^n(t) \geq U^n(0)+1$ or $V^n(t) \geq V^n(0)+1$. \label{item-shift-event1/2,1}
\item The events $A_0^n(0)$ and $A_{\lfloor n/2 \rfloor}^n(0)$ occurs. \label{item-shift-event-F}
\end{enumerate}
 By \cite[Theorem 2.5]{shef-burger} (to deal with the first two conditions) and Proposition~\ref{prop-infinite-F} (to deal with condition~\ref{item-shift-event-F}), there exists $\wt q_0 \in(0,1)$ and $\wt n_0 \in \BB N $ such that for $n\geq \wt n_0$, we have $\BB P\left(Q_n\right) \geq \wt q_0$. We observe that for each $i\in \BB Z$, $n^{-1} (K_i^n-1)$ is a $\pi/2$-cone time for $Z^n$ (Definition~\ref{def-cone-time}) with $v_{Z^n}(n^{-1}(K_i^n-1) ) \leq n^{-1} i$. Consequently, condition~\ref{item-shift-event1,2} in the definition of $Q_n$ implies $K_i^n \leq n$ for each $i\in [1,  n/2]_{\BB Z}$. Similarly, condition~\ref{item-shift-event1/2,1} in the definition of $Q_n$ implies $K_0^n < \lfloor n/2 \rfloor$. 

We claim that on $Q_n$, each $i\in [1,  K_0^n ]_{\BB Z}$ satisfies $K_i^n = K_0^n$. Since $K_i^n \leq n$, it follows from maximality of $K_i^n$ that either $K_i^n = K_0^n$ or $\phi(K_i^n) > 0$. Since two distinct discrete intervals between a $\tb F$ and its match are either nested or disjoint, if $\phi(K_i^n) >0$, then $[\phi(K_i^n) ,   K_i^n]_{\BB Z} \subset (\phi(K_0^n) ,  K_0^n)_{\BB Z}$, which contradicts maximality of $K_i^n$. Therefore we in fact have $K_i^n = K_0^n$. 

We next claim that on $Q_n$, we have $[\phi(K_i^n) , K_i^n]_{\BB Z}\subset  [K_0^n  +1 ,  \phi( K_{\lfloor n/2 \rfloor}^n)-1]_{\BB Z}$ for each $i \in [K_0^n +1 ,  \phi( K_{\lfloor n/2 \rfloor}^n)-1]_{\BB Z}$. 
Indeed, on $Q_n$ both $K_i^n$ and $K_{\lfloor n/2 \rfloor}^n$ are at most $n$, so if $K_i^n > K_{\lfloor n/2 \rfloor}^n$ then since $\phi(K_i^n) \leq i  < \lfloor n/2 \rfloor$, we contradict maximality of $K_{\lfloor n/2 \rfloor}^n$. 
Hence either $K_i^n \in [\phi(K_{\lfloor n/2 \rfloor}^n) , K_{\lfloor n/2 \rfloor}^n]_{\BB Z}$ or $K_i^n \in [i ,  \phi( K_{\lfloor n/2 \rfloor}^n)-1]_{\BB Z}$. The former case is impossible since two distinct discrete intervals between a $\tb F$ and its match are either nested or disjoint, so $K_i^n \in  [i ,  \phi( K_{\lfloor n/2 \rfloor}^n)-1]_{\BB Z}$. 
If $\phi(K_i^n) < 0$, then since $K_0^n + 1 \leq K_i^n \leq n$ we contradict maximality of $K_0^n$. 
Hence we must have $\phi(K_i^n) \geq 0$ so since distinct discrete intervals between a $\tb F$ and its match are either nested or disjoint, we have $\phi(K_i^n) \geq K_0^n +1$. 
 
Let $\mcl I_n$ be the set of maximal $\tb F$-intervals in $[K_0^n  +1 ,  \phi( K_{\lfloor n/2 \rfloor}^n)-1]_{\BB Z}$, i.e.\ the set of discrete intervals $I = [\phi(j) ,   j]_{\BB Z} \subset [K_0^n+1 ,  \phi( K_{\lfloor n/2 \rfloor}^n)-1 ]_{\BB Z}$ with $X_j = \tb F$ which are not contained in any larger such discrete interval. Note that we might have $\phi(K_{\lfloor n/2 \rfloor}^n) < K_0^n$, in which case $\mcl I_n$ is empty. 
For $I = [\phi(j) , j]_{\BB Z} \in \mcl I_n$, we write $|I| = j - \phi(j)$. 
 
We claim that if $Q_n$ occurs and $i \in [\phi(j)  , j -1]_{\BB Z}$ for some $I = [\phi(j) ,  j]_{\BB Z} \in\mcl I_n $, then $K_i^n = j$ (so in particular $A_i^n(0)$ occurs). Indeed, we have $[\phi(K_i^n)  , K_i^n]_{\BB Z}\subset  [K_0^n  +1 ,  \phi( K_{\lfloor n/2 \rfloor}^n)-1]_{\BB Z}$ (by the argument above) and $i\in [\phi(j)   , j - 1]_{\BB Z}$, so the claim follows from maximality of $I$ and of $K_i^n$. Conversely, if $i \in [K_0^n +1  , \phi(K_{\lfloor n/2 \rfloor}^n ) -1]_{\BB Z}$ and $A_i^n(0)$ occurs, then $[\phi(K_i^n) , K_i^n]_{\BB Z} \in \mcl I_n$. Thus $\mcl I_n$ can alternatively be described as the set of discrete intervals $[\phi(K_i^n) , K_i^n]_{\BB Z}$ for $i\in [K_0^n+1 , \phi(K_{\lfloor n/2 \rfloor}^n) ]_{\BB Z}$.
Consequently, if $i\in [K_0^n  +1 ,  \phi( K_{\lfloor n/2 \rfloor}^n)-1]_{\BB Z}$ and $A_i^n(0)$ occurs, then $i \in [\phi(j)  , j -1]_{\BB Z}$ for some $I = [\phi(j) ,  j]_{\BB Z} \in\mcl I_n $. 
By splitting $[1,n/2]_{\BB Z}$ into the three intervals $[0,K_0^n]_{\BB Z}$, $[K_0^n  +1 ,  \phi( K_{\lfloor n/2 \rfloor}^n)-1]_{\BB Z}$, and $[\phi(K_{\lfloor n/2 \rfloor}^n) , n/2]_{\BB Z}$, we obtain 
\eqb \label{eqn-A(0)-sum}
\sum_{i=1}^{\lfloor n/2 \rfloor} \BB 1_{A_i^n(0)} \leq \sum_{I \in \mcl I_n } |I|  + K_0^n    + \lfloor n/2 \rfloor - (\phi(K_{\lfloor n/2 \rfloor}^n) \vee 0 ) .
\eqe 

On the other hand, if $i\in [\phi(j), j-1]_{\BB Z}$ for some $I =[\phi(j),j]_{\BB Z} \in \mcl I_n$ and $i \geq \ep (j-\phi(j)) + \phi(j)$, then since $K_i^n = j$, the event $A_i^n(\ep)$ occurs. As argued above, on $Q_n$ we have $K_i^n = K_0^n$ for each $i\in [1, K_0^n]_{\BB Z}$ so if $i\in [\ep K_0^n ,  K_0^n]_{\BB Z}$ then $A_i^n(\ep)$ occurs.
Therefore, on $Q_n$ we have
\eqb \label{eqn-A(ep)-sum}
\sum_{i=1}^{\lfloor n/2 \rfloor} \BB 1_{A_i^n(\ep)} \geq (1-\ep) \sum_{I \in \mcl I_n } |I|   + (1- \ep)  K_0^n  .
\eqe 
By Proposition~\ref{prop-infinite-F}, $\BB P\left(A_i^n(0) \right) \rta 1$ as $n\rta \infty$ (uniformly in $i$ by translation invariance) so for sufficiently large $n$ (depending only on $\wt q_0$),
\eqbn
\BB E\left( \BB 1_{  Q_n} \sum_{i=1}^{\lfloor n/2 \rfloor} \BB 1_{A_i^n(0)} \right)  = \sum_{i=1}^{\lfloor n/2 \rfloor} \BB P\left(A_i^n(0) \cap  Q_n\right)\geq \left(\BB P(Q_n) - o_n(1) \right)  \lfloor n/2 \rfloor    \geq  \frac{ \wt q_0 }{2}    \lfloor n/2 \rfloor . 
\eqen
By~\eqref{eqn-A(0)-sum},
\eqbn
\BB E\left( \BB 1_{ Q_n} \sum_{I \in \mcl I_n } |I| \right) + \BB E\left(\BB 1_{  Q_n}   K_0^n   \right) +  \BB E\left( \BB 1_{ Q_n}( \lfloor n/2 \rfloor - (\phi(K_{\lfloor n/2 \rfloor}^n) \vee 0 ) ) \right)  \geq  \frac{ \wt q_0 }{2}  \lfloor n/2 \rfloor  . 
\eqen
By~\eqref{eqn-A(ep)-sum},
\eqb \label{eqn-P(A(ep))-sum}
\BB E\left(\BB 1_{ Q_n} \sum_{i=1}^{\lfloor n/2 \rfloor} \BB 1_{A_i^n(\ep)} \right) \geq (1-\ep) \frac{ \wt q_0  }{2} \lfloor n/2 \rfloor - \BB E\left( \BB 1_{  Q_n}( \lfloor n/2 \rfloor - (\phi(K_{\lfloor n/2 \rfloor}^n) \vee 0 ) ) \right)  - \ep \BB E\left(\BB 1_{Q_n} K_0^n\right)   .
\eqe 

On the event $A_{\lfloor n/2 \rfloor}^n(\ep)^c$, we have $ \lfloor n/2 \rfloor - \phi(K_{\lfloor n/2 \rfloor}^n) \leq \ep n$. Therefore,
\alb
\BB E\left( \BB 1_{  Q_n}( \lfloor n/2 \rfloor - (\phi(K_{\lfloor n/2 \rfloor}^n) \vee 0 ) ) \right) &\leq  \lfloor n/2 \rfloor \BB P\left(A_{\lfloor n/2 \rfloor}^n(\ep) \cap   Q_n\right)  +  \ep n \BB P\left(A_{\lfloor n/2 \rfloor}^n(\ep)^c \cap   Q_n  \right)  \\
&\leq  \lfloor n/2 \rfloor \BB P\left(A_{\lfloor n/2 \rfloor}^n(\ep)  \right)  +  \ep n .
\ale
By definition of $K_0^n$, $\BB E\left(\BB 1_{Q_n} K_0^n\right) \leq  n$, so~\eqref{eqn-P(A(ep))-sum} implies that for sufficiently large $n$, 
\alb
\BB E\left(\BB 1_{  Q_n} \sum_{i=1}^{\lfloor n/2 \rfloor} \BB 1_{A_i^n(\ep)} \right)  +   \lfloor n/2 \rfloor \BB P\left(A_{\lfloor n/2 \rfloor}^n(\ep)  \right)      \geq   (1-\ep) \frac{ \wt q_0 }{2} \lfloor n/2 \rfloor  - 2\ep n .
\ale
By~\eqref{eqn-A-translate},  
\eqbn
(1 + \ep ) \lfloor n/2 \rfloor \BB P\left(A_n(\ep)\right) \geq (1- \ep) \frac{\wt q_0}{2}  \lfloor n/2 \rfloor  - 2\ep n .
\eqen
Re-arranging this inequality implies the statement of the lemma for appropriate $\ep_0 >0$, $q_0 \in(0,1/3)$, and $n_0\in\BB N$.  
\end{proof}

From Lemma~\ref{prop-F-symmetry}, we can extract a lower bound for the number of leftover hamburger orders and cheeseburger orders in the word $X(1,K_n)$.

\begin{lem}\label{prop-K-regularity}
Let $K_n$ be defined as in the statement of Lemma~\ref{prop-F-symmetry}. For $\ep  >0$, let $G_n(\ep)$ be the event that $X(1,K_n)$ contains at least $\ep \sqrt{K_n}$ hamburger orders and at least $\ep \sqrt{K_n}$ cheeseburger orders. Let $q_0$ be as in Lemma~\ref{prop-F-symmetry}. There exists $\ep_0 > 0$ and $n_0\in\BB N$ (depending only on $q_0$) such that for $\ep \in (0,\ep_0]$ and $n\geq n_0$, 
\eqbn
\BB P\left(  G_n(\ep) \right)  \geq 2q_0 .
\eqen
\end{lem}
\begin{proof}
The rough idea of the proof is as follows. By Lemma~\ref{prop-F-symmetry}, we know that for small enough $\wt \ep > 0$ we have $\phi(K_n) \leq -\wt \ep K_n$ with uniformly positive probability. By~\cite[Theorem 2.5]{shef-burger} (and since $X_1\dots X_{K_n}$ is independent from $\dots X_{-2} X_{-1}$), the word $X(-\wt \ep K_n , -1)$ is likely to contain at least of order $K_n^{1/2}$ burgers of each type. If this is the case and $X(1,K_n)$ contains too few burgers of either type, then $\phi(K_n)$ would have to be larger than $-\wt \ep K_n$. We now proceed wit the details. 

Let $\wt \ep_0 > 0$ and $\wt n_0 \in \BB N$ be chosen so that the conclusion of Lemma~\ref{prop-F-symmetry} holds (with $\wt \ep_0$ in place of $\ep_0$ and $\wt n_0$ in place of $n_0$). For $n\in\BB N$ let $ A_n(\wt\ep_0)$ be the event of that lemma (with $\ep = \wt\ep_0$). Then for $n\geq \wt n_0$, we have $\BB P\left(A_n(\wt\ep_0) \right) \geq 3q_0$. 

Fix $\alpha \in (0,1)$. Let $F_{K_n}$ be defined as in Section~\ref{sec-F-reg-setup} with $K_n$ in place of $n$ and $X(1,K_n)$ in place of $X(-K_n,-1)$. By Lemma~\ref{prop-few-F}, we can find $m \in \BB N$ such that the probability that there is even one $k\geq m$ such that $X(1,k)$ contains more than $k^\nu$ $\tb F$-symbols is at most $\alpha/2$. By Proposition~\ref{prop-infinite-F}, we can find $n_0' \geq \wt n_0$ such that for $n\geq n_0'$, we have $\BB P\left(K_n \geq m\right) \geq 1-\alpha/2$. For $n\geq n_0'$, we therefore have 
\eqb \label{eqn-K-few-F}
\BB P\left(F_{K_n}  \right) \geq 1 - \alpha .
\eqe 
For $\ep > 0$ and $k\in\BB N$, let $J_k^H(\ep)$ (resp. $J_k^C(\ep)$) be the smallest $j\in\BB N$ for which the word $X(-j,0)$ contains at least $\ep k^{1/2} + k^\nu + 1$ hamburgers (resp. cheeseburgers). By \cite[Theorem 2.5]{shef-burger}, the times $J_k^H(\ep)$ and $J_k^C(\ep)$ are typically of order $\ep^2 k  $. More precisely, we can find $\ep_0  \in (0,\wt\ep_0]$ and $k_0  \in \BB N$ such that for $k\geq k_0$ and $\ep \in (0,\ep_0]$, 
\eqbn
\BB P\left(J_k^H(\ep) \vee J_k^C(\ep) \geq  \wt\ep_0^2 k  \right) \leq \alpha .
\eqen
By Proposition~\ref{prop-infinite-F}, we can find $n_0  \geq n_0'$ such that for $n \geq n_0 $, we have $\BB P\left(K_n \leq   k_0\right) \leq  \alpha$. 

On the event $G_n(\ep)^c \cap F_{K_n}$, we have $-\phi(K_n) \leq J_{K_n}^H(\ep) \vee J_{K_n}^C(\ep)$. Since $G_n(\ep)^c \cap F_n \cap \{K_n \geq k_0\}$ is independent from $\dots X_{-2} X_{-1}$, it follows that for $n\geq n_0'$ we have
\alb
&\BB P\left(G_n(\ep)^c \cap F_{K_n} \cap \{-\phi(K_n) \geq \wt \ep_0^2 K_n    \}    \right) \\
 &\qquad \qquad \leq \BB P\left(K_n \leq   k_0 \right) + \BB E\left( \BB P\left(  G_n(\ep)^c \cap F_{K_n} \cap \{-\phi(K_n) \geq \wt \ep_0^2 K_n    \} \,|\, X_1 X_2 \dots        \right) \BB 1_{(K_n \geq   k_0)} \right) \\
 &\qquad \qquad \leq \alpha + \BB E\left( \BB P\left(J_{K_n}^H(\ep) \vee J_{K_n}^C(\ep) \geq \wt \ep_0^2 K_n \,|\, K_n \right) \BB 1_{(K_n \geq   k_0)} \right) 
 \leq 2\alpha .
\ale
By definition, on the event $A_n(\wt\ep_0)$ we have $-\phi(K_n) \geq \wt\ep_0^2 K_n$, so we have
\eqbn
\BB P\left(  -\phi(K_n) \geq \wt \ep_0^2 K_n  \right) \geq 3q_0 .
\eqen
Therefore,
\eqbn
\BB P\left(G_n(\ep)^c \cap F_{K_n} \right) \leq 1-3q_0 + 2\alpha .
\eqen
By combining this with~\eqref{eqn-K-few-F} we obtain
\eqbn
\BB P\left(G_n(\ep) \right) \geq 3q_0 - 3\alpha  .
\eqen
Since $\alpha$ is arbitrary this implies the statement of the lemma.
\end{proof}

\begin{proof}[Proof of Lemma~\ref{prop-F-regularity-subsequence}]
Let $q_0$ be as in Lemma~\ref{prop-F-symmetry}. For $n\in\BB N$, define the time $K_n$ as in Lemma~\ref{prop-F-symmetry}. Choose $\ep_0 > 0$ and $n_0\in\BB N$ such that the conclusion of Lemma~\ref{prop-K-regularity} holds, and fix $\ep \in (0,\ep_0]$. By Proposition~\ref{prop-infinite-F}, if we are given $j \in \BB N$, we can choose $n \geq n_0$ such that $\BB P\left( j+1 \leq K_n \leq n \right) \geq 1-q_0/2$. Henceforth fix such an $n$. Then with $G_n(\ep)$ as in the statement of Lemma~\ref{prop-K-regularity}, we have
\eqbn
\BB P\left(G_n(\ep ) \cap \{j+1 \leq K_n \leq n\}\right) \geq  \frac32 q_0 .
\eqen
We therefore have
\eqbn
 \frac32 q_0 \leq \sum_{k=j+1}^n \BB P\left(G_n(\ep) \,|\, K_n = k\right) \BB P\left(K_n = k\right) .
\eqen
Hence we can find some $m_j \in [j  , n-1]_{\BB Z}$ for which
\eqbn
\BB P\left(G_n(\ep) \,|\, K_n = m_j+1 \right) \geq  \frac32  q_0 .
\eqen
We can write $\{K_n = m_j+1\}$ as the intersection of the event that $X(1,m_j )$ contains no burgers; and the event that $X_{m_j+1} = \tb F$ and $\mcl N_{\tb F}\left( X(m_j+2,n) \right) = 0$. The latter event is independent of $X_1\dots X_{m_j}$, so the conditional law of $X_1\dots X_{m_j}$ given $\{K_n = m_j+1\}$ is the same as its conditional law given that $X(1,m_j)$ contains no burgers. The event $G_n(\ep) \cap \{K_n = m_j+1\}$ is the same as the event that $K_n = m_j+1$ and $X(1,m_j)$ contains at least $\ep (m_j+1)^{1/2}$ hamburger orders and at least $\ep (m_j+1)^{1/2}$ cheeseburger orders. By Lemma~\ref{prop-few-F} and translation invariance,~\eqref{eqn-subsequence-regularity} holds for this choice of $m_j$ (with a slightly smaller choice of $\ep$) provided $j$ is chosen sufficiently large. Since $m_j \geq j$ and $j \in \BB N$ was arbitrary, we conclude.
\end{proof}

\subsection{Conditioning on an initial segment of the word} 
\label{sec-F-reg-cond}

To state the main result of this subsection, we introduce the following notation for paths corresponding to sub-words of $X$ and their continuum analogues. 

\begin{notation} \label{def-Z-restrict}
For $t_1,t_2\in \BB R$ with $t_1 \leq t_2$, we write
\eqb \label{eqn-Z-restrict}
Z^n_{[t_1,t_2]} := \left(Z^n - Z^n\left(n^{-1}  \lfloor n t_2 -1 \rfloor  \right)\right) |_{\left[t_1, n^{-1}  \lfloor n t_2 -1 \rfloor \right]} \quad \op{and} \quad Z_{[t_1,t_2]}:=(Z  - Z (t_2)) |_{[t_1,t_2]} .
\eqe
We extend the definition of $Z^n_{[t_1,t_2]}$ to $[t_1,t_2]$ be defining it to be identically zero for $t \in \left[n^{-1}  \lfloor n t_2 -1 \rfloor , t_2\right]$. 
\end{notation}
 
The reason why we use $n^{-1} \lfloor n t_2 -1\rfloor$ instead of just $t_2$ in the definition of $Z^n_{[t_1,t_2]}$ is that this choice implies that $Z^n_{[t_1 , t_2]}$ is independent from the future word $X_{\lfloor t_2 n \rfloor} X_{\lfloor t_2 n \rfloor + 1} \dots$. 

In this subsection we will prove a lemma which allows us to estimate the conditional law of $Z^n_{[-1,-s]}$ for $s\in (0,1)$ given $\{J > n\}$ and a realization of $X_{-\lfloor n s \rfloor} \dots X_{-1}$.

\begin{lem} \label{prop-Z-cond-limit}
Fix $\lambda \in (0,1/2)$. For $n\in\BB N$ and $s \in [\lambda , 1-\lambda]$ define
\eqb \label{eqn-h-c}
h_n^s  :=  n^{-1/2} \mcl N_{\tb H}\left( X(-\lfloor sn \rfloor , -1)   \right)  \quad \op{and} \quad c_n^s := n^{-1/2}\mcl N_{\tb C}\left( X(-\lfloor sn \rfloor , -1)   \right)  .
\eqe 
For $\ep_1 , \ep_2  > 0$, let
\eqb \label{eqn-tilde-G^s-def}
\wt G^s(\ep_1 ,\ep_2) := \left\{ \inf_{t\in [-1,-s]} (U_t - U_{-s} ) \geq -s^{1/2} \ep_1 \right\} \cap \left\{ \inf_{t\in [-1,-s]} (V_t - V_{-s}) \geq -s^{1/2} \ep_2 \right\}.   
\eqe 

Suppose given $\ep > 0$ and $\alpha > 0$.
There exists $n_* \in\BB N$ and $\zeta >0$ (depending only on $\lambda$, $\ep$, and $\alpha$) such that the following holds. 
Suppose $n \geq n_*$; $s \in [\lambda , 1-\lambda]$; $\ep_1 , \ep_2 \geq \ep$; and $x$ is a realization of $X_{-\lfloor sn \rfloor} \dots X_{-1}$ for which $X(-\lfloor s n \rfloor , -1)$ contains no burgers, $|h_n^s - \ep_1| \leq \zeta$, $|c_n^s - \ep_2| \leq \zeta$, and $F_{\lfloor sn \rfloor}$ (as defined in Section~\ref{sec-F-reg-setup}) occurs. Then the Prokhorov distance (in the uniform metric) between the conditional law of $Z^n_{[-1,-s]}$ given $\{J > n\} \cap \{X_{-\lfloor s n \rfloor} \dots X_{-1} = x \}$ and the conditional law of $Z_{[-1,-s]}$ given $\wt G^s(\ep_1,\ep_2) $ is at most $\alpha$. Moreover, we can arrange that the same holds if we instead condition on $\{J > n\} \cap \{X_{-\lfloor s n \rfloor} \dots X_{-1} = x \}\cap F_n$. 
\end{lem}
\begin{proof}
Let $\nu$ be as in Section~\ref{sec-F-reg-setup}. For $\ep_1 , \ep_2 > 0$, let $\ul G_{n}^s( \ep_1,\ep_2  )$ be the event that $\mcl N_{\tc H}\left(X(-n , -\lfloor sn \rfloor -1 ) \right) \leq \ep_1 n^{1/2}$ and $\mcl N_{\tc C}\left(X(-n , -\lfloor sn \rfloor -1 ) \right) \leq \ep_2 n^{1/2} $. Let $\ol G_{ n}^s(  \ep_1  ,\ep_2)$ be the event that $\mcl N_{\tc H}\left(X(-n , -\lfloor sn \rfloor-1 ) \right) \leq \ep_1 n^{1/2}  + (sn)^\nu $ and $\mcl N_{\tc C}\left(X(-n , -\lfloor sn \rfloor -1 ) \right) \leq \ep_2 n^{1/2} + (sn)^\nu$.    

Let $x$ be a realization of $X_{-\lfloor sn \rfloor}\dots X_{-1}$ for which $F_{\lfloor sn \rfloor}$ occurs and let $\frk h_n^s$ and $\frk c_n^s$ be the corresponding realizations of $h_n^s$ and $c_n^s$. Then
\begin{align}  \label{eqn-X-cond-G}
& \{X_{-\lfloor sn \rfloor} \dots X_{-1} = x \} \cap   \ul G_{ n}^s(\frk h_n^s , \frk c_n^s )  \subset \{X_{-\lfloor sn \rfloor} \dots X_{-1} = x \} \cap  \{J > n\} \notag\\
 & \qquad \qquad \subset \{X_{-\lfloor sn \rfloor} \dots X_{-1} = x \} \cap   \ol G_{n}^s( \frk h_n^s , \frk c_n^s   ).
\end{align} 
By~\eqref{eqn-X-cond-G} and independence of $X_{-\lfloor sn \rfloor} \dots X_{-1}$ from $Z^n_{[-1,-s]}$, we obtain that for any open subset $\mcl U$ of the space of continuous functions $[-1,-s]\rta \BB R^2$ in the uniform topology,
\begin{align} \label{eqn-Z-cond-G}
\frac{\BB P\left( Z^n_{[-1,-s]} \in \mcl U ,\,     \ul G_{n}^s(   \frk h_n^s , \frk c_n^s   )   \right) }{ \BB P\left(\ol G_n^s(\frk h_n^s , \frk c_n^s)   \right)   }       
 \leq   \BB P\left(Z^n_{[-1,-s]} \in \mcl U    \,|\, J > n,\,  X_{-\lfloor sn \rfloor} \dots X_{-1} =  x  \right)  
 \leq  \frac{\BB P\left( Z^n_{[-1,-s]} \in \mcl U ,\,  \ol G_n^s(  \frk h_n^s , \frk c_n^s  )   \right)  }{ \BB P\left(\ul G_n^s(\frk h_n^s , \frk c_n^s  )  \right)   } .
\end{align} 
Let
\eqbn
r := \inf_{\substack{\ep_1 , \ep_2 \geq \ep \\ s\in [\lambda , 1-\lambda]}} \BB P\left( \wt G^s(\ep_1,\ep_2) \right) .
\eqen
Then $r$ is a positive constant depending only on $\ep$ and $\lambda$. 
We can find $\zeta \in (0,\alpha)$ depending only on $r$ and $\alpha$ such that for $\ep_1 , \ep_2 \geq \ep$ and $s \in [\lambda , 1-\lambda]$,  
\eqb \label{eqn-tilde-G^s-close}
\left| \BB P\left( \wt G^s\left( \ep_1 + \zeta , \ep_2 +  \zeta  \right) \right) - \BB P\left( \wt G^s\left( \ep_1 - \zeta , \ep_2 -  \zeta  \right)\right) \right| \leq  r \alpha  .
\eqe 
By \cite[Theorem 2.5]{shef-burger}, we can find an $n_*   \in \BB N$ depending only on $r$ and $\alpha$ such that for $n\geq n_*$, the Prokhorov distance between the unconditional law of $Z^n|_{[-1,0]}$ and the law of $Z|_{[-1,0]}$ is at most a constant (depending only on $\ep$) times $r \alpha$. By Lemma~\ref{prop-few-F}, by possibly further increasing $n_*$, we can arrange that the same holds with the law of $Z^n|_{[-1,-s]}$ replaced by the conditional law of $Z^n_{[-1,-s]}$ given $F_n$ for each choice of $s\in [\lambda ,1-\lambda]$.
By combining this with our choice of $\zeta$ in~\eqref{eqn-tilde-G^s-close}, we obtain that whenever $n \geq n_*$ and $\ep_1^n , \ep_2^n > 0$ with $|\ep_1^n - \ep_1|$ and $|\ep_2^n - \ep_2|$ each smaller than $\zeta$,  
\eqbn
\left| \BB P\left(   \ol G_n^s(  \ep_1^n , \ep_2^n   )    \right)  -   \BB P\left(  \wt G^s(\ep_1,\ep_2)   \right)     \right| \preceq r \alpha ,
\eqen
with the implicit constant depending only on $\ep$, and similarly with $\ul G_n^s(\ep_1^n ,\ep_2^n)$ in place of $\ol G_n^s(\ep_1^n,\ep_2^n)$. Since $\alpha$ is arbitrary the statement of the lemma now follows from~\eqref{eqn-Z-cond-G}. 
\end{proof}

\subsection{Regularity at all sufficiently large times} 
\label{sec-F-reg-reverse}

In this section we will deduce Proposition~\ref{prop-F-regularity} from Lemma~\ref{prop-F-regularity-subsequence} and an induction argument. See Figure~\ref{fig-F-regularity} for an illustration of the argument. Our first lemma tells us that if $n > m$ with $n\asymp m$ and we condition on the event $\{J  > n\}$ and on a ``good" realization of $X_{-m}\dots X_{-1}$ (i.e.\ one for which $E_m(\ep)$ occurs), then it is likely that we also obtain a good realization of $X_{-n} \dots X_{-1}$. 

\begin{lem}\label{prop-pos-stack}
Let $q\in (0,1)$ and $\lambda \in (0,1/2 )$. There is a $\delta_0 > 0$ (depending only on $q$ and $\lambda$) such that for each $\delta \in (0,\delta_0]$ and each $\ep > 0$, there exists $n_* = n_*(\lambda,\delta,\ep ) \in \BB N$ such that for $n\geq n_*$ and $m \in \BB N$ with $\lambda \leq m/n \leq 1-\lambda$, the following holds. Let $x = x_{-m} \dots x_{-1}$ be any realization of $X_{-m} \dots X_{-1}$ for which $E_m(\ep) \cap F_m$ occurs. Then 
\eqbn
\BB P\left(E_n(\delta) \,|\, X_{-m} \dots X_{-1} = x    ,\, J > n \right) \geq 1-q .
\eqen
\end{lem}
 
The main point of Lemma~\ref{prop-pos-stack} is that $\delta_0$ does not depend on $\ep$ (indeed, the lemma is a trivial consequence of \cite[Theorem 2.5]{shef-burger} without this requirement).  

\begin{proof}[Proof of Lemma~\ref{prop-pos-stack}]
We will deduce the lemma from an analogous estimate for Brownian motion and the scaling limit result~\cite[Theorem 2.5]{shef-burger}. 
For $s\in [0,1]$ and $\delta >0$, let 
\eqbn
G^s(\delta )  :=  \left\{\text{$U_1 - U_s \geq  \delta$ and $V_1 - V_s \geq  \delta$}\right\}  .
\eqen
For $\ep_1 , \ep_2> 0$ define the event $\wt G^s(\ep_1,\ep_2)$ as in~\eqref{eqn-tilde-G^s-def}. By Lemma~\ref{prop-Z-cond-limit}, for each choice of $\delta >0$ we can find $n_* \in \BB N$ (depending on $\ep$, $\delta$, $q$, and $\lambda$) such that the following holds. Suppose $n \geq n_*$; $s \in [\lambda , 1-\lambda]$; and $x$ is a realization of $X_{-\lfloor sn \rfloor} \dots X_{-1}$ for which $F_{\lfloor sn \rfloor}$ occurs, $h_n^s \geq \ep $, and $ c_n^s \geq \ep$, with $h_n^s$ and $c_n^s$ as in~\eqref{eqn-h-c}. Let $\frk h_n^s$ and $\frk c_n^s$ be the corresponding realizations of $h_n^s$ and $c_n^s$. Then
\alb
\BB P\left(E_n(\delta) \,|\, X_{-\lfloor sn \rfloor} \dots X_{-1} = x    ,\, J > n \right) &\geq  \BB P\left(G^s(2\delta ) \,|\, \wt G^s(  \frk h_n^s   ,    \frk c_n^s )    \right) - \frac{q}{2} .
\ale
Taking $m = \lfloor s n \rfloor$, we see that it suffices to prove that for sufficiently small $\delta> 0$, we have 
\eqb \label{eqn-delta-inf}
\inf_{\substack{\ep_1,\ep_2 > 0 \\ s\in [\lambda , 1-\lambda]}} \BB P\left(G^s(\delta ) \,|\, \wt G^s( \ep_1 , \ep_2)    \right)  \geq 1 - \frac{q}{2}.
\eqe 
By \cite[Theorem 2]{shimura-cone} (c.f.\ the proof of Lemma~\eqref{prop-bm-cone-asymp}) the conditional laws $\BB P\left(  \cdot \,|\,    \wt G^s ( \ep_1,\ep_2  ) \right)$ converge weakly as $(\ep_1 , \ep_2) \rta 0$ to a non-degenerate limiting distribution. Hence we can find $\wt\delta_0 > 0$ and $\wt\ep_0 > 0$ depending only on $q$ and $\lambda$ such that whenever $\delta\in (0,\wt\delta_0]$ and $\ep_1 , \ep_2 \in (0, \ep_0]$, we have
\eqb \label{eqn-G-bm-inf}
\inf_{s\in [\lambda,1-\lambda]}  \BB P\left(   G^s ( \delta   )  \,|\,    \wt G^s ( \ep_1,\ep_2  ) \right) \geq 1-q  .
\eqe 
Moreover, by taking the opening angle of the cone in \cite[Theorem 2]{shimura-cone} to be $\pi$ and applying a linear transformation, we find that the  conditional laws  $\BB P\left(  \cdot \,|\,    \wt G^s ( \ep_1,\ep_2  ) \right)$ also converge weakly to a (different) non-degenerate limiting distribution if we send one of $\ep_1$ or $\ep_2$ to 0 and leave the other fixed. Hence we can find $\delta_0 \in (0,\wt\delta_0]$ depending only on $q$, $\lambda$, and $\ep_0$ such that~\eqref{eqn-G-bm-inf} holds whenever $\delta \in (0,\delta_0]$ and one of $\ep_1$ or $\ep_2$ is at least $\ep_0$. Hence if $\delta\in(0,\delta_0]$,~\eqref{eqn-G-bm-inf} holds for every choice of $\ep_1 , \ep_2 > 0$. This completes the proof of the lemma.
\end{proof}

 Our next lemma tells us that if $n > m$, then it is more likely for $\{J > n\}$ to occur if $E_m(\zeta)$ occurs than if $E_m(\zeta)^c$ occurs. Intuitively, the reason why this is the case is that if $E_m(\zeta)^c$ occurs, then at time $-m$ the walk $D$ is close to the boundary of the first quadrant, so it is likely to exit the first quadrant between times $-m$ and $-n$. 

\begin{lem} \label{prop-empty-burger-ratio}
Fix $\lambda \in (0,1/2)$, $q_0 \in (0,1)$, and $\ep > 0$. Suppose we are given $m_0 \in \BB N$ such that $a_{m_0}(\ep) \geq q_0$. Then for $m \in\BB N$ with $\lambda \leq m_0/m \leq 1-\lambda$, $n\in\BB N$ with $\lambda \leq m/n \leq 1-\lambda$, and $\zeta \in (0,1)$ we have  
\eqb \label{eqn-empty-burger-ratio}
\frac{\BB P\left(J > n \,|\, E_m(\zeta) \right)}{\BB P\left(J >n \,|\, E_m(\zeta)^c ,\, J > m \right)} \succeq \frac{ 1  }{   \zeta  + o_{m_0}(1)    }  ,
\eqe  
where the implicit constant depends only on $q_0$, $\lambda$, and $\ep$; and the rate of the $o_{m_0}(1)$ depends only on $\lambda$, $\ep$, and $\zeta$.
\end{lem}
\begin{proof} 
Let $\delta_0 > 0$ be chosen so that the conclusion of Lemma~\ref{prop-pos-stack} holds with given $\lambda$ and $q =1/2$. Let $n_* = n_*(\lambda,\delta_0 ,\ep ) \in \BB N$ be as in that lemma. For $m_0 \geq n_*  $ and $m$ as in the statement of the lemma,  
\eqbn \label{eqn-ep-to-delta}
\BB P\left(E_m(\delta_0 ) \,|\, E_{m_0}(\ep)    ,\, J > m \right) \geq \frac12 .
\eqen   
Hence if $m_0 \geq n_*$ and $\zeta \in (0,1)$, then
\begin{align}
\BB P\left(E_m(\delta_0 ) \,|\, E_{m}(\zeta) \right)   &\geq \BB P\left(E_{m}(\delta_0 ) \,|\, J  > m  \right)\notag \\
&\geq \BB P\left(E_{m}(\delta_0 ) \,|\, E_{m_0}(\ep)    ,\, J  > m \right)  \BB P\left(E_{m_0}(\ep) \,|\, J  > m \right) \notag \\
&\geq \frac12 \BB P\left(E_{m_0}(\ep) \,|\, J  > m \right) . \label{eqn-delta0-split}
\end{align}
By Bayes' rule,
\begin{align}
\BB P\left(E_{m_0}(\ep) \,|\ J  >  m \right) &= \frac{\BB P\left(  J  > m \,|\, E_{m_0}(\ep) \right) \BB P\left(E_{m_0}(\ep) \,|\, J > m_0 \right) }{\BB P\left(  J  > m   \,|\,   J>m_0 \right)} \notag \\
&\geq  \frac{\BB P\left(  J  > m \,|\, E_{m_0}(\ep) \right) a_{m_0}(\ep)  }{\BB P\left(J > m\,|\, E_{m_0}(\ep) \right) a_{m_0}(\ep) + \BB P\left(J  > m ,\,  E_{m_0}(\ep)^c  \,|\, J > m_0 \right) }  .  \label{eqn-ep-to-delta-bayes}
\end{align} 
By \cite[Theorem 2.5]{shef-burger} and our hypothesis on $a_{m_0}(\ep)$, this quantity is bounded below by a constant depending only on $q_0$, $\lambda$, and $\ep$ (not on $\zeta$). By~\eqref{eqn-delta0-split}, we arrive at
\eqbn  
\BB P\left(E_m(\delta_0 ) \,|\, E_m(\zeta) \right) \succeq 1 .
\eqen
By combining this with \cite[Theorem 2.5]{shef-burger} we obtain
\eqb  \label{eqn-top-bound}
\BB P\left(J > n \,|\, E_m(\zeta) \right) \geq \BB P\left(J > n \,|\, E_m(\delta_0) \right)\BB P\left(E_m(\delta_0 ) \,|\, E_m(\zeta) \right) \succeq 1 .
\eqe 
 
Next we consider the denominator in~\eqref{eqn-empty-burger-ratio}. By Lemma~\ref{prop-few-F}, 
\begin{align}
\BB P\left( J> n \,|\, E_m(\zeta)^c   ,\, J > m  \right) &= \frac{\BB P\left(J > n ,\, E_m(\zeta)^c  \,|\, J> m \right)}{\BB P\left( E_m(\zeta)^c   \,|\, J > m \right)} \notag  \\
&\leq \frac{\BB P\left(J > n ,\, F_m ,\, E_m(\zeta)^c  \,|\, J> m \right) + o_{m_0}^\infty(m_0)}{\BB P\left( E_m(\zeta)^c \cap F_m  \,|\, J > m \right)}  .\label{eqn-F-to-no-F}
\end{align}
We have
\alb
 \BB P\left( E_m(\zeta)^c \cap F_m   \,|\, J > m \right)  &\geq \BB P\left( E_m(\zeta)^c \cap F_m   \,|\, E_{m_0}(\ep) ,\, J > m\right) \BB P\left(E_{m_0}(\ep) \,|\, J >  m \right)    \\
&\geq \BB P\left( E_m(\zeta)^c \cap F_{m}   \,|\, E_{m_0}(\ep) \right) \frac{ \BB P\left(E_{m_0}(\ep) \right) }{\BB P\left(J > m\right)} .
\ale
By \cite[Theorem 2.5]{shef-burger}, $\BB P\left(E_{m}(\zeta)^c \,|\, E_{m_0}(\ep) ,\, J > m \right)$ is at least a positive constant depending on $\ep$, $\lambda$, and $\zeta$ but not on $m_0$ (provided $m_0$ is sufficiently large). By Lemma~\ref{prop-X-asymp}, $\frac{ \BB P\left(E_{m_0}(\ep) \right) }{\BB P\left(J >m\right)}$ is bounded below by a constant (depending only on $\ep$ and $\lambda$) times a power of $m_0$. Hence~\eqref{eqn-F-to-no-F} implies
\eqbn
\BB P\left( J> n \,|\,  E_{m}(\zeta)^c   ,\, J > m  \right) \leq \BB P\left( J> n \,|\,  E_{m}(\zeta)^c ,\, F_{m} ,\, J> m  \right) + o_{m_0}^\infty(m_0) .
\eqen 
If $ E_{m}(\zeta)^c \cap F_{m}$ occurs and $J> n$, then $X(-n , -m-1)$ contains either at most $\zeta m^{1/2} + O_n(n^\nu)$ hamburgers or at most $\zeta m^{1/2} + O_n(n^\nu)$ cheeseburgers. By \cite[Theorem 2.5]{shef-burger}, we therefore have
\eqb \label{eqn-bottom-bound}
\BB P\left( J> n \,|\,  E_{m}(\zeta)^c   ,\, J > m  \right)   \preceq \zeta+ o_{m_0}(1) .
\eqe 
We conclude by combining~\eqref{eqn-top-bound} and~\eqref{eqn-bottom-bound}. 
\end{proof}

The following lemma is the main input in the induction argument used to prove Proposition~\ref{prop-F-regularity}.

\begin{lem} \label{prop-F-induct}
Let $q , q_0\in (0,1)$ and $\lambda \in (0,1/2)$. There is a $\ep_0 > 0$ (depending only on $q , q_0$, and $\lambda$) such that for each $\ep \in (0,\ep_0]$ we can find $m_* = m_*(q,q_0,\lambda,\ep) \in \BB N$ with the following property. Suppose $ m <n \in \BB N$ with $m \geq m_*$ and 
\eqb  \label{eqn-m/n-relation}
  \lambda \leq m/n \leq 1-\lambda .
\eqe 
Suppose further that $a_m(\ep) \geq q_0$. Then $a_n(\ep) \geq 1-q$. 
\end{lem}
\begin{proof}
Fix $q\in (0,1)$. Let $\wt m := \frac{m+n}{2}$.
 By Lemma~\ref{prop-pos-stack} we can find $\ep_0 > 0$ (depending only on $  q$ and $\lambda$) such that for $\ep \in (0,\ep_0]$ and $\zeta \in (0,\ep]$, there exists $\wt m_* = \wt m_*(\zeta , \ep,  q,\lambda)\in \BB N$ such that if $m\geq \wt m_*$ and~\eqref{eqn-m/n-relation} holds, then
\begin{align}\label{eqn-last-given-mid}
\BB P\left(E_n(\ep) \,|\, E_{\wt m}(\zeta) ,\, J> n \right)  \geq  1- q  \quad \op{and} \quad \BB P\left(E_{\wt m}(\zeta) \,|\, E_m(\ep),\, J> \wt m\right)  \geq  1- q   . 
\end{align} 
Henceforth fix $\ep \in (0,\ep_0]$. 

Fix $\alpha \in (0,1)$ to be chosen later (depending on $q, q_0, \lambda$, and $\ep$). By Lemma~\ref{prop-empty-burger-ratio}, we can find $  \zeta  \in (0, \ep]$ (depending on $\lambda$, $\alpha$, $q_0$, and $\ep$) and $m_*   \geq \wt m_*$ (depending on $\lambda$, $\alpha$, $q_0$, $\ep$, and $\zeta$) for which the following holds. If $m\geq m_*$,~\eqref{eqn-m/n-relation} holds, and $a_m(\ep) \geq q_0$, then 
\eqb \label{eqn-ratio-compare}
  \BB P\left( J > n \,|\, E_{\wt m}(\zeta)^c ,\, J > \wt m \right)  \leq  \alpha \BB P\left(J >n \,|\, E_{\wt m}(\zeta) \right) .
\eqe
 
Hence if $m\geq m_* $,~\eqref{eqn-m/n-relation} holds, and $a_m(\ep) \geq q_0$ then
\begin{align}
a_n(\ep) &= \frac{\BB P(E_n(\ep))}{\BB P(J >n )}
\geq \frac{ \BB P\left(E_n(\ep) \,|\, E_{\wt m}(\zeta) \right) a_{\wt m}(\zeta) }{\BB P\left(J > n \,|\, E_{\wt m}(\zeta) \right)  a_{\wt m}(\zeta)   + \BB P\left( J > n \,|\, E_{\wt m}(\zeta)^c,\, J > \wt m \right) (1-a_{\wt m}(\zeta))     } \notag \\
&\geq \frac{   \BB P\left(E_n(\ep) \,|\, E_{\wt m}(\zeta) \right)   }{\BB P\left(J > n \,|\, E_{\wt m}(\zeta) \right) } \times   \frac{ a_{\wt m}(\zeta) }{  a_{\wt m}(\zeta)   + \alpha(1-a_{\wt m}(\zeta))     }  . \label{eqn-a-induct1}
\end{align}
By~\eqref{eqn-last-given-mid},
\alb
\frac{   \BB P\left(E_n(\ep) \,|\, E_{\wt m}(\zeta) \right)   }{\BB P\left(J > n \,|\, E_{\wt m}(\zeta) \right) }  = \BB P\left(E_n(\ep) \,|\, E_{\wt m}(\zeta) ,\, J > n\right) \geq 1-  q .
\ale
Furthermore,
\eqb \label{eqn-a(ep)-lower}
a_{\wt m}(\zeta) \geq \BB P\left(E_{\wt m}(\zeta) \,|\, E_m(\ep) ,\, J >\wt m \right) \BB P\left(E_m(\ep) \,|\, J > \wt m\right) \geq (1- q ) \BB P\left(E_m(\ep) \,|\, J > \wt m\right) .
\eqe 
By Bayes' rule,
\begin{align}
\BB P\left(E_{m}(\ep) \,|\ J  >  \wt m \right) &= \frac{\BB P\left(  J  > \wt m \,|\, E_{m}(\ep) \right) \BB P\left(E_{m}(\ep) \,|\, J > m \right) }{\BB P\left(  J  > \wt m   \,|\,   J> m \right)} \notag \\
&\geq  \frac{\BB P\left(  J  > \wt m \,|\, E_{m}(\ep) \right) a_{m}(\ep)  }{\BB P\left(J > \wt m\,|\, E_{m}(\ep) \right) a_{m}(\ep) + \BB P\left(J  > \wt m ,\,  E_{ m}(\ep)^c  \,|\, J >m \right) }  .  
\end{align} 
By \cite[Theorem 2.5]{shef-burger} and our assumption on $a_m(\ep)$, this quantity is at least a positive constant $c $ depending on $q_0$, $\lambda$ and $\ep$ (but not on $\zeta$). Therefore,~\eqref{eqn-a(ep)-lower} implies $a_{\wt m}(\zeta) \geq (1-  q ) c $, so~\eqref{eqn-a-induct1} implies
\eqbn
a_n(\ep )\geq \frac{(1-  q)^2 c }{(1-  q )c  + \alpha } .
\eqen
If we choose $\alpha$ sufficiently small relative to $c $ (and hence $\zeta$ sufficiently small and $m$ sufficiently large), we can make this quantity as close to $1-q$ as we like. Since $q \in (0,1)$ is arbitrary we obtain the statement of the lemma.
\end{proof}

\begin{proof}[Proof of Proposition~\ref{prop-F-regularity}]
Let $q_0$ be as in the conclusion of Lemma~\ref{prop-F-regularity-subsequence}. Also fix $q\in (0,1-q_0]$ and $\lambda \in (0,1/2)$. Let $\ep_0  > 0$ and $m_*  = m_*(q,q_0,\lambda,\ep_0) \in \BB N$ be chosen so that the conclusion of Lemma~\ref{prop-F-induct} holds with this choice of $q_0$. By Lemma~\ref{prop-F-regularity-subsequence} we can find $m \geq m_*$ such that $a_m(\ep_0) \geq q_0$. It therefore follows from Lemma~\ref{prop-F-induct} that $a_n(\ep_0) \geq 1-q$ for each $n \in \BB N$ with $(1-\lambda)^{-1} m \leq n \leq \lambda^{-1} m$. By induction, for each $k\in\BB N$ and each $n\in\BB N$ with $(1-\lambda)^{-k} m \leq n \leq \lambda^{-k} m$, we have $a_n( \ep_0) \geq 1-q \geq q_0$. For sufficiently large $k\in\BB N$, the intervals $[(1-\lambda)^{-k} m , \lambda^{-k} m]$ and $[(1-\lambda)^{-k-1} m , \lambda^{-k-1} m]$ overlap, so it follows that for sufficiently large $n\in\BB N$, we have $[n,\infty)\subset \bigcup_{k\in\BB N} [(1-\lambda)^{-k} m , \lambda^{-k} m]$. Hence $a_n(\ep_0) \geq 1-q$ for each such $n$. Thus~\eqref{eqn-a-lim} holds. 
\end{proof}

\section{Convergence conditioned on no burgers}
\label{sec-no-burger}

\subsection{Statement and overview of the proof} \label{sec-no-burger-setup}
In this section we will prove the following theorem, which is of independent interest but is also needed for the proof of Theorem~\ref{thm-cone-limit}. 
 
\begin{thm} \label{thm-no-burger-conv}
As $n\rta \infty$, the conditional law of $Z^n|_{[-1,0]}$ given the event that $X(-n,-1)$ contains no burgers converges to the law of $\wh Z(-\cdot)$, where $\wh Z$ has the law of a Brownian motion as in~\eqref{eqn-bm-cov} started from 0 and conditioned to stay in the first quadrant until time 1 (as defined just above Lemma~\ref{prop-bm-meander}). 
\end{thm}

\begin{remark}
There is an analogue of Theorem~\ref{thm-no-burger-conv} when we condition on the event that $X(1,n)$ contains no orders, rather than the event that $X(-n,-1)$ contains no orders, which is proven in a similar manner as Theorem~\ref{thm-no-burger-conv}. See Appendix~\ref{sec-no-order-conv}. 
\end{remark}

Throughout this section, we continue to use the notation of Section~\ref{sec-F-reg-setup}, so in particular $J$ is the smallest $j\in\BB N$ for which $X(-j,-1)$ contains a burger.  

The basic outline of the proof of Theorem~\ref{thm-no-burger-conv} is as follows. First, in Section~\ref{sec-burger-times-pos}, we will prove a result to the effect that when $N \in \BB N$ is large, it holds with uniformly positive probability that there is an $i\in [n,Nn]_{\BB Z}$ such that $X(1,i)$ contains no burgers. Using this, in Section~\ref{sec-upper-reg} we will prove a result to the effect that $X(-m_n ,-1)$ is unlikely to have too \textit{many} orders when we condition on $\{J  >n\}$, for $m_n \leq n$ with $m_n\asymp n$ (this complements Proposition~\ref{prop-F-regularity}, which says that $X(-n,-1)$ is unlikely to have too few orders under this conditioning). In Section~\ref{sec-no-burger-tight}, we will use these results to prove tightness of the conditional laws of $Z^n|_{[-1,0]}$ given $\{J>n\}$. In Section~\ref{sec-no-burger-proof}, we will complete the proof of Theorem~\ref{thm-no-burger-conv} by using Lemma~\ref{prop-bm-meander} to identify a subsequential limiting law. 

\subsection{Times with empty burger stack}
\label{sec-burger-times-pos}

In this section, we will prove the following straightforward consequence of Lemma~\ref{prop-F-regularity}, which is a weaker version of Proposition~\ref{prop-late-F} (but which is indirectly needed for the proof of Proposition~\ref{prop-late-F}).

\begin{lem} \label{prop-F-pos}
Fix $N\in\BB N$ and for $n\in\BB N$, let $\mcl E_n = \mcl E_n(N)$ be the event that there is an $i \in [n ,    Nn]_{\BB Z}$ such that $X(1,i)$ contains no burgers. There is a constant $b > 0$ and an $N_*\in\BB N$ (independent of $n$) such that for $N\geq N_*$ and $n\in\BB N$, 
\eqb \label{eqn-F-pos}
\BB P\left(\mcl E_n \right) \geq b  ,\quad \forall n\in\BB N.
\eqe 
\end{lem} 

First we need the following lemma.

\begin{lem} \label{prop-J-compare}
Let $J$ be as in Section~\ref{sec-F-reg-setup} and let $\mu$ be as in~\eqref{eqn-cone-exponent}. For each $N\in\BB N$, we have
\eqbn
 \BB P\left( J  > N n \,|\, J > n \right) \succeq N^{-\mu} + o_n(1) ,
\eqen
with the implicit constant independent of $n$ and $N$. 
\end{lem}
\begin{proof}
By Proposition~\ref{prop-F-regularity}, we can find $\ep  > 0$, independent of $n$, such that (in the notation of that lemma) we have $a_n(\ep ) \geq \frac12 + o_n(1)$. By \cite[Theorem 2.5]{shef-burger} and Lemma~\ref{prop-bm-cone-asymp} we have $\BB P\left( J  > N n \,|\, E_n(\ep) \right) \succeq 
N^{-\mu} + o_n(1)$, with the implicit constant depending on $\ep$ but not on $n$. Therefore,
\alb
 \BB P\left( J  > N n \,|\, J > n \right) \geq  \BB P\left( J  > N n \,|\, E_n(\ep) \right) a_n(\ep) \succeq N^{-\mu} + o_n(1). 
\ale
\end{proof}

\begin{proof}[Proof of Lemma~\ref{prop-F-pos}]
For $i\in\BB N$, let $E_i$ be the event that $X(1,i)$ contains no burgers.
For $j_1 \leq j_2 \in\BB N$, let $B(j_1,j_2)$ be the number of $i\in [j_1+1, j_2]_{\BB Z}$ such that $E_i$ occurs. Set $B_n := B(n,N n)$.
By Lemma~\ref{prop-renewal-relation} (applied with $S_m$ equal to the $m$th time $i$ for which $X(1,i)$ contains no burgers) we have
\begin{align} \label{eqn-second-moment-expand}
\BB E\left(B_n^2 \right) &=  \sum_{i=n}^{Nn-1} \BB P(E_i) +     2 \sum_{i=n}^{N n} \sum_{j=i+1}^{N n} \BB P\left( E_i \cap E_j \right) \notag \\
&=  \BB E(B_n)  +     2 \sum_{i= n}^{ N n -1 } \sum_{j=i+1}^{N n} \BB P\left( E_i  \right) \BB P\left(E_{j-i}\right)\notag \\
&= \BB E(B_n) +     2 \sum_{i= n}^{ N n -1 } \BB P\left( E_i  \right)  \sum_{j=1}^{Nn - i  }  \BB P\left(E_j \right)\notag \\ 
&=  \BB E(B_n) +     2 \sum_{i= n}^{ N n -1  } \BB P\left( E_i  \right) \BB E\left(B(1 , Nn - i   ) \right)       \notag  \\
&\leq  \BB E(B_n)+     2 \BB E(B_n) \BB E\left(B(1 , Nn ) \right)       .
\end{align}
By Lemma~\ref{prop-J-compare}, we can find a constant $c > 0$, independent from $N$ and $n$, such that for sufficiently large $i\in\BB N$ we have (with $J$ as in that lemma) that  
\eqbn
\BB P\left(E_{Ni} \right) = \BB P(J  > Ni ) \geq c N^{-\mu} \BB P(J>i) = c N^{-\mu}\BB P(E_i)  .
\eqen
Therefore,  
\[
\BB E\left(B(1 , Nn ) \right) = \sum_{i=1}^{Nn} \BB P(E_i) \leq c^{-1} N^{ \mu} \sum_{i=1}^{Nn} \BB P(E_{N i})  + O_n(1) . 
\]
Since $\BB P(E_i) = \BB P(J>i)$ is decreasing in $i$, this quantity is at most $c^{-1} N^{ \mu - 1} \BB E\left(B(1 , N^2n ) \right) + O_n(1)$. On the other hand,   
\eqb \label{eqn-B-expand}
 \BB E\left(B(1 , N^2n) \right)  = \BB E\left(B(1 , Nn ) \right) + \sum_{k=2}^{N } \BB E\left(B\left((k-1) Nn+1 , k Nn \right) \right)    .
\eqe 
By monotonicity each term in the big sum in~\eqref{eqn-B-expand} is at most $\BB E(B_n)$. Hence
\eqbn
 cN^{1-\mu}   \BB E\left( B(1 , Nn )    \right)    \leq    \BB E\left(B(1 , Nn ) \right) + (N  - 1) \BB E(B_n)   + O_n(1) .
\eqen
Upon re-arranging we get that for $N$ sufficiently large,
\eqbn
\BB E\left( B(1 , Nn )    \right) \leq \frac{N-1}{c N^{1-\mu}-1} \BB E(B_n)  + O_n(1) \preceq N^\mu \BB E(B_n) + O_n(1)     .
\eqen
By combining this with~\eqref{eqn-second-moment-expand}, we obtain
\alb
\BB E\left(  B_n^2 \right) \preceq  \BB E(B_n)+  N^\mu  \BB E(B_n)^2      .
\ale
Hence the Payley-Zygmund inequality implies
\eqbn
\BB P\left(\mcl E_n \right) =\BB P\left(B_n > 0\right) \succeq N^{-\mu}   .
\eqen
It is clear that $\BB P\left(\mcl E_n\right)$ is increasing in $N$, so we obtain the statement of the lemma. 
\end{proof}

\subsection{Upper bound on the number of orders}
\label{sec-upper-reg}

Proposition~\ref{prop-F-regularity} tells us that it is unlikely that there are fewer than $O_n(n^{1/2})$ hamburger orders or cheeseburgers orders in $X(-n,-1)$ when we condition on $\{J > n\}$. In this section, we will prove some results to the effect that it is unlikely that there are more than $O_n(n^{1/2})$ orders in $X(-n,-1)$ under this conditioning. These results are needed to prove tightness of the conditional law of $Z^n|_{[-1,0]}$ given $\{J > n\}$. 

We first need an elementary lemma which allows us to compare the lengths of the reduced words which we get when we read a given word forward to the lengths when we read the same word backward.

\begin{lem} \label{prop-reverse-length}
For $n\in\BB N$ and $j\in [2, n]_{\BB Z}$, we have 
\eqbn
  |X(j , n)| \leq     |X(1 , n)| +   |X(1,j-1 )|   .
\eqen
\end{lem}
\begin{proof}
For $j\in [1, n]_{\BB Z}$, let $A_j$ denote the set of $k \in [j, n]_{\BB Z}$ with $\phi(k) \in [1, j-1]_{\BB Z}$ and let $B_j$ denote the set of $k \in [j, n]_{\BB Z}$ with $\phi(j) \leq 0$ or $\phi(j) \geq n+1$. Since every symbol in $X$ a.s.\ has a match, it follows that $| X(j,n)| = |A_j| + |B_j|$. On the other hand, for $k\in A_j$ we have that $X_{\phi(k)}$ appears in $X(1,j-1)$ and for $k\in B_j$ we have that $X_k$ appears in $ X(1,n) $. The statement of the lemma follows.
\end{proof}

We next prove a regularity result for the length of the reduced words condition on $\{J > m\}$ which holds for a set of times $m \in \BB N$ which is at most exponentially sparse.

\begin{lem} \label{prop-upper-regularity-subsequence}
For $C>1$ and $m\in\BB N$, let
\eqbn
\wh G_m(C) := \left\{\sup_{j\in [1, m]_{\BB Z}} |X(-j,-1)| \leq C m^{1/2}\right\} .
\eqen
There is an $N_* \in \BB N$ such that for each $N\geq N_*$, there is a constant $c_*(N) >0$ (depending only on $N$) such that the following is true. For each $q\in (0,1/2 )$, there exists $k_* = k_*(q, N)$ such that for $k\geq k_*$, we can find $m \in [N^{k-1} + 1 ,  N^k]_{\BB Z}$ satisfying 
\eqb \label{eqn-upper-regularity-subsequence}
\BB P\left(  \wh G_{m} \left( c_*(N) \log q^{-1} \right) \,|\, J > m\right) \geq 1-q .
\eqe 
\end{lem}
\begin{proof} 
The proof is similar to that of Lemma~\ref{prop-F-regularity-subsequence}. For $k\in\BB N$, define the time $K_{N^k}$ as in Lemma~\ref{prop-F-symmetry} with $n = N^k$. Let $A_k$ be the event that $K_{N^k} \in [N^{k-1}+1,  N^k]_{\BB Z}$. For $C >1$, let 
\[
\wh G_k'(  C) := \left\{ \sup_{i\in [1, K_{N^k}]_{\BB Z}} |X(1,i)| \leq C K_{N^k}^{1/2} \right\} .
\]
By Lemma~\ref{prop-F-pos}, there is an $N_*\in\BB N$, a $k_* \in \BB N$, and a constant $c_0 > 0$ such that for $N\geq N_*$ and $k\geq k_*$ we have $\BB P\left( A_k \right) \geq c_0 N^{-\mu }  $. By~\cite[Lemma 3.13]{shef-burger},  
there are constants $c_1 >0$ and $c_2 > 0$ (depending only on $p$) such that for each $C>1$, 
\eqbn
\BB P\left( \sup_{i\in [1,\dots,N^k ]_{\BB Z} } |X(1,i)|    \geq C N^{k/2}   \right) \leq c_1 e^{-c_2 C} .
\eqen
Hence
\eqbn
\BB P\left(\wh G_k'(C)^c \cap A_k  \right) \leq c_1 e^{-c_2 N^{-1/2} C}  .
\eqen
The right side of this inequality is at most $   q c_0 N^{-\mu} \leq q \BB P(A_k)$ provided we take 
\[
C \geq 2c_*(N) \log q^{-1} , 
\]
for an appropriate choice of $c_*(N) > 0$ depending only on $N$. With this value of $c_*(N)$ we therefore have 
\eqbn
\BB P\left( \wh G_k'\left(2c_*(N) \log q^{-1} \right)    \right)  \geq  \left(1- q  \right) \BB P(A_k) .
\eqen
That is, 
\eqbn
 \sum_{n=N^{k-1}+1}^{N^k} \BB P\left( \wh G_k'\left(2c_*(N) \log q^{-1} \right)   \,|\, K_{N^k} = n\right) \BB P\left(K_{N^k} = n\right) \geq  \left(1-q \right)  \sum_{n=N^{k-1}+1}^{N^k}   \BB P\left(K_{N^k} = n\right).
\eqen
Hence we can find some $m \in [N^{k-1} ,  N^k - 1]_{\BB Z}$ for which
\eqbn
 \BB P\left(    \sup_{i\in [1, m ]_{\BB Z}} |X(1,i)|   \leq   2c_*(N) \log q^{-1}   m^{1/2}       \,|\, K_{N^k} = m +1\right) \geq 1- q .
\eqen
By taking a supremum over all $j$ in the inequality of Lemma~\ref{prop-reverse-length}, we also have
\eqbn
 \BB P\left(    \sup_{j\in [1,  m ]_{\BB Z}} |X(j,m)|   \leq   c_*(N) \log q^{-1}   m^{1/2}       \,|\, K_{N^k} = m +1\right) \geq 1- q .
\eqen
Since the conditional law of $X_1\dots X_m$ given $\{K_{N^k} = m+1\}$ is the same as its conditional law given that $X(1,m)$ contains no burgers and by translation invariance,
\eqbn
\BB P\left(   \wh G_m\left( c_*(N) \log q^{-1}\right)   \,|\, J > m\right) \geq 1- q .
\eqen
\end{proof}

In order to prove a tightness result for the conditional law of $X_{-n} \dots X_{-1}$ given $\{J > n\}$, we only need to prove a regularity condition for an initial segment of the word $X_{-n} \dots X_{-1}$ with length proportional to $n$. The reason why this is sufficient is that once we condition on such a segment and the event $\{J > n\}$, we can estimate the rest of the word using comparison to Brownian motion. In the following lemma, we use Lemma~\ref{prop-upper-regularity-subsequence} to obtain a regularity statement for such an initial segment. 

\begin{lem} \label{prop-initial-reg}
Let $q\in (0,1)$ and $\zeta > 0$. There exists $ \lambda_0, \lambda_1 \in (0,1)$ and $n_* \in \BB N$ (depending on $\zeta$ and $q$) such that for each $n\geq n_*$, we can find a deterministic $m_n = m_n(\zeta ,   q) \in [   \lambda_0  n   ,    \lambda_1 n   ]_{\BB Z}$ such that the following is true. Let $  G_{m_n}(  \zeta)$ be the event that $J  > m_n$ and $|X(-j,-1)| \leq \zeta n^{1/2}$ for each $j\in [1,  m_n]_{\BB Z}$. 
Then we have
\eqb \label{eqn-initial-reg}
\BB P\left( G_{m_n}(  \zeta) \,|\, J > n \right) \geq 1-q .
\eqe  
\end{lem}
\begin{proof}
Fix $\alpha \in (0,1/4)$ to be chosen later (depending on $\zeta$ and $q$). 
Let $N_* \in\BB N$ be chosen sufficiently large that the conclusion of Lemma~\ref{prop-upper-regularity-subsequence} holds. Fix $N\geq N_*$ and let $c_*(N)$ be as in that lemma. Given $\zeta > 0$, let $k_n $ be the largest $k\in\BB N$ for which $c_*(N) \log \alpha^{-1} N^{k/2} \leq \zeta n^{1/2}$. If $n$ is chosen sufficiently large, then by Lemma~\ref{prop-upper-regularity-subsequence} we can find $m_n \in [N^{k_n-1} ,  N^{k_n}]_{\BB Z}$ such that~\eqref{eqn-upper-regularity-subsequence} holds with $\alpha$ in place of $q$. In the notation of~\eqref{eqn-upper-regularity-subsequence} we have
\eqbn
  \wh G_{m_n}\left(  c_*(N) \log \alpha^{-1} \right) \cap \{J > m_n\}  \subset G_{m_n}\left(   \zeta   \right) .
\eqen
Let
\eqbn
 \rho(\alpha) := \left( c_*(N) \log \alpha^{-1} \right)^{-1}  .
\eqen
Then $   \lambda_0(\alpha)  n   \leq m_n \leq   \lambda_1(\alpha)   n $ for $\lambda_0(\alpha)  = N^{-2}\rho(\alpha)^2 \zeta^2$ and $\lambda_1(\alpha)  = \rho(\alpha)^2 \zeta^2$.   

We have $\BB P\left(G_{m_n}(\zeta)  \,|\, J > m_n \right) \geq 1-\alpha$. We need to show that if $\alpha$ is chosen sufficiently small and $n$ is chosen sufficiently large (depending on $\zeta$ and $q$), then we can transfer this to a lower bound when we further condition on $\{J>n\}$. We will do this via a Bayes rule argument. 

By Proposition~\ref{prop-F-regularity}, we can find $\ep > 0$ (depending only on $p$) and $\wt n_* \in \BB N$ (depending on $\ep$, $\alpha$, $N$, and $\zeta$) such that (in the notation of that proposition) we have $a_{m_n}(\ep) \geq 1/2$ for each $n \geq \wt n_*$. 
For this choice of $\ep$, we have for $n\geq \wt n_*$ that
\alb
 \BB P\left(   J > n \,|\, G_{m_n}(\zeta)   \right)  
&\geq    \BB P\left(   J > n \,|\,  E_{m_n}(\ep)    \cap G_{m_n}(\zeta)   \right) \BB P\left(E_{ m_n}(\ep) \cap  G_{m_n}(\zeta) \,|\, J > m_n   \right) \\
&\geq   \frac14 \BB P\left(   J > n \,|\,  E_{m_n}(\ep)    \cap G_{m_n}(\zeta)   \right)   .
\ale
By \cite[Theorem 2.5]{shef-burger} and Lemma~\ref{prop-bm-cone-asymp}, there is an $n_* \geq \wt n_*$ (depending on $\ep$ and $\lambda_0$) and a constant $c_0  >0$ (depending only on $p$) such that for $n\geq n_*$, 
\eqbn
 \BB P\left(   J > n \,|\,  E_{ m_n}(\ep)   \cap G_{m_n}(\zeta)    \right) \geq c_0 \ep^\mu  \left(  N^{- 1} \rho(\alpha)  \zeta \right)^{ \mu} .
\eqen
Hence for $n\geq n_*$, 
\eqbn
 \BB P\left(   J > n \,|\, G_{m_n}(\zeta)  \right)  \geq \wh \rho(\alpha) := \frac{c_0 \ep^\mu}{4} \left(  N^{- 1} \rho(\alpha)  \zeta \right)^{2\mu}   .
\eqen
By Bayes' rule, 
\alb
\BB P\left( G_{m_n}(\zeta)  \,|\, J > n   \right)  
&\geq \frac{ \BB P\left(   J > n \,|\, G_{m_n}(\zeta)  \right)  \BB P\left(G_{m_n}(\zeta) \,|\, J>m_n\right)    }{ \BB P\left(   J > n \,|\, G_{m_n}(\zeta)   \right)  \BB P\left(G_{m_n}(\zeta) \,|\, J>m_n\right)   +  \BB P\left(G_{m_n}(\zeta)^c \,|\, J>m_n\right)     } \\
&\geq \frac{ (1-\alpha) \wh \rho(\alpha)       }{ (1-\alpha) \wh \rho(\alpha)    +  \alpha   } 
= \frac{ (1-\alpha)        }{ (1-\alpha)     +  \alpha  \wh \rho(\alpha)^{-1} }.
\ale
As $\alpha \rta 0$, we have $\alpha  \wh \rho(\alpha)^{-1} \rta 0$, so if $\alpha$ is chosen sufficiently small (depending on $\zeta$ and $q$), and hence $n_*$ is chosen sufficiently large (depending on $\zeta$, $q$, and $\alpha$) this quantity is at least $1-q$.  
\end{proof}

\subsection{Proof of tightness} 
\label{sec-no-burger-tight}

In this section we will prove tightness of the conditional laws of $Z^n|_{[-1,0]}$ given $\{J>n\}$. We first need the following basic consequence of the results of Section~\ref{sec-F-reg}.

\begin{lem} \label{prop-reverse-cond-regularity}
Suppose we are in the setting of Section~\ref{sec-F-reg-setup}. Let $\lambda \in (0,1/2)$ and $q\in (0,1)$. There exists $\ep > 0$ and $n_* \in \BB N$, depending only on $q$ and $\lambda$, such that for each $n\geq n_*$ and $m\in\BB N$ with $\lambda \leq m/n\leq 1-\lambda$, 
\eqbn
\BB P\left( E_m(\ep) \,|\, J > n  \right)  \geq 1-q .
\eqen
\end{lem}
\begin{proof} 
Fix $\alpha \in (0,1)$ to be determined later, depending only on $q$. By Proposition~\ref{prop-F-regularity}, we can find $\ep_0 > 0$ and $m_* \in\BB N$ such that (in the notation of Section~\ref{sec-F-reg-setup}) it holds for each $m \geq m_*$ and $\ep \in (0,\ep_0]$ that $a_{m}(\ep ) \geq 1-\alpha$. By Lemma~\ref{prop-empty-burger-ratio}, we can find $\ep \in (0,\ep_0]$ and $n_* \in \BB N$ with $n_* \geq  \lambda^{-2} m_*$ such that for $n\geq n_*$ and $m$ as in the statement of the lemma, we have  
\eqbn
\BB P\left(  J > n \,|\, E_{m }(\ep)^c ,\, J > m \right)        \leq  \alpha \BB P\left(  J > n \,|\, E_{m}(\ep) \right)  .
\eqen
By Bayes' rule,  
\alb
\BB P\left( E_{m }(\ep) \,|\, J > n \right)  &= \frac{   \BB P\left(  J > n \,|\, E_{m }(\ep) \right) a_{m }(\ep) }{\BB P\left(  J > n \,|\, E_{m  }(\ep) \right) a_{m }(\ep) + \BB P\left(  J > n \,|\, E_{m }(\ep)^c ,\, J > m  \right)(1- a_{m }(\ep) )    } \\
 &\geq \frac{   1-\alpha  }{ 1-\alpha  + \alpha^2    }  .
\ale
By choosing $\alpha$ sufficiently small, in a manner which depends only on $q$, we can make this last quantity greater than or equal to $1-q$. 
\end{proof}

\begin{lem} \label{prop-no-burger-tight}
The conditional laws of $Z^n|_{[-1,0]}$ given $\{J>n\}$ for $n\in\BB N$ are a tight family of probability measures on the set of continuous functions on $[-1,0]$ in the topology of uniform convergence.
\end{lem}
\begin{proof} 
For $\delta , \zeta>0$ and $n\in\BB N$, let $\wt G_n( \zeta , \delta)  $ be the event that the following is true. Whenever $t_1,t_2 \in [-1,0]$ with $|t_1-t_2| \leq \delta$, we have $|Z^n(t_1) - Z^n(t_2)| \leq \zeta$. For a continuous non-decreasing function $\rho : (0,\infty) \rta (0,\infty)$ with $\lim_{\zeta\rta 0} \rho(\zeta) = 0$, let 
\eqbn
\mcl G_n(\rho  ) :=  \bigcap_{\zeta > 0 } \wt G_n(\zeta , \rho(\zeta) )  .
\eqen
By the Arz\'ela-Ascoli theorem (note that equicontinuity implies uniform boundedness in this case since each $Z^n$ vanishes at the origin), it suffices to show that for each given $q \in (0,1)$, we can find $\rho$ as above, independent of $n$, such that  
\eqb \label{eqn-no-burger-tight}
\BB P\left(\mcl G_n(\rho)  \,|\, J > n\right) \geq 1-q ,\quad \forall n\in \BB N .
\eqe 

First suppose we are given $\zeta >0$ and $\alpha \in(0,1)$. By Lemma~\ref{prop-initial-reg}, we can find $n_1 \in \BB N$ and $\lambda_0 , \lambda_1\in (0,1)$ (depending on $\zeta$ and $\alpha$) such that for each $n\geq n_1$ there exists $m_n \in [ \lambda_0 n ,   \lambda_1 n ]_{\BB Z}$ such that (in the notation of Lemma~\ref{prop-initial-reg}), we have that~\eqref{eqn-initial-reg} holds with $1-\alpha/2$ in place of $1-q$. 

By Lemma~\ref{prop-reverse-cond-regularity}, we can find $\ep > 0$ and $n_2 \geq n_1$ (depending on $\zeta$ and $\alpha$) such that for $n \geq n_2$, we have $\BB P\left( E_{m_n}(\ep) \,|\, J > n  \right) \geq 1-\alpha/2$. 
By Lemma~\ref{prop-Z-cond-limit} that we can find $n_3 \geq n_2$ and $\delta_0 = \delta_0(\alpha,\zeta) >0$ such that if $n\geq n_3$, then with conditional probability at least $ 1-\alpha $ given $E_{m_n}(\ep) \cap G_{m_n}(\zeta) \cap  \{J>n\}$, it holds that whenever $t_1 , t_2 \in [-1,-m_n/n]$ with $|t_1-t_2|\leq \delta_0$, we have $|Z^n(t_1) - Z^n(t_2)| \leq \zeta$. Call this last event $A$. If $A$ occurs and $G_{m_n}(\zeta)$ occurs then $\wt G_n( \zeta , \delta_0)$ occurs.
Therefore, if $n\geq n_3$, then 
\alb
\BB P\left( \wt G_n( 2\zeta , \delta_0)    \,|\, J > n\right) &\geq \BB P\left(   A \cap G_{m_n}(\zeta) \,|\, J>n    \right)\\
&\geq \BB P\left(   A   \,|\,  E_{m_n}(\ep) \cap G_{m_n}(\zeta)  ,\, J>n \right) \BB P\left(E_{m_n}(\ep) \cap G_{m_n}(\zeta)  \,|\, J > n\right) \\
&\geq (1-\alpha)^2 . 
\ale
Since there are only finitely many $n\leq n_*^3$, we can find $\delta \in (0,\delta_0]$ depending only on $n_3$ such that  
\eqb \label{eqn-equicont1}
\BB P\left(\wt G_n(2\zeta,\delta)     \,|\, J > n\right) \geq (1-\alpha)^2 ,\quad \forall n\in\BB N. 
\eqe 

Now fix $q\in (0,1)$.  
For $j\in\BB N$, choose $\delta_j  > 0$ for which~\eqref{eqn-equicont1} holds with $\delta =\delta_j$, $\zeta = 2^{-j-1}$, and $\alpha$ chosen so that $(1-\alpha)^2 =1- q 2^{-j-1}$. Let 
\eqbn
\rho(\zeta) := C\BB 1_{[1, \infty)}(\zeta) +  \sum_{j=1}^\infty \delta_j \BB 1_{[2^{-j } , 2^{-j+1})}(\zeta). 
\eqen
Then~\eqref{eqn-no-burger-tight} holds for this choice of $\rho$. 
\end{proof}

\subsection{Identifying the limiting law} 
\label{sec-no-burger-proof}

To identify the law of a subsequential limit of the laws of $Z^n|_{[-1,0]}$ given $\{J>n\}$, we need the following fact from elementary probability theory.

\begin{lem} \label{prop-cond-law-conv}
Let $(X_n,Y_n)$ be a sequence of pairs of random variables taking values in a product of separable metric spaces $\Omega_X\times\Omega_Y$ and let $(X,Y)$ be another such pair of random variables. Suppose $(X_n ,Y_n) \rta (X,Y)$ in law. Suppose further that there is a family of probability measures $\{\mu_y : y\in \Omega_Y\}$ on $\Omega_X$, indexed by $\Omega_Y$, and a family of $\sigma(Y_n)$-measurable events $E_n$ with $\lim_{n\rta\infty} \BB P(E_n) =1$ such that for each bounded continuous function $f : \Omega_X\rta\BB R $, we have
\eqbn
\BB E\left(f(X_n) \, |\, Y_n  \right) \BB 1_{E_n} \rta \BB E_{\mu_Y}(f)  \quad \text{in law}.
\eqen
Then $\mu_Y$ is the regular conditional law of $X$ given $Y$.
\end{lem}
\begin{proof}
Let $g : \Omega_Y\rta \BB R$ be a bounded continuous function. Then for each bounded continuous function $f : \Omega_X\rta\BB R $,
\alb
\BB E\left(f(X) g(Y) \right) &= \lim_{n\rta \infty} \BB E\left(f(X_n) g(Y_n) \right) \\
&= \lim_{n\rta \infty} \BB E\left(f(X_n) g(Y_n) \BB 1_{E_n} \right) \\
&= \lim_{n\rta\infty} \BB E\left( \BB E\left(f(X_n) \, |\, Y_n \right) \BB 1_{E_n} g(Y_n)   \right)\\
&=  \BB E\left( \BB E_{\mu_Y}(f) g(Y)\right) .
\ale
By the functional monotone class theorem, we have $\BB E\left(F(X,Y) \right) = \BB E\left( \BB E_{\mu_Y}(F(\cdot,Y) ) \right)$ for every bounded Borel-measurable function $F$ on $\Omega_X\times\Omega_Y$. This implies the statement of the lemma. 
\end{proof}

\begin{proof}[Proof of Theorem~\ref{thm-no-burger-conv}]
By Lemma~\ref{prop-no-burger-tight} and the Prokhorov theorem, from any sequence of integers tending to $\infty$, we can extract a subsequence along which the conditional laws of $Z^{n }$ given $ J > n $ converge to the law of some random continuous function $\wt Z  = (\wt U , \wt V): [-1,0]\rta \BB R^2$ as $n\rta\infty$, restricted to this subsequence. We must show that $\wt Z \eqD \wh Z(-\cdot)$, with $\wh Z$ as defined in the statement of the theorem. 

By Lemma~\ref{prop-reverse-cond-regularity}, we a.s.\ have $\wt U(s) > 0$ and $\wt V(s) > 0$ for each $s  \in (0,1)$. 
By Lemma~\ref{prop-bm-meander}, it therefore suffices to show that for each $\zeta \in (0,1)$, the conditional law of $\wt Z|_{[-1,-\zeta]}$ given $\wt Z|_{[-\zeta,0]}$ is that of a Brownian motion with covariances as in~\eqref{eqn-bm-cov}, starting from $\wt Z(-\zeta)$, parametrized by $[-1,-\zeta]$, and conditioned to stay in the first quadrant. 

To lighten notation we henceforth consider only values of $n$ in our subsequence and implicitly assume that all statements involving $n$ are for $n$ restricted to this subsequence. 

Fix $\zeta   \in (0,1)$. Also let $\wh D_\zeta$ be the path defined in the same manner as the path $D$ of~\eqref{eqn-discrete-paths} in Section~\ref{sec-burger-prelim} but with the following modification: if $j\in [-\zeta n,  -1]_{\BB Z}$, $X_j = \tb F$, and $-\phi (-j)  > \zeta n$, then $\wh D_\zeta(-j) - \wh D_\zeta(-j+1)$ is equal to zero rather than $( 1,0)$ or $(0, 1)$. Extend $\wh D_\zeta$ to $[-\zeta n ,0]$ by linear interpolation (we require it to be constant on $[-\zeta n , -\lfloor \zeta n \rfloor]$). For $t\in [-\zeta,0]$, let 
$\ul Z^n_\zeta(t) := n^{-1/2} \wh D_\zeta(n t)$. It follows from Lemma~\ref{prop-few-F} that $\sup_{t\in [-\zeta ,0]} |\wh Z^n_\zeta(t) - Z^n(t)| \rta 0$ in law, even if we condition on $\{J > n\}$, whence $\ul Z^n_\zeta \rta \wt Z|_{[-\zeta , 0]}$ in law.   
We note that $\ul Z^n_\zeta$ determines and is determined by $X_{-\lfloor \zeta n \rfloor} \dots X_{-1}$, so is independent from $\dots X_{-\lfloor \zeta n \rfloor-2} X_{-\lfloor \zeta n \rfloor-1}$ and hence also from $Z^n_{[-\zeta,0]}$ (Notation~\ref{def-Z-restrict}). 

Let $(\wh X^n)$ be a sequence of random words distributed according to the conditional law of $X_{-n} \dots X_{-1}$ given $\{J>n\}$. 
Let $(\wh{  Z}^n)$ be the corresponding paths, so that each $\wh{ Z}^n$ has the conditional law of $ Z^n$ given $\{J >n\}$. Let $\ul{\wh Z}^n_\zeta$ be the corresponding random paths $\ul Z_\zeta^n$. By the Skorokhod theorem, we can couple $(\wh{ X}^n )$ with $\wt Z$ (with $n$ restricted to our subsequence) in such a way that a.s.\ $\ul{\wh Z}_\zeta^n \rta \wt Z|_{[-\zeta,0]}$ uniformly. 

For $\ep_1 , \ep_2 > 0$, define $\wt G^\zeta(\ep_1 , \ep_2)$ as in~\eqref{eqn-tilde-G^s-def} with $s=\zeta$. 
By Lemma~\ref{prop-Z-cond-limit}, for each fixed $\ep > 0$, the Prokhorov distance between the conditional law of $Z^n_{[-1,-\zeta]}$ given $J  > n$ and any realization of $\wh Z_\zeta^n$ for which $E_{\lfloor \zeta n \rfloor}(\ep) \cap F_{\lfloor \zeta n \rfloor}$ (defined as in Section~\ref{sec-F-reg-setup}) occurs; and the conditional law of $Z_{[-1,-\zeta]}$ given the event $\wt G^\zeta \left( \wt U(-\zeta) ,  \wt V(-\zeta)   \right)$ converges to zero as $n\rta\infty$. By combining this with Lemma~\ref{prop-reverse-cond-regularity},  
we obtain that for any bounded continuous function $f$ from the space of continuous functions on $[-1 , -\zeta]$ (in the uniform topology) to $\BB R$, we have
\eqb \label{eqn-cond-conv}
 \BB E\left( f\left(Z^n_{[-1,-\zeta]}\right)    \,|\, J > n,\,  \ul Z_\zeta^n    \right) \BB 1_{F_{\lfloor \zeta n \rfloor}} \rta  \BB E\left(f \left( Z_{[-1,-\zeta]}\right) \,|\,  \wt G^\zeta\left( \wt U(-\zeta) ,  \wt V(-\zeta)   \right)\right)  
\eqe  
in law.
We now conclude by applying Lemma~\ref{prop-cond-law-conv} with $X_n = \wh Z^n_{[-1,-\zeta]}$, $Y_n = \ul{ Z}_\zeta^n$, $X = \wt Z|_{[-1,-\zeta ]}$, and $Y    = \wt Z|_{[-\zeta,0]}$.
\end{proof}

\section{Convergence of the cone times}
\label{sec-cone-conv}

In this section we will conclude the proof of Theorem~\ref{thm-cone-limit}. 
We start in Section~\ref{sec-reg-var} by proving that the law of the random variable $J$ studied in Section~\ref{sec-F-reg} is regularly varying (Proposition~\ref{prop-J-reg-var}).
This will be accomplished by means of Theorem~\ref{thm-no-burger-conv}. 
We will deduce several consequences from this regular variation, including Proposition~\ref{prop-late-F}.
In Section~\ref{sec-cone-conv}, we will deduce Theorem~\ref{thm-cone-limit} from Proposition~\ref{prop-late-F}.
In Section~\ref{sec-cone-cond}, we record an analogue of Theorem~\ref{thm-cone-limit} in the setting of Theorem~\ref{thm-no-burger-conv}, i.e.\ when we condition on the event that the reduced word contains no burgers.

\subsection{Regular variation}
\label{sec-reg-var}

We say that the law of a random variable $A$ is \emph{regularly varying} with exponent $\alpha$ if for each $c>1$, 
\eqbn
\lim_{a\rta \infty} \frac{\BB P\left(A > c a\right)}{\BB P\left(A >a\right)} = c^{-\alpha} .
\eqen
In this subsection we will prove that the laws of several quantities associated with the word $X$ are regularly varying. In doing so, we will obtain Proposition~\ref{prop-late-F}. See Appendix~\ref{sec-no-order} for analogues of the results of this subsection when we read $X$ forward and condition on no orders.

\begin{prop} \label{prop-J-reg-var}
Let $J$ be the smallest $j\in\BB N$ for which $X(-j,-1)$ contains a burger. The law of $J$ is regularly varying with exponent $\mu$, as defined in~\eqref{eqn-cone-exponent}. 
If $\wt J$ denotes the smallest $j\in \BB N$ for which $X(-j,-1)$ contains no $\tb F$-symbols, then $\wt J$ is also regularly varying with exponent $\mu$. 
\end{prop} 

We note that Proposition~\ref{prop-J-reg-var} can be viewed as an analogue for the random path $D = (d,d^*)$ studied in this paper of the tail asymptotics for the exit time from a cone of a random walk with independent increments obtained in~\cite[Theorem 1]{dw-cones}. However, unlike the estimate which is implicit in Proposition~\ref{prop-J-reg-var}, the estimate of~\cite{dw-cones} does not involve a slowly varying correction.

\begin{proof}[Proof of Proposition~\ref{prop-J-reg-var}]
Fix $c > 1$. For $z  \in (0,\infty)^2$, write $\Phi_c(z)$ for the probability that a two-dimensional Brownian motion with covariances~\eqref{eqn-bm-cov} started from $z$ stays in the first quadrant until time $c-1$. Note that $\Phi_c$ is a bounded continuous function of $z$.

Let $\wt Z = (\wt U , \wt V)$ have the law of $Z|_{[-1,0]}$ conditioned to stay in the first quadrant. For $n\in\BB N$, let $\wh Z^n$ be defined in the same manner as the path $\wh Z^n_\zeta$ used in the proof of Theorem~\ref{thm-no-burger-conv}, but with 1 in place of $\zeta$, so that $\wh Z^n $ determines and is determined by $X_{-n} \dots X_{-1}$ and is independent from $\dots X_{- n-2} X_{- n-1}$ and hence also from $Z^n_{[-c,-1]}$.

By the same argument used to obtain~\eqref{eqn-cond-conv} in the proof of Theorem~\ref{thm-no-burger-conv}, we have that 
\eqb \label{eqn-J>cn-conv}
 \BB P\left( J  >c n   \,|\, J > n,\,\wh Z^n \right) \BB 1_{F_n}  \rta \Phi_c(\wt Z(1)) 
\eqe 
in law, where here (as usual) $F_n$ is the event that $X(-n,-1)$ contains at most $n^\nu$ flexible orders for some $\nu \in (\mu',1/2)$.

Since the conditional law of $\wh Z^n$ given $J > n$ converges to the law of $\wt Z$ and $\lim_{n\rta\infty} \BB P(F_n) = 1$, we can take expectations to get
\eqbn
 \BB P\left( J  >c n   \,|\, J > n  \right) = \frac{\BB P\left( J  > c n \right)}{\BB P\left(J > n\right)} \rta f(c),   
\eqen
where $f(c) : = \BB E\left( \Phi_c(\wt Z(1)) \right)$. 

We have $f(1) = 1$, $f(c) \in (0,1)$ for each $c > 1$, and
\eqbn
f(c) f(c') = \lim_{n\rta\infty} \frac{\BB P\left( J  > c n \right)}{\BB P\left(J > n\right)} \times \frac{\BB P\left( J  > c c' n \right)}{\BB P\left(J > c n\right)} = f(c c') .
\eqen
We infer that $f(c) = c^{-\alpha}$ for some $\alpha > 0$. 

To identify $\alpha$, we need only consider the asymptotics of $\BB E\left( \Phi_c(\wt Z(1)) \right)$ as $c\rta\infty$. To this end, we apply \cite[Equation 4.3]{shimura-cone} (c.f.\ the proof of Lemma~\ref{prop-bm-cone-asymp}) to get that for fixed $z\in (0,\infty)^2$, we have 
\[
\lim_{c\rta\infty} c^{ \mu} \Phi_c(z) = \Psi(z)
\]
for some positive continuous function $\Psi$ on $(0,\infty)^2$ which is bounded in every neighborhood of the origin. By the formula \cite[Equation 3.2]{shimura-cone} for the density of the law of $\wt Z(1)$, it follows that $\BB P\left(|\wt Z(1)| > A\right)$ decays quadratic-exponentially in $A$. By Brownian scaling and \cite[Equation 4.2]{shimura-cone}, 
\eqbn
\sup_{z\in B_A(0) \cap (0,\infty)^2} |\Phi_c(z)| \preceq c^{-\mu} A^{2\mu} 
\eqen
with the implicit constant depending only on $p$. Hence
\eqbn
\BB E\left(|c^{ \mu} \Phi_c(\wt Z(1) ) | \BB 1_{\{|c^{ \mu} \Phi_c(\wt Z(1) ) |  \geq A \}} \right) 
\rta 0
\eqen
 as $A\rta \infty$, uniformly in $c$. By the Vitalli convergence theorem, $ c^\mu f(c) =  \BB E\left( c^{ \mu} \Phi_c(\wt Z(1) ) \right)$ converges to a finite constant as $c\rta\infty$, whence we must have $\alpha = \mu$.

For the last statement, we note that with probability $1-p/2$ we have $\wt J = 1$, and with probability $p/2$, $\wt J$ is equal to the smallest $j\in\BB N$ for which $X(-j,-2)$ contains a burger. It follows that for $n \geq 2$ we have $\BB P\left(\wt J > n\right) = \frac{p}{2} \BB P\left(J > n-1\right)$. Hence
\eqbn
\lim_{n\rta\infty} \frac{\BB P\left( J  > c n \right)}{\BB P\left(J > n\right)} = \lim_{n\rta\infty} \frac{\BB P\left( \wt J  > c n \right)}{\BB P\left(\wt J > n\right)} .
\eqen
\end{proof}

From Proposition~\ref{prop-J-reg-var}, we can deduce that there a.s.\ exist macroscopic $\tb F$-excursions, which is the key input in our proof of Theorem~\ref{thm-cone-limit} in the next section. 

\begin{proof}[Proof of Proposition~\ref{prop-late-F}]
For $m\in\BB N$, let $\wt J_m$ be the $m$th smallest $j\in\BB N$ for which $X(-j,-1)$ contains no $\tb F$ symbols. Then the words $X_{-\wt J_m} \dots X_{-\wt J_{m-1}-1}$ are iid. By Corollary~\ref{prop-J-reg-var}, $\wt J_1$ is regularly varying with exponent $ \mu \in (0,1)$. For $n\in\BB N$ let $M_n$ be the largest $m\in\BB N$ for which $\wt J_m \leq n$. By the Dynkin-Lamperti theorem~\cite{dynkin-limits,lamperti-limits}, $n^{-1}\left(n - J_{M_n}\right)$ converges in law to a generalized arcsine distribution with parameter $\mu$. Since this distribution does not have a point mass at the origin we obtain the statement of the proposition.
\end{proof}

We end by recording some consequences of Proposition~\ref{prop-J-reg-var} which are of independent interest, but are not needed for the proof of Theorem~\ref{thm-cone-limit}. 

\begin{cor} \label{prop-few-F-optimal}
The statement of Lemma~\ref{prop-few-F} holds, exactly as stated, with $1-\mu$ in place of $\mu'$. 
\end{cor}
\begin{proof}
For $i\in\BB N$, let $E_i$ be the event that $X(1,i)$ contains no burgers. By Proposition~\ref{prop-J-reg-var} and translation invariance,
\eqbn
\BB P\left(E_i\right) = \BB P(J > i) = i^{-\mu+o_i(1)} .
\eqen
The corollary now follows from Lemma~\ref{prop-few-renewal} (c.f.\ the proof of Lemma~\ref{prop-few-F}). 
\end{proof}

\begin{cor} \label{prop-K-reg-var}
Let $K^F$ be the smallest $i\in\BB N$ for which $X(1,i)$ contains a flexible order. The law of $ K^F$ is regularly varying with exponent $1-\mu$.  
\end{cor}
\begin{proof}
For $m\in\BB N$, let $K^F_m$ be the smallest $i\in\BB N$ for which $X(1,i)$ contains at least $m$ flexible orders. The words $X_{K^F_{m-1}+1}\dots X_{K^F_m}$ are iid. For $n\in\BB N$, let $M_n^*$ be the largest $m\in\BB N$ for which $K^F_m \leq n$. Equivalently, $K^F_{M_n^*}$ is the greatest integer $i\in [1, n]_{\BB Z}$ such that $X_i =\tb F$ and $\phi(i) \leq 0$. By translation invariance, we have $K^F_{M_n^*} \eqD n - \wt J_{M_n}$, with the latter defined in the proof of Proposition~\ref{prop-late-F}. Hence the law of $n^{-1} K^F_{M_n^*}$ converges to the generalized arcsine distribution with parameter $\mu$. Therefore $n^{-1}\left( n - K^F_{M_n^*} \right)$ converges in law to a generalized arcsine distribution with parameter $1-\mu$. By the converse to the Dynkin-Lamperti theorem, $K^F_{M_n^*}$ is regularly varying with exponent $1-\mu$.
\end{proof}

\begin{remark}
In the terminology of~\cite{blr-exponents}, Corollary~\ref{prop-K-reg-var} states that the law of the area of the part traced after time 0 of the ``envelope" of the smallest loop surrounding the root vertex in the infinite-volume model is regularly varying with exponent $1-\mu$. In \cite[Section 1.2]{blr-exponents}, the authors conjecture that the tail exponent for the law of the area of this loop itself is $1-\mu$. We expect that this conjecture (plus a regular variation statement for the tail) can be deduced from Proposition~\ref{prop-J-reg-var} and Corollary~\ref{prop-K-reg-var} via arguments which are very similar to some of those given in Sections~\ref{sec-F-reg} and~\ref{sec-no-burger} of the present paper, but we do not carry this out here.
\end{remark}

\subsection{Proof of Theorem~\ref{thm-cone-limit}}
\label{sec-cone-proof}

In this section, we will complete the proof of Theorem~\ref{thm-cone-limit}. 
We first need a general deterministic statement about the convergence of $\pi/2$-cone times which in particular will allow us to deduce condition~\ref{item-cone-limit-times} in the theorem statement from the other conditions. 
To state this result, we need to introduce the notion of a strict $\pi/2$-cone time, which is defined in the same manner as a weak $\pi/2$-cone time (Definition~\ref{def-cone-time}) but with weak inequalities replaced by strict inequalities. 

\begin{defn}\label{def-strict-cone-time}
A time $t$ is called a \emph{strict $\pi/2$-cone time} for a function $Z = (U,V) : \BB R \rta \BB R^2$ if there exists $t' < t$ such that $U_s > U_{t }$ and $V_s > V_{t }$ for $s\in (t'   , t )$. Equivalently, $Z((t'   , t  ))$ is contained in the open cone $Z_{t } + \{z\in \BB C : \op{arg} z \in [0,\pi]\}$. We write $\wt v_Z(t)$ for the infimum of the times $t'$ for which this condition is satisfied.  
\end{defn} 
 
If $t$ is a strict $\pi/2$-cone time for $Z$, then $t$ is also a weak $\pi/2$-cone time for $Z$ and we have $\wt v_Z(t) \leq v_Z(t)$. The reverse inequality need not hold. For example, $Z$ might enter the closed cone at time $\wt v_Z(t)$, hit the boundary of the closed cone at time $v_Z(t) \in (\wt v_Z(t) , t)$, then stay in the open cone until time $t$.

\begin{lem} \label{prop-cone-conv}
Let $Z  = (U,V) : \BB R \rta \BB R^2$ be a continuous path with the following properties.
\begin{enumerate}
\item Each weak $\pi/2$-cone time $t$ for $Z$ is a strict $\pi/2$-cone time for $Z$ and satisfies $\wt v_Z(t) = v_Z(t)$. \label{item-Z-weak}
\item $Z$ has no weak $\pi/2$-cone times $t$ with $Z_{v_Z(t)} = Z_t$.\label{item-Z-direction}
\item $\liminf_{t\rta -\infty} U(t) = \liminf_{t\rta -\infty} V(t) = -\infty$. \label{item-Z-liminf}
\end{enumerate}
Let $Z^n = (U^n,V^n)$ be a sequence of continuous paths $\BB R \rta \BB R^2$ such that $Z^n\rta Z$ uniformly on compacts. Suppose that for each $n\in \BB N$, $t_n$ is a weak $\pi/2$-cone time for $Z^n$. Suppose further that almost surely $\liminf_{n\rta\infty}( t_n -  v_{Z^n}(t_n) ) > 0$. If $t_n \rta t $ for some $t\in \BB R$, then $t$ is a strict $\pi/2$-cone time for $Z$. Furthermore, $\lim_{n\rta\infty} v_{Z^n}(t_n)  = v_Z(t)$, $\lim_{n\rta\infty} u_{Z^n}(t_n) = u_Z(t)$, and the direction of the $\pi/2$-cone time $t_n$ for $Z^n$ is the same as the direction of the $\pi/2$-cone time $t$ for $Z$ for sufficiently large $n$. 
\end{lem}

Note that the conditions on $Z$ of Proposition~\ref{prop-cone-conv} are a.s.\ satisfied for the correlated Brownian motion of~\eqref{eqn-bm-cov}. 

\begin{proof}[Proof of Lemma~\ref{prop-cone-conv}]
We can choose a compact interval $[a_0 ,b]\subset \BB R$ such that $t_n \in [a_0 ,b]$ for each $n\in\BB N$. By our assumption~\ref{item-Z-liminf} on $Z$, we can find $a_1 < a_0$ such that $\inf_{s\in [a_1 , a_0]} U(s) < \inf_{s\in [a_0,b]} U(s) $ and $\inf_{s\in [a_1,a_0]} V(s) < \inf_{s\in [a_0,b]} V(s)$. For sufficiently large $n$, the same is true with $(U^n ,V^n)$ in place of $(U,V)$. Therefore, we can find $a \in (-\infty , a_1]$ such that $t_n, v_{Z^n}(t_n)$, and $u_{Z^n}(t_n)$ belong to $[a,b]$ for each $n\in\BB N$. 

By local uniform convergence of $Z^n$ to $Z$, we can find $\delta>0$ such that $U(s) \geq U(t) $ and $V(s) \geq V(t)$ for each $s\in [t-\delta, t]$, so $t$ is a weak $\pi/2$-cone time for $Z$. By assumption~\ref{item-Z-weak}, $t$ is in fact a strict $\pi/2$-cone time for $Z$. 

Suppose without loss of generality that $t$ is a left $\pi/2$-cone time for $Z$, i.e.\ $V(v_Z(t)) = V(t)$. 
Let $v$ be any subsequential limit of the times $  v_{Z^n}(t_n)$. Then with $n$ restricted to our subsequence we have $  \lim_{n\rta\infty} U^n( v_{Z^n}(t_n)) = U(v)   $ and $ \lim_{n\rta\infty} V^n(  v_{Z^n}(t_n)) = V(v) $. Furthermore, $U(s) \geq U(t)$ and $V(s) \geq V(t)$ for each $s\in [v , t]$. Therefore $v \geq v_Z(t)$. We clearly have $v < t$, so since $t$ is not a right $\pi/2$-cone time for $Z$ (assumption~\ref{item-Z-direction}) we have $U(v) > U(t)$. Hence $U^n(  v_{Z^n}(t_n)) > U(t)$ for sufficiently large $n$ in our subsequence. Since $U^n(t_n) \rta U(t)$, we have $U^n(  v_{Z^n}(t_n)) > U^n(t_n)$ for sufficiently large $n$ in our subsequence. Hence $V^n(  v_{Z^n}(t_n)) = V^n(t_n)$ for sufficiently large $n$ in our subsequence. Since this holds for every choice of subsequence we infer $V^n( v_{Z^n}(t_n)) = V^n(t_n)$ for sufficiently large $n$. Moreover, for every  choice of subsequence we have $V(v) = \lim_{n\rta\infty} V^n(t_n) = V(t)$, whence $v = v_Z(t)$ and $  v_{Z^n}(t_n) \rta v_Z(t)$. 
 
Finally, let $u$ be any subsequential limit of the times $u_{Z^n}(t_n)$. Then along our subsequence we have $U(u) = \lim_{n\rta\infty} U^n(u_{Z^n}(t_n))  =\lim_{n\rta\infty} U^n(t_n) = U(t) $. Furthermore, $U(s) \geq U(t)$ for each $s\in [u , t]$. Therefore $u = u_Z(t)$. Since this holds for every such subsequential limit we obtain $\lim_{n\rta\infty}  u_{Z^n}(t_n)  = u_Z(t)$. 
\end{proof}

The following lemma is the main ingredient in the proof of Theorem~\ref{thm-cone-limit}.

\begin{figure}[ht!]
 \begin{center}
\includegraphics{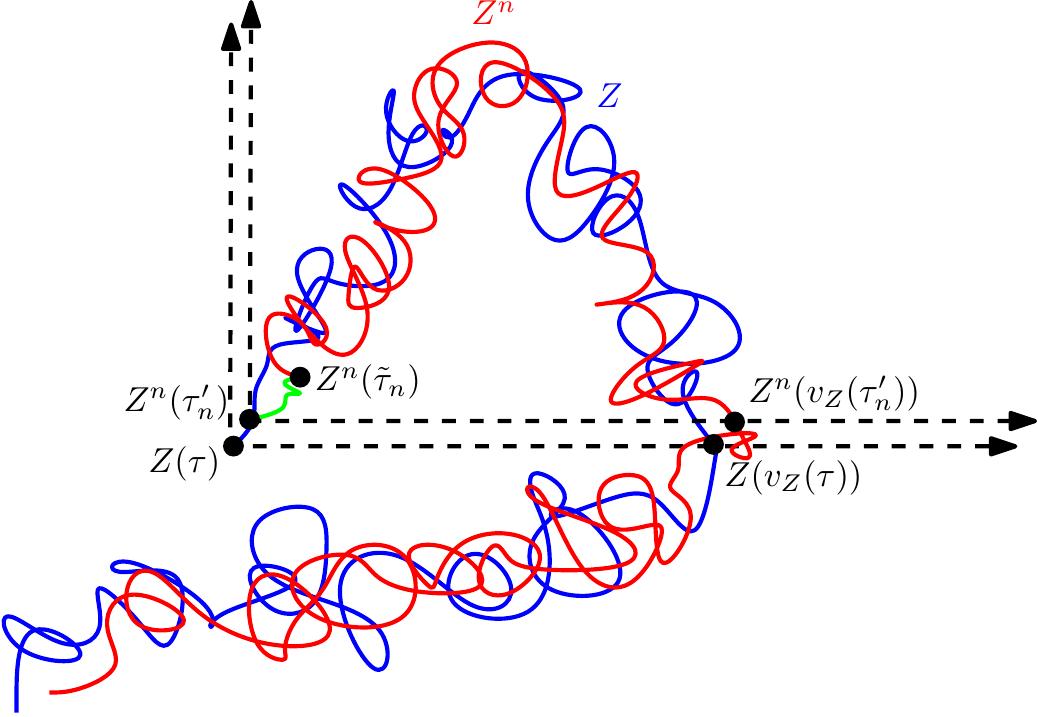} 
\caption{An illustration of the proof of Lemma~\ref{prop-F-conv}. By uniform convergence, we can find an ``approximate" $\pi/2$ cone time $\wt\tau_n$ for $Z^n$ which is close to $\tau$, and which is defined in such a way that $\wt\tau_n$ is a stopping time for the filtration generated by the word $X$. By the Markov property and Proposition~\ref{prop-late-F}, it holds with high probability that when we grow a little bit more of the path $Z^n$ (shown in green), then we arrive at a true $\pi/2$-cone time $\tau_n'$ for $Z^n$ shortly after time $\wt\tau_n$ which corresponds to a flexible order. This $\pi/2$-cone time $\tau_n'$ is close to the time $\tau_n = n^{-1}\iota_n$ which we are trying to show converges to $\tau$.} \label{fig-cone-conv}
\end{center}
\end{figure}

\begin{lem} \label{prop-F-conv}
Fix $a \in\BB R$ and $r>0$. Define the times $\tau^{a,r}$, $\iota_n^{a,r}$, and $\tau_n^{a,r}$ as in the statement of Theorem~\ref{thm-cone-limit}. 
Suppose we have (using \cite[Theorem 2.5]{shef-burger} and the Skorokhod representation theorem) coupled countably many instances $X^n$ of the infinite word $X$ with the Brownian motion $Z$ in such a way that $Z^n\rta Z$ uniformly on compacts a.s., with $Z^n$ constructed from the word $X^n$. Then $\tau_n^{a,r} \rta \tau^{a,r}$ in probability. 
\end{lem}
\begin{proof}
By translation invariance we can assume without loss of generality that $a=0$. To lighten notation, in what follows we fix $r$ and omit both $a$ and $r$ from the notation. To prove the lemma, we will define random times $\wt\iota_n , \iota_n'\ in \BB N$ and an event $G_n$ (depending on $X^n$ and $Z$) such that $\BB P(G_n) \rta 1$ and on $G_n$, $\wt\iota_n \leq \iota_n \leq \iota_n'$ and $n^{-1} \wt\iota_n$ and $n^{-1} \iota_n'$ are each close to $\tau$. 
See Figure~\ref{fig-cone-conv} for an illustration of the proof.

Let $\ep > 0$ be arbitrary.  
We observe the following.
\begin{enumerate}
\item By Proposition~\ref{prop-late-F}, we can find $\zeta_1 \in (0,\ep)$ (depending only on $\ep$) and an $\wt N \in \BB N$ such that for each $n\geq \wt N$, it holds with probability at least $1-   \ep/2$ that there is an $i\in [ \zeta_1 n    ,   \ep n ]_{\BB Z}$ such that $X_i = \tb F$ and $\phi(i) \leq 0$. Note that for such an $i$, $X(1,i)$ has no burgers. By \cite[Theorem 2.5]{shef-burger}, after possibly increasing $\wt N$ we can find $\delta_1 > 0$ (depending only on $\zeta_1)$ such that for $n\geq \wt N$, it holds with probability at least $1-  \ep$ that $X(1  ,  \zeta_1 n  )$ contains at least $\delta_1 n^{1/2}$ hamburger orders and at least $\delta_1 n^{1/2}$ cheeseburger orders. Hence with probability at least $1- \ep$, there is an $i\in [\zeta_1 n   ,   \ep n ]_{\BB Z}$ such that $X_i = \tb F$, $\phi(i) \leq 0$, and $X(1  ,i)$ contains at least $ \delta_1 n^{1/2}$ hamburger orders and at least $ \delta_1 n^{1/2}$ cheeseburger orders. \label{item-cones-late-F}
 
\item Since $\tau$ is a.s.\ finite, there is some deterministic $b > 1$ such that $\BB P(\tau < b -1) \geq 1-\ep$. \label{item-cones-b} 

\item  
For $t\geq 0$ let
\eqb \label{eqn-hat-Z-def}
\ol V(t) := V(t) - \inf_{s\in [t-r,t]} V(s) ,\quad \ol U(t) := U(t)- \inf_{s\in [t-r , t]} U(s), \quad \ol Z(t) = (\ol U(t),\ol V(t)) .
\eqe 
Note that zeros of $\ol Z$ are precisely the $\pi/2$-cone times of $Z$ in $[0,\infty)$ with $t-v_Z(t) \geq r$. 
For $\delta_2 > 0$, the sets $\ol Z^{-1}(\ol{B_{\delta_2}(0)}) \cap [0,b]$ are compact, and their intersection is $\ol Z^{-1}(0) \cap [0,b ]$. Therefore there a.s.\ exists a random $\delta_2 \in (0,1)$ such that $\ol Z^{-1}(\ol{B_{\delta_2}(0)}) \cap [0,b ] \subset B_{\zeta_1}(\ol Z^{-1}(0)) \cap [0,b ]$, i.e.\ whenever $t\in [0,b ]$ with $|\ol Z(t)| \leq \delta_2$, we have $\ol Z(s) = 0$ for some $s\in [0,b]$ with $|s-t| \leq \zeta$. We can find a deterministic $\delta_2 \in (0,1)$ such that this condition holds with probability at least $1-\ep$. \label{item-cones-excursions}

\item Set $\delta = \frac14 (\delta_1\wedge \delta_2)$. By equicontinuity we can find a deterministic $\zeta_2  \in (0,\zeta_1]$ such that with probability at least $1-\ep$, we have $|Z^n(t) - Z^n(s)| \leq \delta/2$ and $|Z(t) - Z(s)| \leq \delta/2$ whenever $t,s\in [-v_Z(\tau)-1,\tau+1]$ and $|t-s| \leq \zeta_2$. \label{item-cones-0-cont}

\item By uniform convergence, we can find a deterministic $N \in\BB N$ such that $N\geq \zeta_2^{-1} \vee \wt N$ and with probability at least $1-\ep $, we have for each $n\geq N$ that $\sup_{t\in[-r-1, b]} |Z(t) - Z^n(t)| \leq \delta/4$. \label{item-cones-conv}
\end{enumerate}
Let $E $ be the event that the events described in observations~\ref{item-cones-b} through~\ref{item-cones-conv} above hold simultaneously. Then $\BB P(E ) \geq 1-4\ep$. 

For $n\in\BB N$ let $\wt \iota_n$ be the smallest $i \in\BB N$ such that $V^n(n^{-1} i) \leq V^n(s) +  \delta$ and $U^n(n^{-1} i) \leq U^n(s) +  \delta$ for each $s\in  [n^{-1} i - r   ,n^{-1} i ] $ and let $\wt\tau_n = n^{-1} \wt\iota_n$. 
Since $\delta$ is deterministic, $\wt\iota_n$ is a stopping time for $X^n$, read forward.
We note that the defining condition for $\wt\iota_n$ is satisfied with $i = \iota_n$, so we necessarily have $\iota_n \geq \wt\iota_n$. 

We claim that if $n\geq N$, then on $E$ we have 
\eqb \label{eqn-tau-close}
\tau  - \zeta_1 \leq \wt \tau_n   \leq \tau .
\eqe 
Since $\tau$ is a $\pi/2$-cone time for $Z$ with $\tau-v_Z(\tau) \geq r$, it follows from our choice of $\zeta_2$ in observation~\ref{item-cones-0-cont} and our choice of $N$ in observation~\ref{item-cones-conv} that the condition in the definition of $\wt\iota_n$ is satisfied provided $i$ is chosen such that $n^{-1}i \in  [\tau- \zeta_2 , \tau]$ (such an $i$ must exist since $N \geq \zeta_2^{-1}$). Therefore $\wt\tau_n \leq \tau$. By our choice of $\delta$ in observation~\ref{item-cones-0-cont} and our choice of $N$ in observation~\ref{item-cones-conv} we have on $E$ (in the notation of~\eqref{eqn-hat-Z-def}) 
\eqbn
\ol V(\wt\tau_n) \leq V^n(\wt\tau_n) - \inf_{s\in [\wt \tau_n - r , \wt \tau_n]} V^n(s)  +2 \delta  \leq  \delta_2  ,
\eqen
and similarly with $\ol U$ in place of $\ol V$. 
By observation~\ref{item-cones-excursions} there exists $s \in [0,\tau+1]$ such that $|s-\wt\tau_n| \leq \zeta_1$ and $\ol Z(s) = 0$. This $s$ is a $\pi/2$-cone time for $Z$ with $s - v_Z(s) \geq r$.  
By definition, $s \geq \tau$, so $\wt \tau_n \geq s - \zeta_1 \geq \tau-\zeta_1$. This proves~\eqref{eqn-tau-close}.

Since $\wt \iota_n$ is a stopping time for $X^n$, the strong Markov property and observation~\ref{item-cones-late-F} together imply that it holds with probability at least $1-\ep$ that there exists $i\in [\wt \iota_n  +   \zeta_1 n    , \wt \iota_n  +  \ep n  ]_{\BB Z}$ such that $X_i = \tb F$, $\phi(i) \leq  \wt \iota_n  $, and $X(\wt \iota_n  + 1, i )$ contains at least $\delta_1 n^{1/2}$ hamburger orders and at least $\delta_1 n^{1/2}$ cheeseburger orders. Let $  \iota_n'$ denote the smallest such $i$ (if such an $i$ exists) and otherwise let $\iota_n' = \wt\iota_n$. For $n\in\BB N$ let $G_n$ be the event that $\iota_n' > \wt\iota_n$. Then for $n\geq N$ we have $\BB P(G_n \cap E) \geq 1- 5\ep$. 

Let $\tau_n' = n^{-1} \iota_n'$. 
By~\eqref{eqn-tau-close}, on the event $G_n \cap E$ we have $\tau_n' \geq \wt \tau_n + \zeta_1 \geq \tau  $ and $0\leq \tau_n' - \tau \leq |\wt \tau_n  -\tau| + \ep \leq 2\ep$. By combining this with~\eqref{eqn-tau-close} we obtain that if $E$ occurs (even if $G_n$ does not occur) then
\eqb \label{eqn-tau'-close}
|\tau_n' - \tau| \leq 2\ep \quad \op{and} \quad |\wt\tau_n - \tau| \leq \ep .
\eqe
Since $V^n(\wt \tau_n ) \leq V^n(s) +  \delta$ and $U^n(\wt \tau_n ) \leq U^n(s) +  \delta$ for each $s\in  [\wt \tau_n - r   , \wt \tau_n ] $ on the event $E\cap G_n$, the word $X(\wt \iota_n - r n , \wt \iota_n) $ contains at most $ \delta n^{1/2} \leq \delta_1 n^{1/2}$ burgers of each type. On $G_n$, the word $X(\wt\iota_n+1 , \iota_n' )$ contains at least $\delta_1 n^{1/2}$ hamburger orders and at least $\delta_1 n^{1/2}$ cheeseburger orders, so on $G_n \cap E$ we necessarily have $\phi(\iota_n') \leq \wt \iota_n - r n\leq \iota_n' - r n$. 
We showed above that $\iota_n \leq \iota_n$ on $E$, so on $G_n \cap E$,  
\eqb \label{eqn-iotas-compare}
\wt\iota_n \leq \iota_n \leq \iota_n '  . 
\eqe 
By~\eqref{eqn-tau'-close}, on $G_n\cap E$ we have $|\tau_n - \tau| \leq 2\ep$. Since $\BB P(G_n\cap E) \geq 1-5\ep$, we obtain the desired convergence in probability. 
\end{proof}

\begin{proof}[Proof of Theorem~\ref{thm-cone-limit}]
By \cite[Theorem 2.5]{shef-burger} and the Skorokhod representation theorem we can couple countably many instances of $X$ with $Z$ in such a way that a.s.\ $Z^n\rta Z$ uniformly on compacts. 
Define the times $\iota_n^{a,r}$ and $\tau_n^{a,r}$ as in condition~\ref{item-cone-limit-stopping} and the times $\wh\iota_n^{a,r}$, and $\wh\tau_n^{a,r}$ as in Lemma~\ref{prop-F-conv}. 
Then a.s.\  $ \tau_n^{a,r} \rta \tau^{a,r}$ as $n\rta\infty$ for each $(a,r) \in \mcl Q\times (\mcl Q\cap (0,\infty))$. It follows that the finite-dimensional marginals of the law of 
\[
\{Z^n\} \cup \left\{ \tau_n^{a,r} \,:\, (a,r) \in \mcl Q\times( \mcl Q \cap (0,\infty) )\right\}
\]
converge to those of 
\[
\{Z\} \cup \left\{ \tau^{a,r} \,:\, (a,r) \in \mcl Q\times( \mcl Q \cap (0,\infty) )\right\}
\]
as $n\rta\infty$. By the Skorokhod representation theorem, we can re-couple in such a way that $Z^n\rta Z$ uniformly on compacts and $ \tau_n^{a,r} \rta \tau^{a,r}$ a.s.\  as $n\rta\infty$ for each $a,r \in \mcl Q\times( \mcl Q \cap (0,\infty) )$. Henceforth fix such a coupling. By definition, in any such coupling conditions~\ref{item-cone-limit-Z} and~\ref{item-cone-limit-stopping} in the theorem statement are satisfied. We must verify conditions~\ref{item-cone-limit-maximal} and~\ref{item-cone-limit-times} for this coupling. 

We start with condition~\ref{item-cone-limit-times}. Suppose given sequences $n_k\rta\infty$ and $\{i_{n_k}\}_{k\in\BB N}$ with $n_k^{-1} i_{n_k} \rta t$ as in condition~\ref{item-cone-limit-times}. By Lemma~\ref{prop-cone-conv}, it is a.s.\ the case that $t$ is a $\pi/2$-cone time for $Z$ and we a.s.\ have $v_{Z^{n_k}}(n_k^{-1} i_{n_k}) \rta v_Z(t)$ and $u_{Z^{n_k}}(n_k^{-1} i_{n_k}) \rta u_Z(t)$. Since also $n_k^{-1}( i_{n_k} - 1) \rta t$ and $v_{Z^n}(n_k^{-1}(i_{n_k} -1)) = n_k^{-1} \phi^{n_k} (i_{n_k})$ (recall the discussion just after Definition~\ref{def-cone-time}), we infer that $n_k^{-1} \phi^{n_k}(i_{n_k}) \rta v_Z(t)$. 
The time $n_k u_{Z^{n_k}}(n_k^{-1}( i_{n_k} - 1))$ coincides with the largest $j < \phi^{n_k}_* (i_{n_k})$ for which the reduced word $X^{n_k}(j ,  \phi^{n_k}_*(i_{n_k}))$ contains a burger of the type opposite $X^{n_k}_{\phi^{n_k} (i_{n_k})}$. 
For each $\ep > 0$, if $t$ is a right $\pi/2$-cone time then there exists $\delta > 0 $ for which $V(s) \geq V(u_Z(t)) + \delta$ for each $s\in [u_Z(t) + \ep , t-\ep]$ and if $t$ is a left $\pi/2$-cone time the same holds with $U$ in place of $V$. Since $Z^{n_k} \rta Z$ uniformly on compacts, we infer that $\lim_{k\rta\infty} n_k^{-1} (u_{Z^{n_k}}(n_k^{-1}( i_{n_k} - 1)) - \phi_*^{n_k}(i_{n_k})) = 0$ so $n_k^{-1} \phi_*^{n_k}(i_{n_k}) \rta u_Z(t)$. 

Now we turn our attention to condition~\ref{item-cone-limit-maximal}. 
Fix a bounded open interval $I\subset \BB R$ with endpoints in $\mcl Q$, $a\in I\cap \mcl Q$, and $\ep > 0$. Let $t$, $i_n$, and $t_n$ be as in condition~\ref{item-cone-limit-maximal}. Since $t \not= a$ a.s., we can a.s.\ find $r \in \mcl Q \cap (0,\infty)$ (random and depending on $\ep$) such that 
$t \in [\tau^{a,r}   , \tau^{a,r} + \ep]$ and $v_Z(t) \in [v_Z(\tau^{a,r}) - \ep, v_Z(\tau^{a,r})]$ (in particular, we choose $r$ slightly smaller than $t-v_Z(t)$).
By condition~\ref{item-cone-limit-stopping}, we a.s.\ have $n^{-1} \iota_n^{a,r} \rta \tau^{a,r}$ as $n\rta\infty$. By condition~\ref{item-cone-limit-times}, we a.s.\ have $n^{-1}\phi( \iota_n^{a,r} ) \rta v_Z(\tau^{a,r})$ as $n\rta\infty$. Since $I$ is open and a.s.\ neither $t $ nor $v_Z(t)$ is equal to $a$, if we choose $\ep$ sufficiently small (random and depending on $a$ and $I$) then it is a.s.\ the case that for sufficiently large $n\in\BB N$, $a n \in [\phi( \iota_n^{a,r}  ),  \iota_n^{a,r}] \subset n I$. Hence for sufficiently large $n\in\BB N$, we have $t_n \geq n^{-1} \iota_n^{a,r} \geq t- 2\ep$. Since $\ep$ is arbitrary, a.s.\ $\liminf_{n\rta\infty} t_n \geq t$. Similarly $\limsup_{n\rta\infty } v_{Z^n}(t_n) \leq v_Z(t)$. 

To show that $\lim_{n\rta\infty} t_n= t$, we observe that from any sequence of integers tending to $\infty$, we can extract a subsequence $n_j\rta\infty$ and a $t' \in \ol I$ such that $t_{n_j} \rta t'$. Our result above implies that $[v_Z(t) , t]\subset [v_Z(t') , t']$. 
Since $\liminf_{j \rta\infty}( t_{n_j} - v_{Z^{n_j}}(t_{n_j}) )\geq t - v_Z(t)$, condition~\ref{item-cone-limit-times} implies that $ t'$ is a $\pi/2$-cone time for $Z$ with $[v_Z(t') ,t'] \subset \ol I$. Since $I$ has endpoints in $\mcl Q$ it is a.s.\ the case that neither of these endpoints is a $\pi/2$-cone time for $Z$ or $v_Z$ of a $\pi/2$-cone time for $Z$, simultaneously for all choices of $I$. Hence in fact $[v_Z(t' ) , t'] \subset I$ for every such choice of subsequence. By maximality $t' = t$. Thus $t_n \rta t$.
\end{proof}

\subsection{Convergence of the cone times conditioned on no burgers}
\label{sec-cone-cond}

For the sake of completeness, in this subsection we will state and prove a corollary to the effect that Theorem~\ref{thm-cone-limit} remains true if we condition on $\{J>n\}$, where as per usual $J$ is the smallest $j\in\BB N$ for which $X(-j,-1)$ contains a burger. This corollary will be used in the subsequent paper~\cite{gms-burger-finite}.

\begin{cor} \label{prop-cone-limit-no-burger}
Let $\wh Z = (\wh U , \wh V)$ be a correlated two-dimensional Brownian motion as in~\eqref{eqn-bm-cov}, defined on $(-\infty,0]$ and conditioned to stay in the first quadrant until time $-1$ when run backward. For $n \in\BB N$, let $\wh X^n $ be sampled according to the conditional law of the word $\dots X_{-1} X_0$ given $\{ J  >n\}$ and let $\wh Z^n : (-\infty , 0]$ be the path~\eqref{eqn-Z^n-def} corresponding to $\wh X^n$. Fix a countable dense set $\mcl Q\subset \BB R$. Fix a countable dense set $\mcl Q\subset \BB R$. There is a coupling of $\{\wh X^n\}_{n\in\BB N}$ with $Z$ such that when $Z^n$, $\phi^n$, and $\phi_*^n$ are defined as in~\eqref{eqn-Z^n-def} and Notation~\ref{def-match-function}, respectively, with $\wh X^n$ in place of $X$, the following holds a.s.  
\begin{enumerate}
\item $\wh Z^n \rta \wh Z$ uniformly on compact subsets of $(-\infty,0]$. \label{item-cone-limit-Z'} 
\item Suppose given a bounded open interval $I \subset (-\infty, 0)$ with endpoints in $\mcl Q$ and $a \in I \cap \mcl Q$. Let $t$ be the maximal (Definition~\ref{def-maximal}) $\pi/2$-cone time for $\wh Z$ in $I$ with $a\in [ v_{\wh Z}(t),t]$. For $n\in\BB N$, let $i_n$ be the maximal flexible order time (with respect to $X^n$) $i$ in $n I$ with $a n \in [\phi^n(i)   , i ]$ (or $i_n = \lfloor a n \rfloor$ if no such $i$ exists); and let $t_n := n^{-1} i_n$. Then $t_n \rta t$.  
\label{item-cone-limit-maximal'}  
\item For $r  > 0$ and $a\in (-\infty, 0)$, let $\wh \tau^{a,r}$ be the minimum of 0 and the smallest $\pi/2$-cone time $t$ for $\wh Z$ such that $t\geq a$ and $t - v_{\wh Z}(t) \geq r$. For $n\in\BB N$, let $\wh \iota_n^{a,r}$ be the minimum of 0 and the smallest $i\in\BB Z$ such that $\wh X^n_i = \tb F$, $i \geq a n$, and $i - \phi^n(i) \geq r n - 1$; and let $\tau_n^{a,r} := n^{-1} \iota_n^{a,r}$. Then $\tau_n^{a,r} \rta \iota_n^{a,r}$ 
for each $(a,r) \in (\mcl Q\cap (-\infty,0)   ) \times  (\mcl Q\cap(0,\infty) )$. \label{item-cone-limit-stopping'}
\item For each sequence of positive integers $n_k \rta \infty$ and each sequence $\{i_{n_k} \}_{k\in\BB N}$ such that $\wh X^{n_k}_{i_{n_k}} = \tb F$ for each $k\in\BB N$, $n_k^{-1} i_{n_k} \rta t \in \BB R$, and $\liminf_{k\rta\infty} (i_{n_k} - \phi^{n_k}(i_{n_k}) ) > 0$, it holds that $t$ is a $\pi/2$-cone time for $Z$ which is in the same direction as the $\pi/2$-cone time $n_k^{-1} i_{n_k}$ for $Z^{n_k}$ for large enough $k$ and in the notation of Definition~\ref{def-cone-time}, we have
\eqbn
\left(n_k^{-1} i_{n_k} , \, n_k^{-1} \phi^{n_k}(i_{n_k}) , \, n_k^{-1} \phi^{n_k}_*(i_{n_k}) \right) \rta \left( t , v_Z(t) , u_Z(t) \right) .
\eqen 
\label{item-cone-limit-times'}
\end{enumerate}
\end{cor}
\begin{proof}
We will prove that we can choose a coupling such that a.s.\ $\wh Z^n \rta \wh Z$ and $n^{-1} \iota_n^{a,r}  \rta  \tau^{a,r}$ for each  $(a,r) \in (\mcl Q\cap (-\infty,0)   ) \times  (\mcl Q\cap(0,\infty) )$. It follows as in the proof of Theorem~\ref{thm-cone-limit} that such a coupling also satisfies the other conditions in the statement of the corollary. 

Fix $\zeta \in (0,1)$. For $n\in \BB N$, define
\alb
&\ol X^{n,\zeta} :=  \dots \wh X^n_{-2-\lfloor \zeta n \rfloor} \wh X^n_{-1-\lfloor \zeta n \rfloor}    ,\quad   \ul X^{n,\zeta} := \wh X^n_{-\lfloor \zeta n \rfloor} \dots \wh X^n_{0} , \\
 &\ol Z^\zeta = (\ol U^\zeta , \ol V^\zeta) := (\wh Z - \wh Z(-\zeta))|_{(-\infty,-\zeta]} ,\quad \ul Z^\zeta  = (\ul U^\zeta , \ul V^\zeta)  := \wh Z|_{[-\zeta,0]} , \quad \op{and} \\
&\ol Z^{n,\zeta} = (\ol U^{n,\zeta}, \ol V^{n,\zeta}) := \left(\wh Z^n - \wh Z^n\left(-\zeta \right) \right)|_{\left(-\infty ,  -\zeta \right]} .
\ale
Also let $\wh D_\zeta^n$ be as in the proof of Theorem~\ref{thm-no-burger-conv}, i.e.\ $\wh D_\zeta^n$ is the path defined in the same manner as the path $D$ of~\eqref{eqn-discrete-paths} in Section~\ref{sec-burger-prelim} but with the following modification: if $j\in [-\zeta n,  -1]_{\BB Z}$, $\wh X^n_j = \tb F$, and $-j$ does not have a match in $\mcl R\left(\ul X^{n,\zeta} \right)$, then $\wh D_\zeta^n(-j) - \wh D_\zeta^n(-j+1)$ is equal to zero rather than $( 1,0)$ or $(0, 1)$. Extend $\wh D_\zeta^n$ to $[-\zeta n ,0]$ by linear interpolation. For $t\in [-\zeta,0]$, let 
$\ul Z^{n,\zeta}(t) = ( \ul U^{n,\zeta}(t) , \ul V^{n,\zeta}(t))  := n^{-1/2} \wh D_\zeta(n t)$. Then $\ul Z^{n,\zeta}$ is determined by $\ul X^{n,\zeta} $.

By Theorem~\ref{thm-no-burger-conv} and the Skorokhod representation theorem, we can find a coupling of the sequence of words $\{ \ul X^{n,\zeta} \}_{n\in\BB N}$ with $\ul Z^\zeta$ such that $\ul Z^{n,\zeta} \rta \ul Z^\zeta$ a.s. 
For each $n\in\BB N$ and each realization $x$ of $\ul X^{n,\zeta}$, the conditional law of $\ol Z^{\zeta, n}$ given $\{ \ul X^{n,\zeta} =x \}$ is the same as the conditional law of $\dots  X_{-2} X_{-1} $ given the event
\eqbn
G_n(x) := \left\{ \text{$\mcl R\left(X(-n +\lfloor \zeta n \rfloor , -1) x   \right)$ contains no burgers}   \right\} .
\eqen
By \cite[Theorem 2.5]{shef-burger}, Lemma~\ref{prop-few-F}, and our choice of coupling (see also the proof of Lemma~\ref{prop-Z-cond-limit}), the conditional law of $\ol Z^{n,\zeta}$ given $\ul X^{n,\zeta}$ converges a.s.\ to the conditional law of $\ol Z^\zeta$ given $ \wt G^\zeta(\ul U^\zeta(-\zeta), \ul V^\zeta(-\zeta) )$, with $\wt G^\zeta(\cdot,\cdot)$ as in~\eqref{eqn-tilde-G^s-def}. We note that $\wh\tau_n^{a,r} \wedge (-\zeta)$ (resp. $\wh\tau^{a,r}\wedge (-\zeta)$) is determined by $\ol X^{n,\zeta}$ (resp. $\ol Z^\zeta$), so it follows from Theorem~\ref{thm-cone-limit} and the argument of Lemma~\ref{prop-Z-cond-limit} that in fact the finite dimensional marginals of the joint conditional law given $\ul X^{n,\zeta}$ of 
\eqbn
\{\ol Z^{n,\zeta}\} \cup \left\{ \wh\tau_n^{a,r} \wedge (-\zeta) \,:\, (a,r) \in   (\mcl Q\cap (-\infty,0)   ) \times  (\mcl Q\cap(0,\infty) )  \right\}
\eqen
converge a.s.\ to the finite dimensional marginals of the joint conditional law given $\wt G^\zeta(\ul U^\zeta(-\zeta), \ul V^\zeta(-\zeta) )$ of 
\eqbn
\{\ol Z^{ \zeta}\} \cup \left\{ \wh\tau^{a,r} \wedge (-\zeta) \,:\, (a,r) \in  (\mcl Q\cap (-\infty,0)   ) \times  (\mcl Q\cap(0,\infty) )   \right\} .
\eqen
Therefore, finite dimensional marginals of the the joint law of 
\eqbn
\{ \wh Z^n\} \cup \left\{ \wh\tau_n^{a,r} \wedge (-\zeta) \,:\, (a,r) \in  (\mcl Q\cap (-\infty,0)   ) \times  (\mcl Q\cap(0,\infty) )   \right\}
\eqen
converge to finite dimensional marginals of the the joint law of
\eqbn
\{  \wh Z \} \cup \left\{ \wh\tau^{a,r} \wedge (-\zeta) \,:\, (a,r) \in  (\mcl Q\cap (-\infty,0)   ) \times  (\mcl Q\cap(0,\infty) )    \right\} .
\eqen
Since $\zeta$ is arbitrary and $|\wh\tau_n^{a,r} \wedge (-\zeta) - \wh\tau_n^{a,r}|  $ and $\wh\tau^{a,r} \wedge (-\zeta) - \wh\tau^{a,r}|$ are each at most $\zeta$, the same holds if we don't truncate at $-\zeta$. 

We now obtain a coupling such that a.s.\ $\wh Z^n \rta \wh Z$ and $n^{-1} \iota_n^{a,r}  \rta  \tau^{a,r}$ for each  $(a,r) \in (\mcl Q\cap (-\infty,0)   ) \times  (\mcl Q\cap(0,\infty) )$ by means of the Skorokhod theorem, and conclude as in the proof of Theorem~\ref{thm-cone-limit}.
\end{proof}

\appendix

\section{Results for times with no orders}
\label{sec-no-order}

In this appendix, we will explain how to adapt the proofs found in Sections~\ref{sec-F-reg},~\ref{sec-no-burger}, and~\ref{sec-cone-conv} to obtain analogues of the results of those sections when we consider the event that $X(1,n)$ contains no orders, rather than the event that $X(-n,-1)$ contains no burgers. Although the results of this appendix are not needed for the proof of Theorem~\ref{thm-cone-limit}, they are of independent interest and will be needed in the sequels to this work~\cite{gms-burger-local,gms-burger-finite}. 

In Section~\ref{sec-no-order-conv} we will consider an analogue of Theorem~\ref{thm-no-burger-conv} with no orders rather than no burgers and in Section~\ref{sec-no-order-reg-var} we will prove some regular variation estimates. In Section~\ref{sec-cone-limit-forward}, we will consider a generalization of Theorem~\ref{thm-cone-limit}. 

Throughout this section, we let $I$ denote the smallest $i \in \BB N$ for which $X(1,i)$ contains an order. 

\subsection{Convergence conditioned on no orders}
\label{sec-no-order-conv}

In this subsection we will explain how to adapt the arguments of Sections~\ref{sec-F-reg} and~\ref{sec-no-burger} to obtain the following result. 

\begin{thm} \label{thm-no-order-conv}
As $n\rta \infty$, the conditional law of the path $Z^n|_{[0,1]}$ defined in~\eqref{eqn-Z^n-def} given $\{I > n\}$ (i.e.\ the event that $X(1,n)$ contains no orders) converges to the law of a correlated Brownian motion as in~\eqref{eqn-bm-cov} conditioned to stay in the first quadrant until time 1. 
\end{thm} 

The first step in the proof of Theorem~\ref{thm-no-order-conv} is to establish an exact analogue of Proposition~\ref{prop-F-regularity}, which reads as follows.

\begin{prop} \label{prop-ancestor-reg}
For $\ep > 0$, let $E_n(\ep)$ be the event that $X(1,n)$ contains at least $\ep n^{1/2}$ burgers of each type. Then we have
\eqbn
\lim_{\ep\rta 0} \liminf_{n \rta \infty} \BB P\left(E_n(\ep) \,|\, I > n\right) = 1.
\eqen
\end{prop}

To adapt the proof of Proposition~\ref{prop-F-regularity} in order to obtain Proposition~\ref{prop-ancestor-reg}, one needs an appropriate analogue of the times $i \in\BB Z$ with $X_i = \tb F$. 
 
\begin{defn} \label{def-ancestor}
Say that $i \in \BB Z$ is a \emph{pre-burger time} if $X_{i+1}$ is a burger. For a pre-burger time $i$, we write $\ol\phi(i)$ for the smallest $j \geq i+1$ for which $X(i+1,j)$ contains an order. 
\end{defn} 

Suppose $i$ is a pre-burger time. We observe the following
\begin{enumerate}
\item The word $X(i+1, \ol\phi(i))$ contains a single order and some number of burgers, all of the same type. If the single order is a $\tb F$, there are no burgers. Otherwise, the burgers are of the type opposite the order. \label{item-ancestor-reduced}
\item We need not have $\ol\phi(i) \in \{\phi(i) , \phi(i+1)\}$. To see this consider the word $X_1\dots X_5 = \tb C \tc H \tc C \tb C \tb C$. Here 1 is a pre-burger time and $\ol\phi(1) = 5$.  
\item If $i'$ is another pre-burger time with $i' \in (i , \ol \phi(i))_{\BB Z}$, then $\ol\phi(i') \in [i+1,\ol\phi(i)]_{\BB Z}$. To see this, we observe that $X_{\ol\phi(i)}$ is an order whose match is at a time before $i+1$ (and hence also before $i'+1$), whence $X(i'+1 , \ol\phi(i))$ contains an order. Note, however, that we can have $\ol\phi(i') = \ol\phi(i)$, which does not happen for nested $\tb F$-excursions. \label{item-ancestor-nested}
\item The time $n^{-1} i$ is a weak forward $\pi/2$-cone time for $Z^n$, as defined just below, and $\ol v_{Z^n}(n^{-1} i) = n^{-1} ( \ol\phi(i) - 1)$.\label{item-ancestor-cone}
\end{enumerate}

\begin{defn}\label{def-cone-time-forward}
A time $t$ is called a \emph{(weak) forward $\pi/2$-cone time} for a function $Z = (U,V) : \BB R \rta \BB R^2$ if there exists $t' > t$ such that $U_s \geq U_{t }$ and $V_s \geq V_{t }$ for $s\in  [t    , t']$. Equivalently, $Z([t,t'])$ is contained in the cone $Z_{t } + \{z\in \BB C : \op{arg} z \in [0,\pi/2]\}$. We write $ \ol v_Z(t)$ for the supremum of the times $t'$ for which this condition is satisfied, i.e. $\ol v_Z(t)$ is the exit time from the cone. We write $\ol u_Z(t)$ for the infimum of the times $t_* > t$ for which $\inf_{s\in [t,t_*]} U_s < U_t$ and $\inf_{s\in [t,t_*]} V_s < V_t$. 
\end{defn} 

Note that a forward $\pi/2$-cone time for $Z$ is a $\pi/2$-cone time in the sense of Definition~\ref{def-cone-time} for the time reversal of $Z$.

The following is the analogue of Lemma~\ref{prop-F-symmetry} for the case of no orders, rather than no burgers.

\begin{lem} \label{prop-ancestor-symmetry}
For $n\in\BB N$, let $P_n$ be the largest $k\in[1, n]_{\BB Z}$ for which $X(-k,-1)$ contains no orders (or $P_n = n+1$ if no such $j$ exists). For $\ep \geq 0$, let $A_n(\ep)$ be the event that $P_n < n+1$ and $P_n \leq (1-\ep) (\ol \phi(-P_n)+ P_n)$. There exists $\ep_0 > 0$, $n_0\in\BB N$, and $q_0 \in (0,1/3)$ such that for each $\ep \in (0,\ep_0]$ and $n\geq n_0$, 
\eqbn
\BB P\left(   A_n(\ep) \right)  \geq 3 q_0 .
\eqen
\end{lem} 

In light of observations~\ref{item-ancestor-reduced} through~\ref{item-ancestor-cone} above, Lemma~\ref{prop-ancestor-symmetry} can be proven via an argument which is nearly identical to the proof of Lemma~\ref{prop-F-symmetry}, except that one reads the word $X$ backward and considers maximal discrete intervals of the form $[k,\ol\phi(k)]_{\BB Z}$ with $k$ a pre-burger time instead of maximal $\tb F$-excursions. Note that there are a.s.\ infinitely many such intervals containing any fixed $i\in\BB Z$ by Proposition~\ref{prop-infinite-F}.

Using Lemma~\ref{prop-ancestor-symmetry} and almost exactly the same argument which appears in Section~\ref{sec-F-reg-forward} one obtains the existence of a sequence of positive integers $m_j \rta \infty$ and an $\ep  >0$ such that (in the notation of Proposition~\ref{prop-ancestor-reg})
\eqbn
 \liminf_{j\rta\infty} \BB P\left(E_{m_j}(\ep) \,|\, I > m_j\right)  > 0 .
\eqen
This, in turn, leads to a proof of Proposition~\ref{prop-ancestor-reg} by means of the inductive argument of Section~\ref{sec-F-reg-reverse}, but with the word $X$ read forward rather than backward. We note that the obvious analogue of Lemma~\ref{prop-Z-cond-limit} holds in this setting, with a similar but slightly easier proof (since the initial segment of the word one is conditioning on contains no flexible orders). 

With Proposition~\ref{prop-ancestor-reg} established, the argument of Section~\ref{sec-no-burger} carries over more or less verbatim to yield Theorem~\ref{thm-no-order-conv}. The only difference is that the word is read in the opposite direction and times $j$ for which $X(-j,-1)$ contains no orders are used in place of times $i$ for which $X(1,i)$ contains no burgers.

\subsection{Regular variation for times with no orders}
\label{sec-no-order-reg-var}

In this subsection we will prove analogues of some of the results of Section~\ref{sec-reg-var} for times when the word has no orders, rather than no burgers. Recall the definition of regular variation from Section~\ref{sec-reg-var}. 
 
\begin{lem} \label{prop-I-reg-var}
Let $I$ be defined as in the beginning of this appendix. Then the law of $I$ is regularly varying with exponent $\mu$ (defined as in~\eqref{eqn-cone-exponent}).
\end{lem}
\begin{proof}
This follows from Theorem~\ref{thm-no-order-conv} and the results of \cite{shimura-cone} via exactly the same argument used in the proof of Proposition~\ref{prop-J-reg-var}. 
\end{proof}

Borrowing some terminology from~\cite{wedges}, we say that $i\in \BB N$ is \emph{ancestor free}
if there is no $k\in [1,i ]_{\BB Z}$ such that $X(k,i)$ contains no orders. Equivalently, $X(i-j, i)$ contains an order for every $j\in [0,i-1]_{\BB Z}$; or there is no pre-burger time (Definition~\ref{def-ancestor}) $k\leq i-1$ such that $i \in [k+1,\ol\phi(k)]_{\BB Z}$. The ancestor free times can be described as follows. 
 
\begin{lem} \label{prop-I-af}
Let $I_1 = I$ be the the smallest $i\in\BB N$ for which $X(1,i)$ contains an order. Inductively, for $m\geq 2$ let $I_m$ be the smallest $i\geq I_{m-1}+1$ for which $X(I_{m-1}+1 , i)$ contains an order. Then $I_m$ is precisely the $m$th smallest ancestor free time in $\BB N$.
\end{lem}
\begin{proof}
Let $I_0 = \wt I_0 = 0$. For $m\in\BB N$, let $\wt I_m$ denote the $m$th smallest ancestor free time. We must show $\wt I_m = I_m$ for each $m\geq 0$. We prove this by induction, starting with the trivial base case $m=0$. Now suppose $m\geq 1$ and we have shown $\wt I_{m-1} = I_{m-1}$. By definition, $X_{I_m}$ is an order whose match is $\leq I_{m-1}$. Hence this order appears in $X(j , I_m)$ for each $j\in [ I_{m-1} +1 , I_m]_{\BB Z}$. By the inductive hypothesis, $I_{m-1}$ is ancestor free, so $X(j,I_{m-1})$ contains an order for each $j\in [1,I_{m-1}]_{\BB Z}$. Therefore $X(j , I_m)$ contains an order for each $j \in [1,I_m]_{\BB Z}$, so $I_m$ is ancestor free and $I_m \geq \wt I_m$. 

Conversely, since $\wt I_m$ is ancestor free, the word $X(\wt I_{m-1} + 1 , \wt I_m)$ contains an order. Since $\wt I_{m-1}  = I_{m-1}$ (by the inductive hypothesis) it follows that $\wt I_m \geq I_m$.
\end{proof}
 
The following is the analogue of Corollary~\ref{prop-K-reg-var} for times with no orders.

\begin{lem} \label{prop-P-reg-var}
Let $P$ be the smallest $j \in \BB N$ for which $X(-j,-1)$ contains no orders. Then the law of $P$ is regularly varying with exponent $1-\mu$, with $\mu$ as in~\eqref{eqn-cone-exponent}. 
\end{lem} 
\begin{proof}
Define the times $I_m$ for $m\in\BB N$ as in Lemma~\ref{prop-I-af}. For $n\in\BB N$, let $M_n$ be the largest $m\in\BB N$ for which $I_m \leq n$. For $l\in\BB N$, let $P_l$ be the $l$th smallest $j\in\BB N$ for which $X(-j,-1)$ contains no orders and let $L_n$ be the largest $l\in\BB N$ for which $P_l\leq n$.

By Lemma~\ref{prop-I-af}, for $k\in\BB N$ the event $\{I_{M_n} = k\}$ is the same as the event that $k$ is ancestor free, i.e.\ $X(j , k)$ contains an order for each $j\in [1,k]_{\BB Z}$; and $I_{M_n+1}  > n$, i.e.\ $X(k+1,n)$ contains no orders. 
The event $\{P_{L_n} = k' \}$ is the same as the event that $X(-k' ,-1)$ contains no orders; and $X(-j,-k' )$ contains an order for each $j \in [k'  +1 ,n]_{\BB Z}$. By translation invariance, 
\eqbn
P_{L_n} \eqD n - I_{M_n}   .
\eqen

By Lemma~\ref{prop-I-reg-var} and the Dynkin-Lamperti theorem~\cite{dynkin-limits,lamperti-limits}, it follows that $n^{-1}(n - I_{M_n})$ converges in law to the generalized arcsine distribution with parameter $\mu$ as $n\rta\infty$. Therefore $n^{-1}(n - P_{L_n})$ converges in law to a generalized arcsine distribution with parameter $1-\mu$. By the converse to the Dynkin-Lamperti theorem, we obtain the statement of the lemma. 
\end{proof}

\subsection{Convergence of the forward cone times}
\label{sec-cone-limit-forward}

In this subsection we record a generalization of Theorem~\ref{thm-cone-limit} which includes convergence of the times of this subsection to the forward $\pi/2$-cone times of the correlated Brownian motion $Z$. We first need the following analogue of Definition~\ref{def-maximal}.

\begin{defn}\label{def-maximal-forward}
A forward $\pi/2$-cone time for a path $Z$ is called a \emph{maximal forward $\pi/2$-cone time} in an (open or closed) interval $I\subset \BB R$ if $[t, \ol v_Z(t)] \subset I$ and there is no forward $\pi/2$-cone time $t'$ for $Z$ such that $[t' , \ol v_Z(t')  ]\subset I$ and $[ t , \ol v_Z(t)] \subset (  t' , \ol v_Z(t'))$. Equivalently, $-t$ is a maximal $\pi/2$-cone time for $Z(-\cdot)$ (Definition~\ref{def-maximal}). An integer $i\in \BB Z$ is called a \emph{maximal pre-burger time} in an interval $I\subset \BB R$ if $i$ is a pre-burger time (Definition~\ref{def-ancestor}), $[   i ,\ol\phi(i) ]_{\BB Z} \subset I$, and there is no pre-burger time $i' \in \BB Z$ with $[i, \ol\phi(i)  ]_{\BB Z} \subset (i',  \ol \phi(i')   )_{\BB Z}$ and $[i', \ol\phi(i')  ]_{\BB Z} \subset I$. 
\end{defn}

\begin{thm} \label{thm-cone-limit-forward}
Let $Z$ be a correlated Brownian motion as in~\eqref{eqn-bm-cov}. There is a coupling of countably many instances $\{ X^n \}_{n\in\BB N} $ of the bi-infinite word $X$ described in Section~\ref{sec-burger-prelim} with $Z$ such that the conditions of Theorem~\ref{thm-cone-limit} are satisfied and the following additional conditions hold a.s. 
\begin{enumerate}
\setcounter{enumi}{5}  
\item Suppose given a bounded open interval $I \subset \BB R$ with endpoints in $\mcl Q$ and $a \in I \cap \mcl Q$. Let $\ol t$ be the maximal forward $\pi/2$-cone time for $Z$ in $I$ with $a\in [\ol t, \ol v_Z(\ol t)]$. For $n\in\BB N$, let $\ol i_n$ be the maximal pre-burger time $i$ (with respect to $X^n$) in $n I$ with $a n \in [ i , \ol\phi(i) ]$ (or $\ol i_n = \lfloor a n \rfloor$ if no such $i$ exists); and let $\ol t_n = n^{-1} \ol i_n$. Then 
\eqbn
\left( \ol t_n , v_{Z^n} (\ol t_n) , u_{Z^n}(\ol t_n) \right) \rta \left( \ol t , v_Z(\ol t) ,u_Z(\ol t) \right) .
\eqen
\label{item-cone-limit-maximal''}
\item For $r  > 0$ and $a\in \BB R$, let $\ol\tau^{a,r}$ be the greatest forward $\pi/2$-cone time $t$ for $Z$ such that $t\leq a$ and $\ol v_Z(t) - t \geq r$. For $n\in\BB N$, let $\ol\iota_n^{a,r}$ be the greatest pre-burger time $i\in\BB Z$ such that $i \leq a n$ and $\ol\phi(i) - i \geq r n - 1$ (or $\ol\iota_n^{a,r} = - \infty$ if no such $i$ exists); and let $\ol\tau_n^{a,r} = n^{-1} \ol\iota_n^{a,r}$. 
Then 
\eqbn
\left( \ol \tau_n^{a,r}  , v_{Z^n} (\ol \tau_n^{a,r}) , u_{Z^n}(\ol \tau_n^{a,r}) \right) \rta \left( \ol \tau^{a,r} , v_Z(\ol \tau^{a,r}) ,u_Z(\ol \tau^{a,r}) \right)  
\eqen
for each $(a,r) \in \mcl Q\times (\mcl Q\cap (0,\infty))$. \label{item-cone-limit-stopping''}
\item For each sequence of positive integers $n_k \rta \infty$ and each sequence $\{i_{n_k} \}_{k\in\BB N}$ such that $X^{n_k}_{i_{n_k}+1}  $ is a burger for each $k\in\BB N$, $n_k^{-1} i_{n_k} \rta t \in \BB R$, and $\liminf_{k\rta\infty} (  \ol v_{Z^{n_k}}(n_k^{-1} i_{n_k}) -  n_k^{-1} i_{n_k} ) > 0$, it holds that $t$ is a forward $\pi/2$-cone time for $Z$ and in the notation of Definition~\ref{def-cone-time-forward}, we have
\eqbn
\left(n_k^{-1} i_{n_k} , \,  \ol v_{Z^{n_k}}(n_k^{-1} i_{n_k}), \,  \ol u_{Z^{n_k}}(n_k^{-1} i_{n_k}) \right) \rta \left( t , v_Z(t) , u_Z(t) \right) .
\eqen
\end{enumerate}
\end{thm}
\begin{proof}
From Lemma~\ref{prop-I-reg-var}, the Dynkin-Lamperti theorem, and the same argument used in the proof of Lemma~\ref{prop-F-conv}, one obtains an analogue of the latter lemma with the times $\ol\tau^{a,r}$ and $\ol\iota_n^{a,r}$ in place of the times $\tau^{a,r}$ and $\iota_n^{a,r}$. From this, Lemma~\ref{prop-cone-conv}, Lemma~\ref{prop-F-conv}, and the Skorokhod theorem, we infer that we can find a coupling of the sequence $(X^n)$ with the path $Z$ such that conditions~\ref{item-cone-limit-Z} and~\ref{item-cone-limit-stopping} of Theorem~\ref{thm-cone-limit} and condition~\ref{item-cone-limit-stopping''} of the present theorem hold simultaneously a.s. The rest of the theorem now follows from exactly the same argument given in the proof of Theorem~\ref{thm-cone-limit}.
\end{proof}

\begin{remark}
One can also obtain versions of Corollary~\ref{prop-cone-limit-no-burger} in the setting of this appendix, i.e.\ the natural analogues of Theorem~\ref{thm-cone-limit-forward} hold when we condition on the event that $X(-n,-1)$ contains no burgers (resp. $X(1,n)$ contains no orders) and consider only negative (resp. positive) time.  
\end{remark}

\section{Index of symbols}
\label{sec-index}

Here we record some commonly used symbols in the paper, along with the location where they are first defined. 

\begin{multicols}{2}
\begin{itemize}
\item $\mcl R(\cdot)$; Section~\ref{sec-burger-prelim}.
\item $|x|$; Section~\ref{sec-burger-prelim}.
\item $p$; Section~\ref{sec-burger-prelim}. 
\item $X$; Section~\ref{sec-burger-prelim}.
\item $X(a,b)$;~\eqref{eqn-X(a,b)}.
\item $\phi$; Notation~\ref{def-match-function}. 
\item $\phi_*$; Notation~\ref{def-match-function}. 
\item $\mcl N_{\tb F}(x)$, etc.; Definition~\ref{def-theta-count}. 
\item $D = (d,d^*)$; Definition~\ref{def-theta-count}.
\item $Z^n = (U^n ,V^n)$;~\eqref{eqn-Z^n-def}.
\item $Z = (U,V)$;~\eqref{eqn-bm-cov}. 
\item $v_Z(t)$ and $u_Z(t)$; Definition~\ref{def-cone-time}.
\item $\ell_j^n$; Section~\ref{sec-loop-limit}.
\item $M_j^{n,\infty}$; Section~\ref{sec-loop-limit}.
\item $M_j^{n,\op{in}}$; Section~\ref{sec-loop-limit}.
\item $\sigma_j$; Section~\ref{sec-loop-limit}.
\item $T_j$; Section~\ref{sec-loop-limit}.
\item $\Sigma_j$; Section~\ref{sec-loop-limit}.
\item $o_x^\infty(x)$; Notation~\ref{def-o-notation}. 
\item $(Q,\BB e_0)$; Section~\ref{sec-burger-bijection}. 
\item $\lambda$; Section~\ref{sec-burger-bijection}.
\item $\mcl I = \mcl I^L \cup \mcl I^R$; Section~\ref{sec-detecting-inner-discrete}.
\item $M_j^0$; Section~\ref{sec-fk-loops}.
\item $\iota_j$; Section~\ref{sec-fk-loops}. 
\item $\wt\theta_j$ and $\theta_j$; Section~\ref{sec-fk-loops}.
\item $I_j$; Section~\ref{sec-fk-loops}.
\item $\Theta_j$; Section~\ref{sec-fk-loops}.
\item $\mu$;~\eqref{eqn-cone-exponent}.
\item $\mu'$;~\eqref{eqn-cone-exponent}.
\item $J$; Section~\ref{sec-F-reg-setup}.
\item $E_n(\ep)$; Section~\ref{sec-F-reg-setup}.
\item $F_n$; Section~\ref{sec-F-reg-setup}.
\item $a_n(\ep)$; Section~\ref{sec-F-reg-setup}.  
\item $Z^n_{[t_1,t_2]}$ and $Z_{[t_1,t_2]}$; Definition~\ref{def-Z-restrict}. 
\end{itemize}
\end{multicols}

\bibliography{cibiblong,cibib} 
\bibliographystyle{hmralphaabbrv}

\end{document}